\documentclass[11pt]{article}
\usepackage{amsmath,amssymb,amsthm,bm,graphicx,color,epsfig,cases,enumerate,caption,subcaption}

\usepackage{algpseudocode}
\usepackage{algorithm}
\usepackage{algorithmicx}

\newtheorem{theorem}{Theorem}[section]
\newtheorem{lemma}[theorem]{Lemma}

\newtheorem{proposition}[theorem]{Proposition}
\newtheorem{definition}[theorem]{Definition}

\newtheorem{algo}[theorem]{Algorithm}

\newcommand{\R}{\mathbb{R}}
\newcommand{\Z}{\mathbb{Z}}

\newcommand{\eps}{\varepsilon}

\newcommand{\grad}{\nabla}

\newcommand{\SNR}{{\it SNR}}

\newcommand{\VAR}{{\it Var}}

\begin{document}

\title{Synchrosqueezed Curvelet Transform\\
 for 2D Mode Decomposition}
\author{Haizhao Yang$^{\dagger}$ and Lexing Ying$^{\sharp}$\\
  \vspace{0.1in}\\
  $\dagger$ Department of Mathematics, Stanford University\\
  $\sharp$ Department of Mathematics and ICME, Stanford University
}

\date{October 2013}
\maketitle

\begin{abstract}
This paper introduces the synchrosqueezed curvelet transform as an optimal tool for 2D mode decomposition of wavefronts or banded wave-like components. The synchrosqueezed curvelet transform consists of a generalized curvelet transform with application dependent geometric scaling parameters, and a synchrosqueezing technique for a sharpened phase space representation. In the case of a superposition of banded wave-like components with well-separated wave-vectors, it is proved that the synchrosqueezed curvelet transform is capable of recognizing each component and precisely estimating local wave-vectors. A discrete analogue of the continuous transform and several clustering models for decomposition are proposed in detail. Some numerical examples with synthetic and real data are provided to demonstrate the above properties of the proposed transform.

\end{abstract}



%----------------------------------------------------------
\section{Introduction}
\label{sec:intro}

%Problem statement
In various applications  (e.g., medicine \cite{medicine,medicine2} and engineering \cite{Eng1,Eng2}), one is faced with a signal which is a superposition of several components (perhaps nonlinear and non-stationary). The frequency or wave-vector of each component is localized in the time-frequency or phase space representation. A natural question would be whether it is possible to set them apart according to their localized representation and estimate their local frequencies or wave-vectors. Classical time-frequency or phase space analysis provides several powerful tools for representing and analyzing complex signals. All of these tools essentially fall into two categories: linear or quadratic. As discussed in \cite{Daubechies2011}, linear methods have simple and efficient algorithms for forward and inverse transforms, but the resolution is unavoidably limited by the Heisenberg uncertainty principle. Although quadratic methods provide high resolution, the corresponding reconstruction methods are less straightforward and significantly more costly. Furthermore, non-physical interference between components is more pronounced.

By introducing the synchrosqueezing technique, Daubechies et al proposed the synchrosqueezed wavelet transform in \cite{Daubechies1996} and demonstrated that, an important class of signals under the assumption of well-separated frequencies, could be precisely decomposed. Synchrosqueezing, the key idea, is a reallocation method \cite{Auger1995,Chassande-Mottin2003,Chassande-Mottin1997,Daubechies1996} aiming at a sharpened time-frequency representation by reassigning values of the original representation. Though it has been shown to provide good results for 1D signals, even with a substantial amount of noise, in higher dimensional space the application of the synchrosqueezed wavelet transform is limited. It cannot distinguish two components sharing the same wave-number but having different wave-vectors, because of the isotropic character of the high dimensional wavelet transform. In fact, this is a common phenomenon in many applications of high frequency wave propagation. To specify this problem, let us consider a simple superposition of two plane waves $e^{2\pi ip\cdot x}$ and $e^{2\pi iq\cdot x}$ with the same wave-number ($|p|=|q|$) but different wave-vectors ($p \not= q$). In the Fourier domain, the gray region in Figure \ref{fig:freqpart} (left) shows the support of one continuous wavelet. The wavelet cannot distinguish these two plane waves in the sense that the gray region has to cover two dots $p$ and $q$ simultaneously, or has to exclude them simultaneously.

To overcome this inherent limitation of the synchrosqueezed wavelet transform in high dimensional space, the synchrosqueezed wave packet transform (SSWPT) was developed in \cite{SSWPT}, inspired by the localized support of wave packets in the Fourier domain. The finer supports result in better resolution for wave-number separation and, more importantly, the anisotropic supports contribute to the angular separation of wave-vectors. As shown in Figure \ref{fig:freqpart} (middle), in the Fourier domain, the supports of $e^{2\pi ip\cdot x}$ and $e^{2\pi iq\cdot x}$ are in the supports of two different wave packets, as long as $p$ and $q$ are well-separated. \cite{SSWPT} proved that SSWPT could identify different nonlinear and non-stationary high frequency wave-like components with different wave-vectors in high dimensional space in a general case, even with severe noise. It has also been shown that SSWPT can capture the edges of incomplete components, so that it could identify the discontinuity of wave propagation and extract connected continuous components. 
%Another method on 2D synchrosqueezing transforms was proposed later in \cite{2Dwavelet} by %introducing anisotropic 2D monogenic wavelet transform. However, the analysis is based on 1D %spectral line of wave-vectors (i.e., the amplitude of wave-vector). The resulting monogenic %synchrosqueezed wavelet transform is still isotropic by lack of angular separation of wave-vectors.

\begin{figure}[ht!]
  \begin{center}
    \begin{tabular}{ccc}
      \includegraphics[height=1.6in]{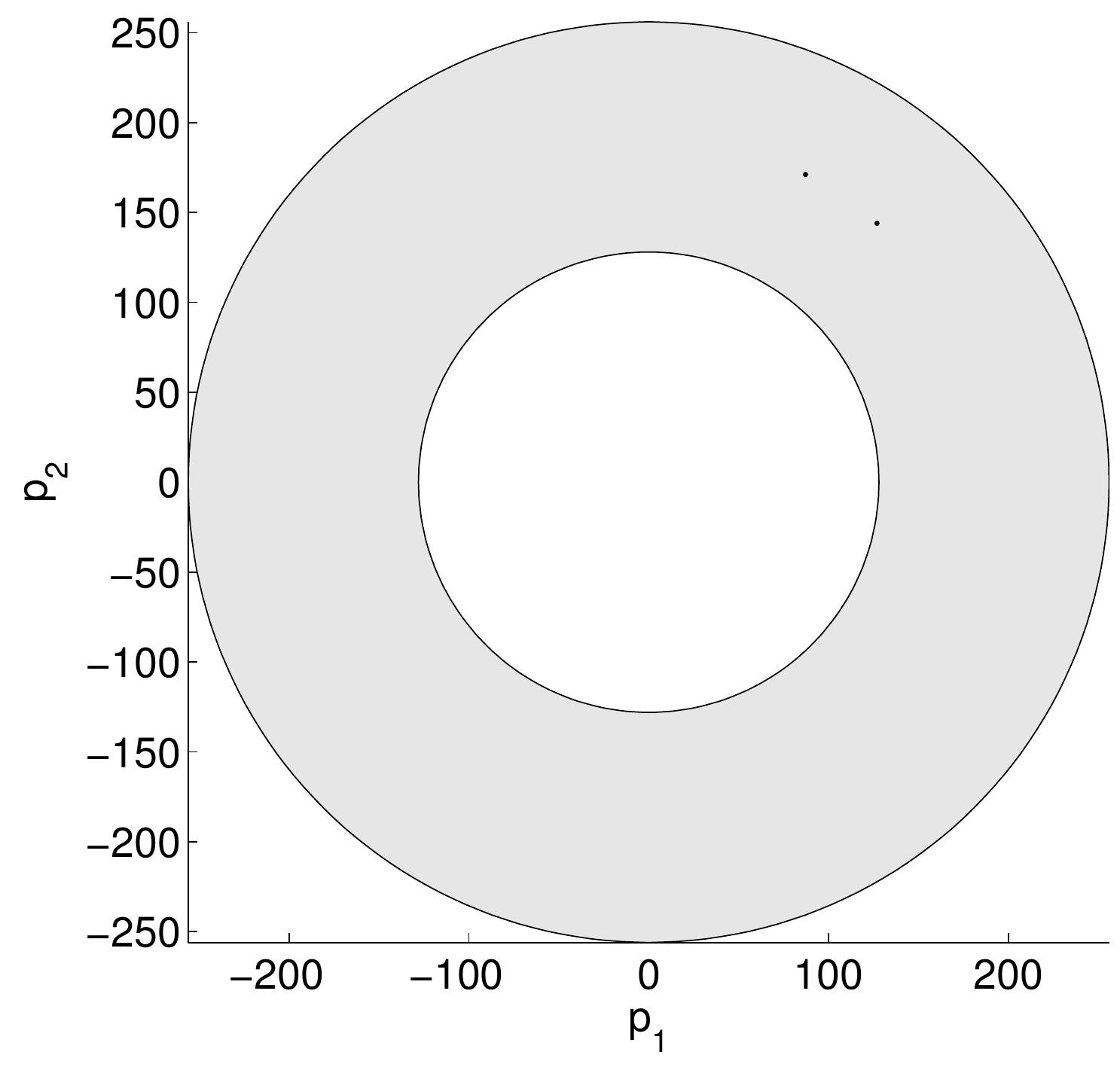} &  \includegraphics[height=1.6in]{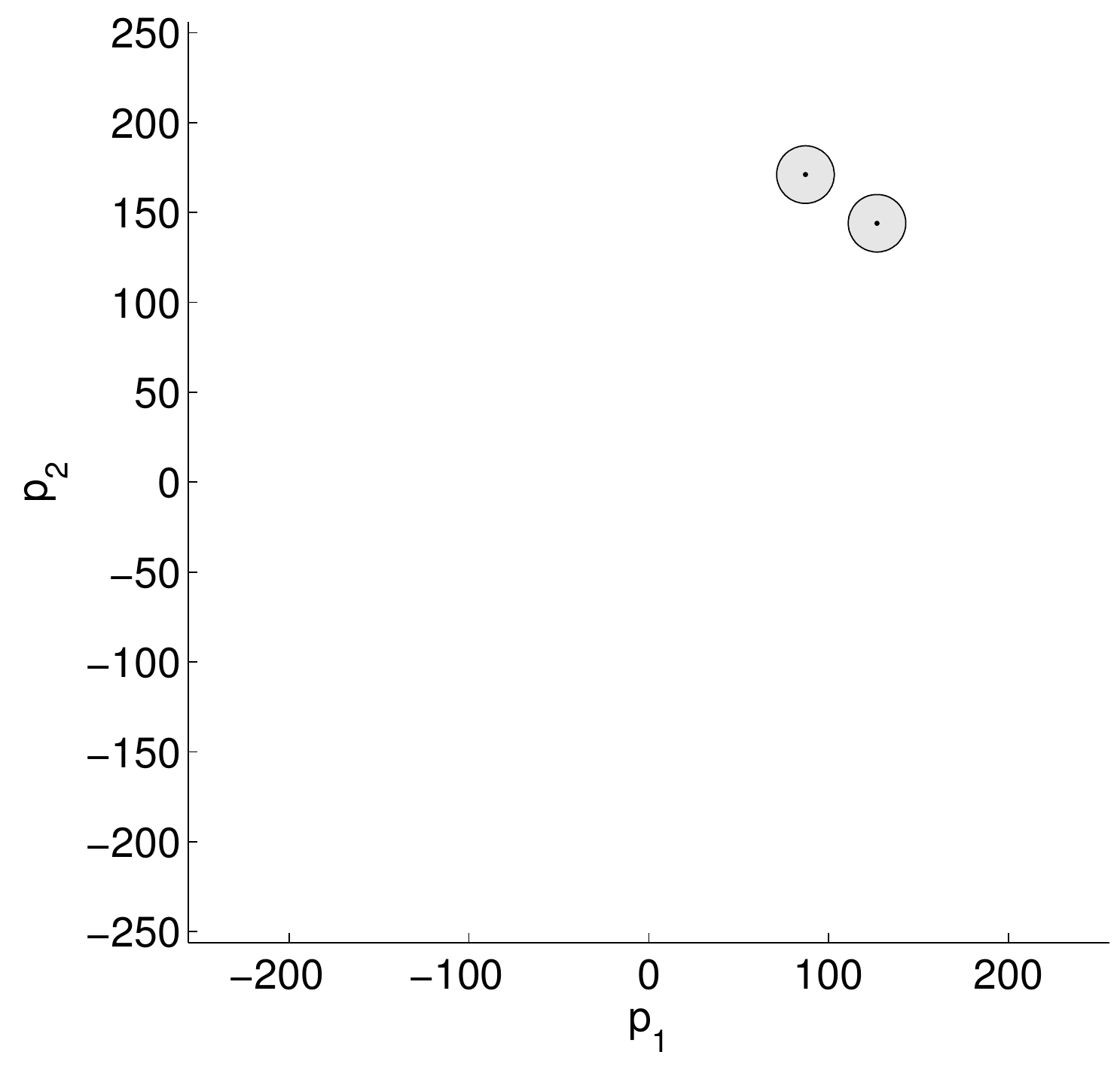} & \includegraphics[height=1.6in]{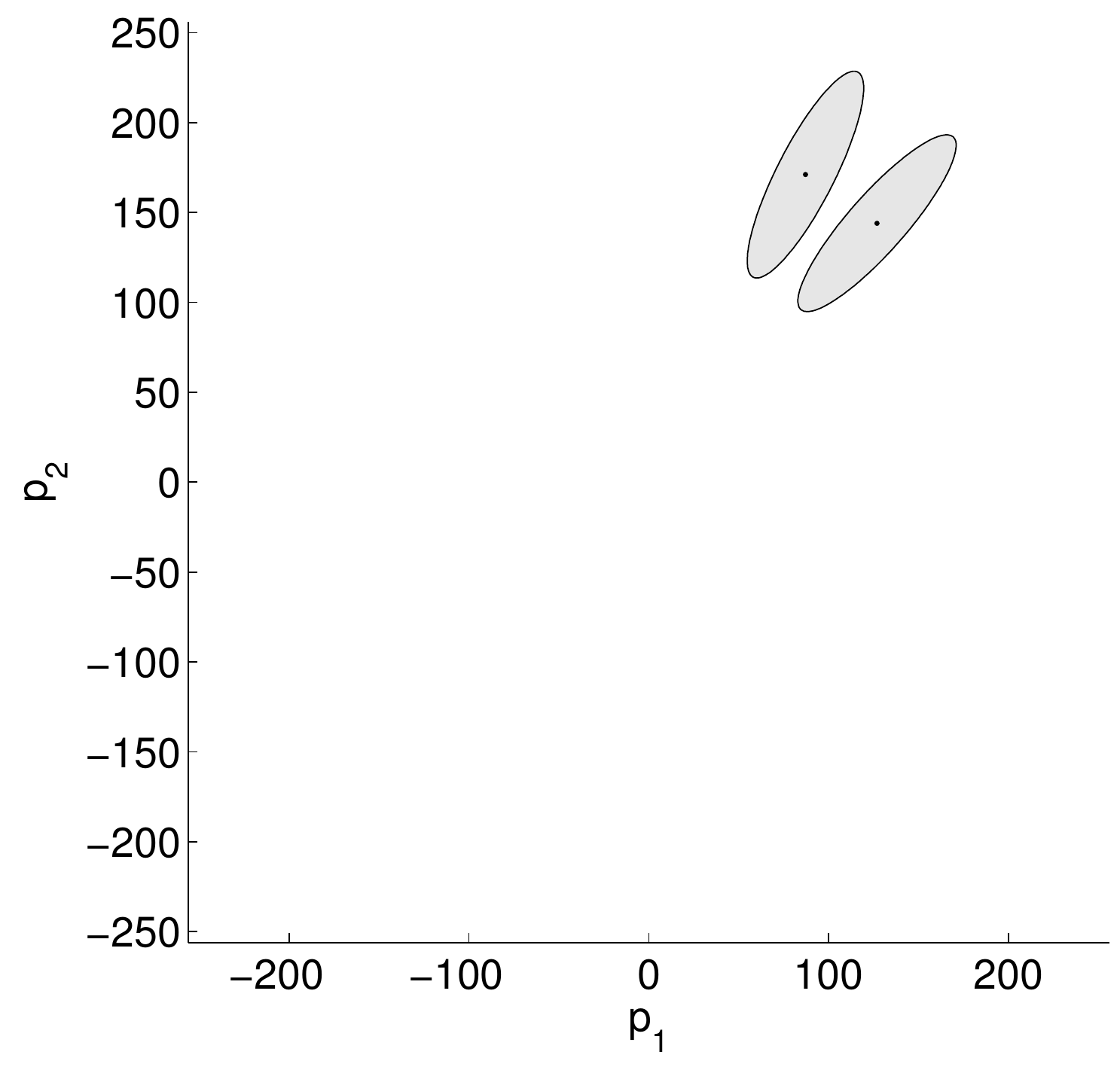}
    \end{tabular}
  \end{center}
  \caption{Comparison of localized supports of continuous wavelets
    (left), wave packets (middle) and curvelets (right) in the Fourier
    domain. Two dots in each plot show the support of the Fourier transforms of the
    superposition of two plane waves $e^{2\pi ip\cdot x}$ and $e^{2\pi iq\cdot x}$ with the same     wave-number ($|p|=|q|$) but different wave-vectors ($p \not= q$).}
  \label{fig:freqpart}
\end{figure}

When one separates overlapping wavefronts or banded wave-like components, the boundary of these components gives rise to many nonzero coefficients of wave packet transform, which results in unexpected interferential synchrosqueezed energy distribution (see Figure \ref{fig:band_wp} middle). This would dramatically reduce the accuracy of local wave-vector estimation, because the locations of nonzero energy provide estimation of local wave-vectors. As shown in Figure \ref{fig:band_wp} (right), there exists misleading local wave-vector estimates at the location where the signal is negligible. Even if at the location where the signal is relevant, the relative error is still unacceptable. To solve this problem, an empirical idea is that, good basis elements in the synchrosqueezed transform should look like the components, i.e., they should appear in a needle-like shape. An optimal solution is curvelets. The curvelet transform is anisotropic (as shown in Figure \ref{fig:freqpart} right), and is designed for optimally representing curved edges \cite{Candes2002,Candes2004} and banded wavefronts \cite{Candes2006}. This motivates the design of the synchrosqueezed curvelet transform (SSCT) as an optimal tool to estimate local wave-vectors of wavefronts or banded wave-like components in this paper. The estimate of local wave-vectors provided by SSCT is much better than that by SSWPT as shown in Figure \ref{fig:band_wp}.

\begin{figure}[ht!]
  \begin{center}
    \begin{tabular}{cc}
      \includegraphics[height=2.4in]{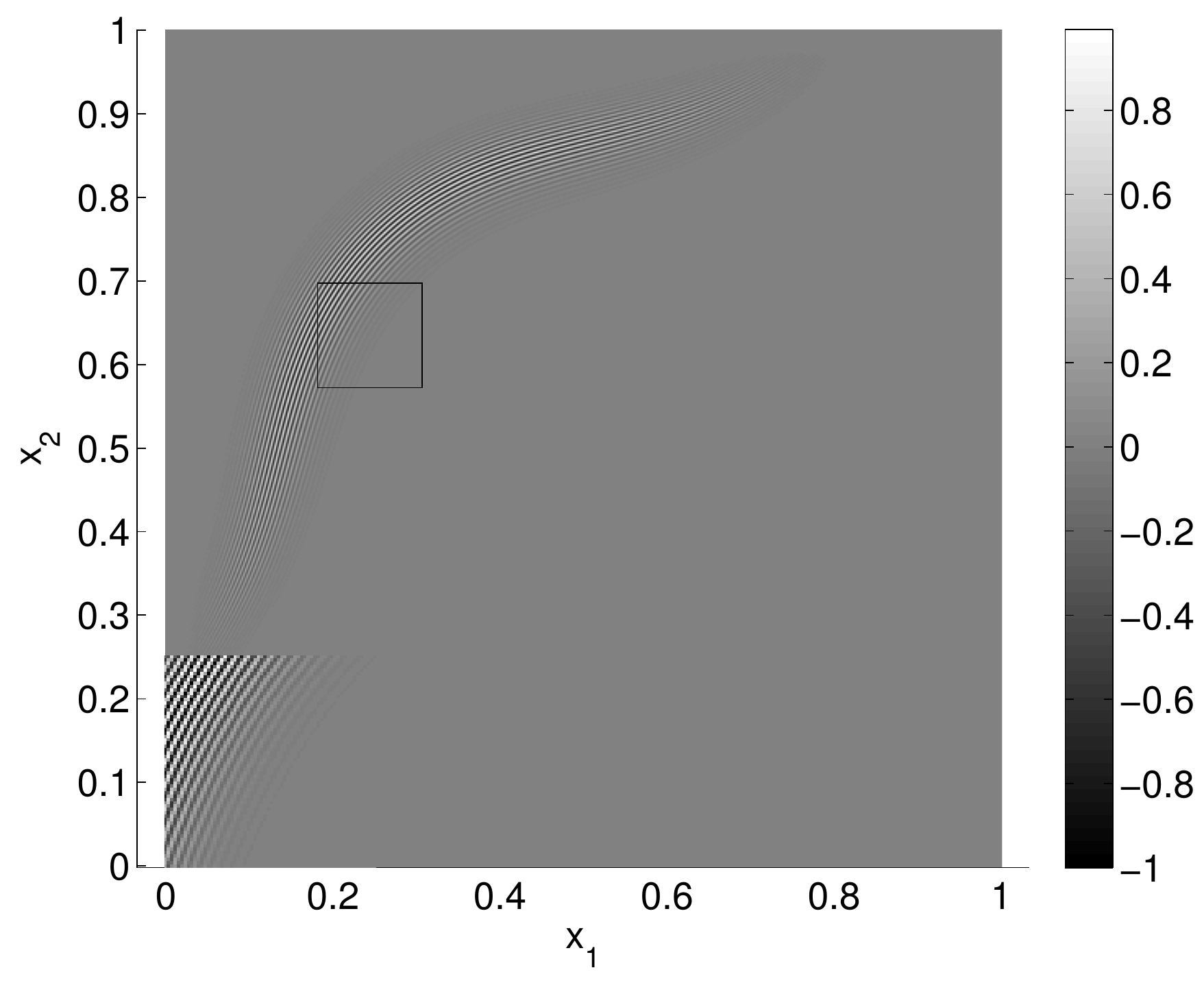} &  \includegraphics[height=2.4in]{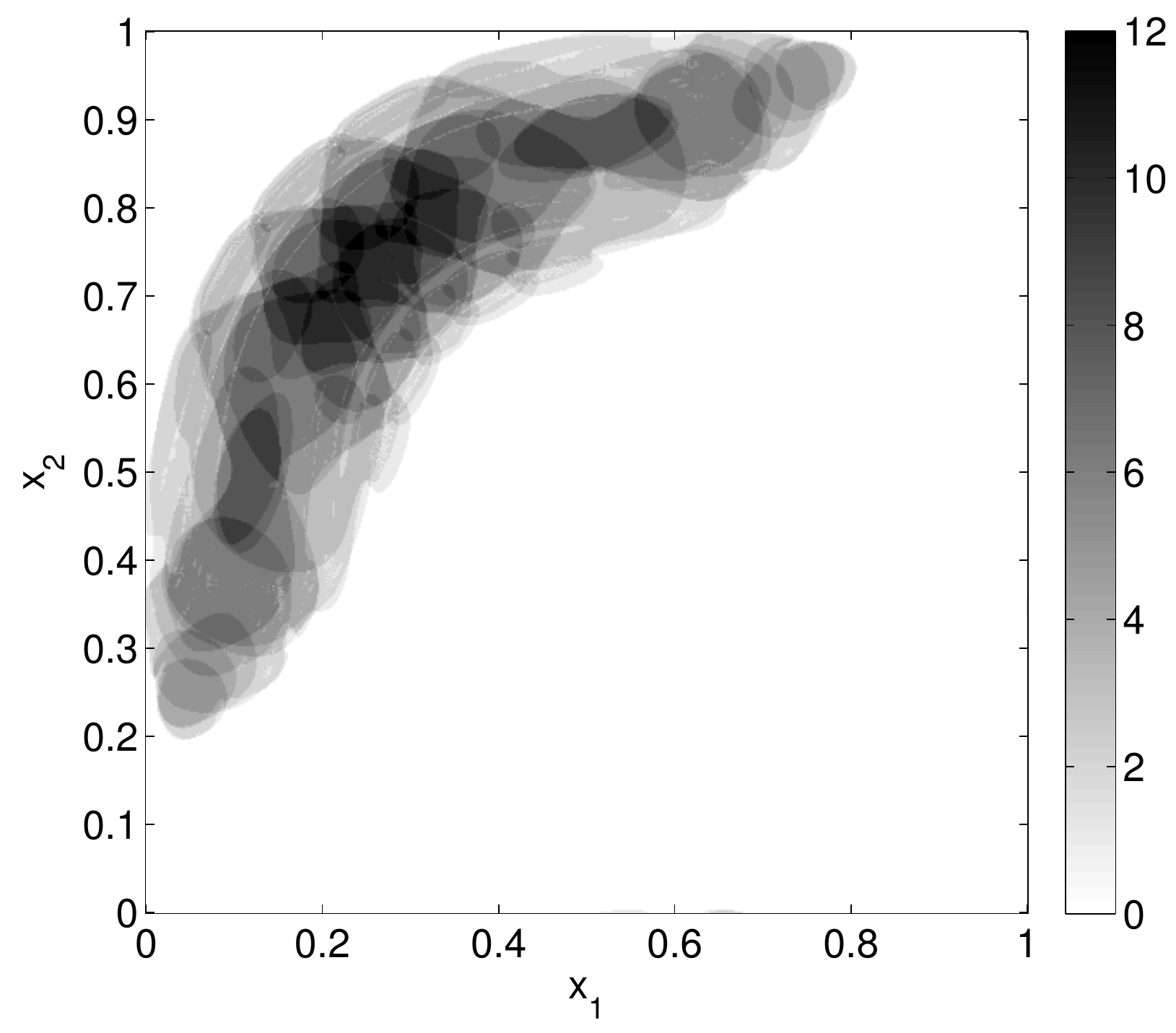}\\
      \includegraphics[height=2.4in]{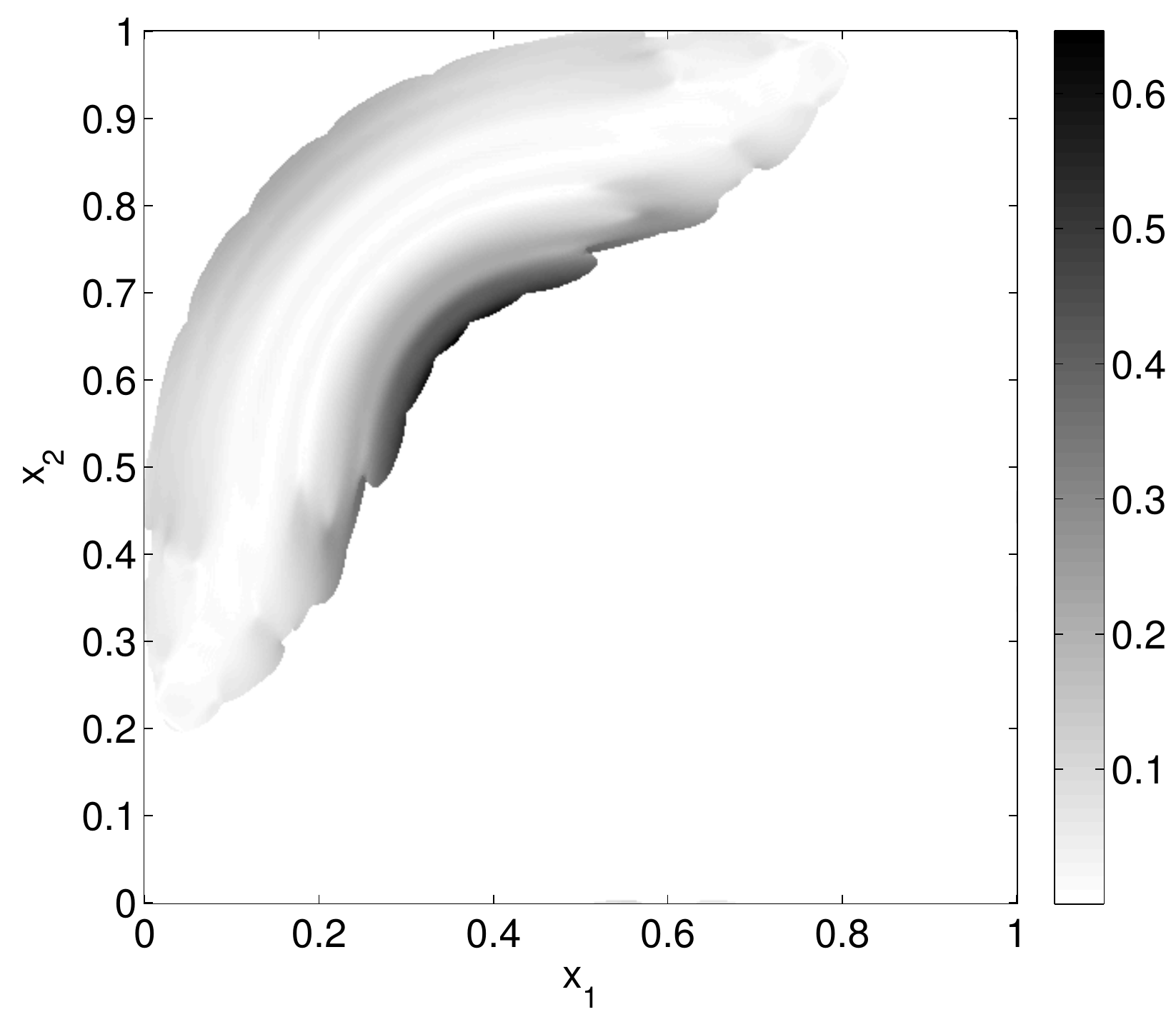} & \includegraphics[height=2.4in]{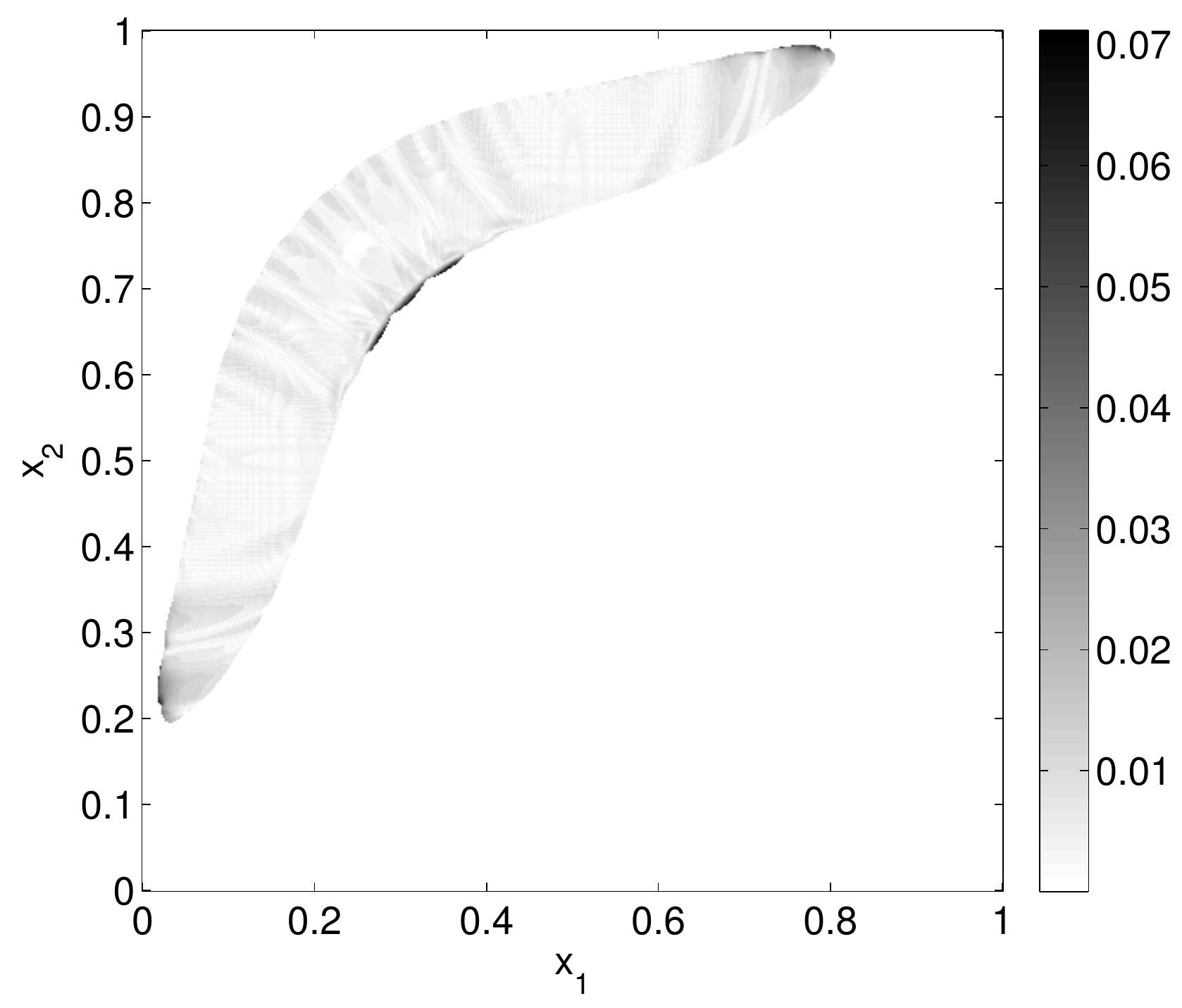}
    \end{tabular}
  \end{center}
  \caption{Top-left: A banded deformed plane wave, $f(x)=e^{-\frac{(\phi(x)-0.7)^2}{\sigma^2}}e^{2\pi i N\phi(x)}$, where $\sigma=\frac{4}{135}$, $N=135$ and $\phi(x)=x_1+(1-x_2)+0.1\sin(2\pi x_1)+0.1\sin(2\pi (1-x_2))$. Top-right: Number of nonzero discrete synchrosqueezed energy of SSWPT at each grid point of space domain. Bottom-left: Relative error between the mean local wave-vector estimate (defined in \cite{SSWPT}) and the exact local wave-vector using SSWPT. Bottom-right: Relative error between the mean local wave-vector estimate and the exact local wave-vector using SSCT.}
  \label{fig:band_wp}
\end{figure}

%------------------------------------------------------
\subsection{Synchrosqueezed curvelet transform (SSCT)}
\label{sub:SSCT}

Following is a brief introduction to the general curvelet transform with a radial scaling parameter $t<1$ and an angular scaling parameter $s\in(\frac{1}{2},t)$. Similar to the discussion in \cite{SSWPT}, it is crucial to assume $\frac{1}{2}<s<t<1$, so as to obtain accurate estimates of local wave-vectors for reasonable large wavenumbers. It is proved in the next section, $s<t$ guarantees precise estimates in the case of banded wave-like components. Here are some notations for the general curvelet transform. 
\begin{enumerate}
\item The scaling matrix
\[A_a=\left( \begin{array}{cc}
a^t&0\\
0&a^s 
\end{array}\right),\]
where $a$ is the distance from the center of one curvelet to the origin of Fourier domain.
\item The rotation angle $\theta$ and rotation matrix 
\[R_\theta=\left( \begin{array}{cc}
\cos\theta & -\sin\theta\\
\sin\theta& \cos\theta\end{array}\right).\]
\item The unit vector $e_\theta=(\cos \theta,\sin \theta)^T$ of rotation angle $\theta$.
\item $\theta_\alpha$ represents the argument of given vector $\alpha$.
\item $w(x)$ of $x\in \R^2$ denotes the mother curvelet, which is in the Schwartz class and has a non-negative, radial, real-valued, smooth Fourier transform
$\widehat{w}(\xi)$ with support equal to the unit ball $B_1(0)$ in the Fourier
domain. The mother curvelet is required to obey the admissibility condition: $\exists 0<c_1<c_2<\infty$ such that
\[
c_1\leq \int_0^{2\pi} \int_1^\infty a^{-(t+s)} |\widehat{w}(A^{-1}_a R^{-1}_\theta (\xi-a\cdot e_\theta))|^2 adad\theta \leq c_2
\]
for any $|\xi|\geq 1$.
\end{enumerate}
With the notations above, it is ready to define a family of curvelets through scaling, modulation, and translation as follows, controlled by the geometric parameter $s$ and $t$.
\begin{definition}
  \label{def:GC2D}
For $\frac{1}{2}<s<t<1$, define $\widehat{w_{a\theta b}}(\xi)=\widehat{w}(A^{-1}_a R^{-1}_\theta (\xi-a\cdot e_\theta))e^{-2\pi ib\cdot\xi}a^{-\frac{t+s}{2}}$ as a general curvelet in the Fourier domain. Equivalently, in the space domain, the corresponding general curvelet is
\begin{eqnarray*}
w_{a\theta b}(x)&=&\int_{R^2}\widehat{w}(A^{-1}_a R^{-1}_\theta (\xi-a\cdot e_\theta)) e^{-2\pi ib\cdot \xi}e^{2\pi i\xi\cdot x}a^{-\frac{t+s}{2}}d\xi\nonumber\\
&=&a^{\frac{t+s}{2}}\int_{R^2}\widehat{w}(y)e^{-2\pi ib\cdot(R_\theta A_a y+a\cdot e_\theta)}e^{2\pi ix\cdot(R_\theta A_a y+a\cdot e_\theta)}dy\nonumber\\
&=&a^{\frac{t+s}{2}}e^{2\pi ia(x-b)\cdot e_\theta}w(A_a R^{-1}_\theta(x-b)).
\end{eqnarray*}
In such a way, a family of curvelets $\{w_{a\theta b}(x),a\in [1,\infty),\theta\in[0,2\pi),b\in \R^2\}$ is constructed. 
\end{definition}
By definition, the Fourier transform $\widehat{w_{a\theta b}}(\xi)$ is supported in an ellipse $\{x:|A_a R^{-1}_\theta(x-b)|\leq 1\}$ centered at $a\cdot e_\theta$ with a major radius $a^t$ and a minor radius $a^s$. It is natural to require $a\geq 1$ in order to keep the consideration regarding the shape of curvelets valid. Meanwhile, $w_{a\theta b}(x)$ is centered in space at $b$ with an essential support of length $O(a^{-s})$ and width $O(a^{-t})$. By this appropriate construction, each curvelet is scaled to have the same $L^2$ norm with the mother curvelet $w(x)$. Notice that if $s=\frac{1}{2}$ and $t=1$, these functions would be
qualitatively similar to standard 2D curvelets. When $s=t$, these functions would become general wave packets in \cite{SSWPT}. As $s=t$ approaching $1$ or $\frac{1}{2}$, they are getting close to wavelets or wave atoms \cite{Demanet2007}, respectively.

Similar to the classical curvelet transform, the general curvelet transform is defined to be the inner product of a given signal and each curvelet as follows.
\begin{definition}
  \label{def:GCT}
The general curvelet transform of a function $f(x)$ is a function 
\begin{eqnarray*}
W_f(a,\theta,b)&=&\langle w_{a\theta b},f\rangle = \int_{\R^2}\overline{w_{a\theta b}(x)}f(x)dx\nonumber\\
&=&\langle\widehat{w_{a\theta b}},\widehat{f}\rangle=\int_{\R^2}\overline{\widehat{w_{a\theta b}(\xi)}}\widehat{f}(\xi)d\xi\nonumber
\end{eqnarray*}
for $a\in [1,\infty)$, $\theta\in[0,2\pi)$, $b\in \R^2$.
\end{definition}
If the Fourier transform $\widehat{f}(\xi)$ vanishes for $|\xi|< 1$, one can check the following $L^2$ norms equivalence up to a uniform constant factor following the proof of Theorem $1$ in \cite{CandesII}, i.e.,
\[
%\int a^{2(t+s)} |W_f(a,\theta,b)|^2 da d\theta db \eqsim \int |f(x)|^2 dx.
c_1 \int |f(x)|^2 dx\leq\int|W_f(a,\theta,b)|^2 ada d\theta db \leq c_2 \int |f(x)|^2 dx.
\]

Below is a simple example to show how the synchrosqueezing technique  estimates local wave-vectors. Let us consider a plane wave function
%Before introducing the synchrosqueezing technique, let us recall its basic idea that reallocating %the coefficients of a certain transform to obtain a concentrated result. Then, a natural question %would be whether we can reassign the coefficients around local wave-vectors so as to estimate the %wave-vectors. Fortunately, the answer is yes and the following simple example will provide direct %insight into the above question.
%
%Given a plane wave function
\[
f(x)= \alpha e^{2\pi i N \beta \cdot x},
\]
where $\alpha$ and $\beta$ are nonzero constants of order $O(1)$ and
$N$ is a sufficiently large constant. The general curvelet transform of $f(x)$ is
\begin{eqnarray*}
  W_f(a,\theta ,b)
  &=&\int_{\R^2} \alpha e^{2\pi i N\beta \cdot x}a^{\frac{s+t}{2}}w(A_a R^{-1}_\theta (x-b))e^{-2\pi ia(x-b)\cdot e_\theta}dx\nonumber\\
  &=&a^{-\frac{s+t}{2}} \alpha \int_{\R^2}e^{2\pi i N\beta \cdot (b+R_\theta A^{-1}_a y)}w(y)e^{-2\pi i a^{1-t}y_1}dy\nonumber\\
  &=&a^{-\frac{s+t}{2}} \alpha e^{2\pi i N\beta\cdot b} \overline{\widehat{w}(A^{-1}_a R^{-1}_\theta(a\cdot e_\theta-N\beta)))}.\nonumber
\end{eqnarray*}
Notice that $\widehat{w}(\xi)$ is compactly supported in the unit ball, $W_f(a,\theta,b)$ is able to provide a preliminary estimate of the local wave-vector $N\beta$, since the nonzero $W_f(a,\theta,b)$ is located in the regime
\[
|A^{-1}_a R^{-1}_\theta(a\cdot e_\theta-N\beta)|\leq 1.
\]
This implies that, for each $b$, $W_f(a,\theta,b)$ has a support of length $O(|N\beta|^t)$ and width $O(|N\beta|^s)$ around the wave-vector $N\beta$ in the variable $a$ and $\theta$. Nevertheless, the resolution of this estimate is too low. Further observation tells us that the oscillation of
$W_f(a,\theta,b)$ in the $b$ variable in fact uncovers $N\beta$ by
\begin{eqnarray*}
 \grad_b W_f(a,\theta,b)
&=& 2\pi i N\beta a^{-\frac{s+t}{2}} \alpha e^{2\pi i N\beta\cdot b} \overline{\widehat{w}(A^{-1}_a R^{-1}_\theta(a\cdot e_\theta-N\beta))}\nonumber\\
&=& \big(2\pi iW_f(a,\theta,b)\big)N\beta.\nonumber
\end{eqnarray*}
This motivates the definition of the local wave-vector estimation for a general
function $f(x)$ as follows.
\begin{definition} 
  \label{def:LW}
  The local wave-vector estimation of a function $f(x)$ at
  $(a,\theta,b)$ is
  \begin{equation}
    v_f(a,\theta,b)=\frac{ \grad_b W_f(a,\theta,b) }{ 2\pi i W_f(a,\theta,b)} \label{eq:IWE}
  \end{equation}
  for $a\in [1,\infty)$, $\theta\in [0,2\pi)$, $b\in\R^2$ such that $W_f(a,\theta,b)\not=0$.
\end{definition}

It is remarkable that $v_f(a,\theta,b)$ estimates the local wave-vectors independently of the amplitude $\alpha$ or the position $b$. Hence, if the coefficients with the same $v_f$ are reallocated together, then there would be a sharpened phase space representation of $f(x)$, a clear picture of nonzero energy concentrating around local wave-vectors. Mathematically speaking, the synchrosqueezed energy distribution is defined as follows.
\begin{definition}
  Given $f(x)$, $W_f(a,\theta,b)$, and $v_f(a,\theta,b)$, the synchrosqueezed energy
  distribution $T_f(v, b)$ is 
  \begin{equation}
    T_f(v,b) = \int |W_f(a,\theta,b)|^2 \delta(\Re v_f(a,\theta,b)-v) adad\theta \label{eq:SED}
  \end{equation}
  for $v\in \R^2$, $b\in \R^2$.
\end{definition} 
For $f(x)$ with Fourier transform vanishing for
$|\xi|< 1$, the following norm equivalence holds
\[
\int T_f(v,b) dvdb =\int |W_f(a,\theta,b)|^2 adad\theta db \eqsim \|f\|_2^2
\]
as a consequence of the $L^2$ norm equivalence between $W_f(a,\theta,b)$ and
$f(x)$.

Equipped with the definitions above, let us consider now a general function of the form
\[
f(x) = e^{-(\phi(x)-c)^2/\sigma^2}\alpha(x) e^{2\pi i  N\phi(x)}
\]
with a smooth amplitude $\alpha(x)$, a smooth phase $\phi(x)$, a banded parameter $\sigma= \Theta(N^{-\eta})$ ($\eta<t$) and a
sufficiently large $N$. It will be shown that the general curvelet
transform $W_f(a,\theta,b)$ for each $b$ is essentially supported in the following
set
\begin{equation}
\label{support}
\{(a,\theta): |A^{-1}_a R^{-1}_\theta(a\cdot e_\theta-N\grad\phi(b))| \leq 1\}.
\end{equation}
In the meantime, $v_f(a,\theta,b)$ is an accurate estimation of the local wave-vector $N\grad\phi$ independent of $a$ and $\theta$, which implies that the essential support of the synchrosqueezed energy distribution $T_f(v,b)$ in $v$ is concentrating around $N\grad\phi$ at each location.

\subsection{Mode decomposition}
\label{sub:mode}
In the previous subsection, the property of the synchrosqueezed curvelet transform that it concentrates the energy of a banded wave-like component around its wave-vectors has been informally discussed. In what follows, the procedure of the mode decomposition after synchrosqueezing will be presented. For simplicity, let
\[
f(x)=e^{-(\phi_1(x)-c_1)^2/\sigma^2_1}\alpha_1(x) e^{2\pi i N \phi_1(x)} + e^{-(\phi_2(x)-c_2)^2/\sigma^2_2}\alpha_2(x) e^{2\pi i N \phi_2(x)},
\]
with smooth amplitudes $\alpha_1(x)$ and $\alpha_2(x)$, banded parameters $\sigma_1$ and $\sigma_2$ of order $\Theta(N^{-\eta})$ ($\eta<t$), smooth phases
$N\phi_1(x)$ and $N\phi_2(x)$ for a sufficiently large $N$. Let us assume that at each position the
local wave-vectors $N\grad\phi_1(x)$ and $N\grad\phi_2(x)$ are
sufficiently large and well-separated from each other. 

The decomposition relies on four steps summarized below.
\begin{enumerate}
\item By \eqref{support}, the essential supports of $W_{f_1}(a,\theta,b)$ and $W_{f_2}(a,\theta,b)$ are contained in the following sets
\begin{eqnarray*}
P_1 = \{(a,\theta,b): |A^{-1}_a R^{-1}_\theta(a\cdot e_\theta-N\grad\phi_1(b))| \leq 1\}, \nonumber\\
P_2 = \{(a,\theta,b): |A^{-1}_a R^{-1}_\theta(a\cdot e_\theta-N\grad\phi_2(b))| \leq 1 \}.\nonumber
\end{eqnarray*}
Because both $|N\grad\phi_1(b)|$ and $|N\grad\phi_2(b)|$ are large, and $N\grad\phi_1(x)$ and $N\grad\phi_2(b)$ are sufficiently
well-separated, these two sets are essentially disjoint. Hence, the essential support of $W_{f}(a,\theta,b)$ is separated into two essentially disjoint sets, each of which corresponds to one component in $f(x)$.
\item The separation in Step $1$ implies that for each $b$
\[
v_f(a,\theta,b)=v_{f_1}(a,\theta,b)\approx N\grad\phi_1 \text{ in }P_1,
\]
and 
\[
v_f(a,\theta,b)=v_{f_2}(a,\theta,b)\approx N\grad\phi_2\text{ in }P_2.
\]
Though $v_f(a,\theta,b)$ is defined wherever $W_f(a,\theta,b)\neq 0$, it is only relevant when $|W_f(a,\theta,b)|$ is above a significant level, as it will be shown in Theorem \ref{thm:main} \eqref{thm:R}. Hence, it is sufficient to compute $v_f(a,\theta,b)$ in these disjoint essential supports $P_1$ and $P_2$ to estimate local wave-vectors of each component.
\item The separation in Step $2$ shows that $T_f(v,b)$ is essentially concentrating around two well-separated 2D surfaces \[S_1 = \{ (N|\grad
\phi_1(b)|,\theta_{\grad\phi_1(b)}, b): b\in \R^2 \}\] and \[S_2 = \{ (N|\grad \phi_2(b)|,\theta_{\grad\phi_2(b)}, b): b\in \R^2 \}.\]
Hence, the essential support of $T_f(v,b)$ separates into two well disjoint sets $U_1$ and $U_2$.
\item Notice that $T_f(v,b) = T_{f_1}(v,b)$ in $U_1$ and, respectively, $T_f(v,b) = T_{f_2}(v,b)$ in $U_2$. Once $U_1$ and $U_2$ are identified by some clustering technique, each component of $f(x)$ can be recovered by
\begin{align}
  & f_1(x) = \int_{\Re v_f(a,\theta,b) \in U_1} \tilde{w}_{a\theta b}(x) W_f(a,\theta,b) dad\theta db, \label{eq:REC}\nonumber\\
  & f_2(x) = \int_{\Re v_f(a,\theta,b) \in U_2} \tilde{w}_{a\theta b}(x) W_f(a,\theta,b) dad\theta db, \nonumber
\end{align}
where the set of functions $\{ \tilde{w}_{a\theta b}(x), a\in [1,\infty),\theta\in [0,2\pi), b\in \R^2 \}$ is
the dual frame of $\{w_{a\theta b}(x),a\in [1,\infty),\theta\in [0,2\pi), b\in \R^2\}$. 
\end{enumerate}
The synchrosqueezing step 2 and 3 are indispensable, because they improve the resolution of original results significantly so that clustering is possible for decomposition. In step 4, the reconstruction is based on the Calderon-type reconstruction formula
%Although \cite{Daubechies2011} provides a nice explicit reconstruction formula by integrating the %synchrosqueezed results over scale variable and proves an error bound for such formula, we would %not expect further beyond this Calderon-type reconstruction formula, 
for the reason that curvelet transforms, unlike wavelet transforms in \cite{Daubechies2011}, do not have a reconstruction formula that integrates their coefficients over the scale parameter with a proper weight. In effect, numerical examples in \cite{Daubechies2011} are based on the Calderon-type reconstruction formula, since it works more robustly in noisy cases.

%------------------------------------------------------
\subsection{Related work}
\label{sub:RelatedWork}
There is another interesting line of work for mode decomposition, which is the empirical mode decomposition (EMD) initiated and refined by Huang et al in \cite{Huang1998,Huang2009}. Starting from the most oscillatory mode, the EMD method decomposes a signal into a collection of intrinsic mode functions (IMFs) and estimates instantaneous frequencies via the Hilbert transform. However, the dependence on local extrema limits its applications in noisy cases. To address the robustness problem, some variants were proposed in \cite{Hou2009,Wu2009}. Following the idea of EMD, there are two existing methods for high dimensional mode decomposition. The first one is based on high dimensional interpolation \cite{B1,B2,B3,B4} and the second one applies a 1D decomposition to each dimension and then combines the results with a proper combination strategy \cite{E1,E2,MEEMD}. In spite of their considerable success, these existing methods in this research line are not suitable to separate two modes with similar wave-numbers but different wave-vectors due to the lack of anisotropic angular separation as discussed in \cite{SSWPT}.

Following the same methodology of extracting modes one by one from the most oscillatory one, Hou et al proposed an optimization scheme for mode decomposition in \cite{Hou2011,Hou2012}. Inspired by recent developments of compressive sensing, the first paper \cite{Hou2011} is based on total variations, while the second one \cite{Hou2012} is based on the sparse representation in a data-driven time-frequency dictionary. The convergence of the data-driven time-frequency analysis method under a certain sparsity assumption is proved recently in \cite{Hou2013}. However, the analysis of high dimensional case is still under active research.

There is another research line of adaptive time-frequency representations, the empirical transforms proposed in \cite{EWT} and generalized to 2D in \cite{2DEWT}. The 2D methods in \cite{2DEWT} fall into two kinds. The first one is based on the Fourier spectra of 1D data slices and, hence, lacks the anisotropic angular separation for the same reason of the 2D EMD methods. The second one is based on 2D Pseudo-Polar FFT \cite{Averbuch06fastand,Averbuch_aframework} and suffers the problem of inconsistency, i.e., the results of Fourier boundaries detections in different directions in the 2D Fourier domain are discontinuous. To avoid this problem, the authors compute an average spectrum where the averaging is taken with respect to the angle. The resulting methods are short of the angular separation for the same reason of the synchrosqueezed wavelet transform in \cite{2Dwavelet} as discussed in \cite{SSWPT}.

The rest of the paper is organized as follows. The main theoretical results of SSCT is presented in Section \ref{sec:theorems}. We prove that SSCT is able to estimate the local wave-vectors under some well-separation condition of the local wave-vectors of multiple highly oscillatory components.  In Section \ref{sec:algo}, a discrete analogue of SSCT and some clustering methods in the phase space are introduced. Section \ref{sec:results} compares several numerical examples on local wave-vector estimation using SSWPT and SSCT, and provides decomposition examples with synthetic and real data to demonstrate the proposed properties of SSCT. Finally, this article will end up with some discussions in Section \ref{sec:diss}.

%--------------------------------------------------------
\section{Analysis of the transform}
\label{sec:theorems}

In this section, we define a class of superpositions of multiple banded components with well-separated local wave-vectors and prove that the synchrosqueezed curvelet transform is able to estimate these local wave-vectors accurately. Throughout the analysis, the scaling parameters $s$ and $t$ are fixed such that $\frac{1}{2}<s<t<1$ and $\eta<t$.
\begin{definition}
  \label{def:IMTF}
For any $c\in \R$, $N>0$ and $M>0$, a function $f(x)=e^{-(\phi(x)-c)^2/\sigma^2}\alpha(x)e^{2\pi iN \phi(x)}$ is a banded intrinsic mode function of type $(M,N)$, if $\sigma = \Theta(N^{-\eta})$, $\alpha(x)$ and $\phi(x)$ satisfy
  \begin{align*}
    \alpha(x)\in C^\infty, \quad |\grad \alpha|\leq M, \quad 1/M \leq \alpha\leq M, \\
    \phi(x)\in C^\infty,  \quad  1/M \leq |\grad \phi|\leq M, \quad |\grad^2 \phi|\leq M.\\
   \end{align*}
\end{definition}
If $\eta$ tends to $-\infty$, the banded intrinsic mode function will become the one discussed in \cite{SSWPT}. So, the model in this article is more general.
\begin{definition}
  \label{def:SWSIMC}
  A function $f(x)$ is a well-separated superposition of type
  $(M,N,K)$ if
  \[
  f(x)=\sum_{k=1}^K f_k(x),
  \] 
  where each $f_k(x)=e^{-(\phi_k(x)-c_k)^2/\sigma_k^2}\alpha_k(x)e^{2\pi iN \phi_k(x)}$ is a banded intrinsic mode function of type $(M,N)$ and they satisfy the
  separation condition: $\forall a\in [1,\infty)$ and $\forall \theta\in [0,2\pi)$, there is at most one banded intrinsic mode function $f_k$ satisfying that
  \[
  |A^{-1}_a R^{-1}_\theta (a\cdot e_\theta-N\grad\phi_k(b))|\leq 1.
  \]
We denote by $F(M,N,K)$ the set of all
  such functions.
\end{definition}
%Below is a more concrete sufficient condition for well-separation.
%\begin{proposition}
%\label{pop:WS}
%For any $\epsilon\in(0,1)$ and $f_k(x)$ of type $(M,N,\eta)$, $k=1$, $\dots$, $K$, the function 
%\[
%f(x)=\sum_{k=1}^K f_k(x)
%\]
%is a well-separated superposition of type $(M,N,\eta,K)$, if for any $b$ and $k\neq l$, either the %angle separation or the radius separation defined below holds for sufficiently large $N$ depending %only on $(M,\epsilon)$.
%\begin{itemize}
%\item Angle separation: $|\theta_{\grad\phi_k(b)}-\theta_{\grad\phi_l(b)}|%>2\arcsin((\frac{M}%{N})^{t-s})\approx 2(\frac{M}{N})^{t-s}$.
%\item Radius separation: let 
%\[
%R_{max}=\max\{N|\grad\phi_k(b)|,N|\grad\phi_l(b)|\},
%\]
%and 
%\[
%R_{min}=\min\{N|\grad\phi_k(b)|,N|\grad\phi_l(b)|\},
%\]
%then
%\[
%R_{max}>\frac{R_{min}}{1-\epsilon}+(\frac{R_{min}}{1-\epsilon})^t.
%\]
%\end{itemize}
%\end{proposition}
%\begin{proof}
%The proof of this proposition is given after the following theorem.
%\end{proof}

Recall that $W_f(a,\theta,b)$ is the general curvelet transform of a function $f(x)$ with
geometric scaling parameter $\frac{1}{2}<s<t<1$, and
$v_f(a,\theta,b)$ is the local wave-vector estimation. The following
theorem is the main theoretical result for the synchrosqueezed curvelet transform.
\begin{theorem}
  \label{thm:main}
  For a function $f(x)$, which is a well-separated superposition of some type $(M,N,K)$, and any $\eps>0$, define
  \begin{equation}
  \label{thm:R}
  R_{f,\eps} = \left\{(a,\theta,b): |W_f(a,\theta,b)|\geq a^{-\frac{s+t}{2}}\sqrt \eps\right\}
  \end{equation}
  and 
  \[
  Z_{f,k} = \left\{(a,\theta,b): |A^{-1}_a R^{-1}_\theta(a\cdot e_\theta-N\grad\phi_k(b))|\leq 1 \right\}
  \]
  for $1\le k\le K$. For fixed $M$, $K$, and any $\epsilon$, %$there exists a constant $\eps_0(M,K) > 0$ such that, $\forall \eps \in (0,\eps_0)$, 
there exists $N_0(M,K,\eps)>0$ such that for any $N>N_0(M,K,\eps)$ and $f(x)\in F(M,N,K)$ the following statements
  hold.
  \begin{enumerate}[(i)]
  \item $\left\{Z_{f,k}: 1\le k \le K\right\}$ are disjoint and $R_{f,\eps}
    \subset \bigcup_{1\le k \le K} Z_{f,k}$;
  \item For any $(a,\theta,b) \in R_{f,\eps} \cap Z_{f,k}$, 
    \[
    \frac{|v_f(a,\theta,b)-N\grad \phi_k(b)|}{ |N \grad \phi_k(b)|}\lesssim\sqrt \eps.
    \]
  \end{enumerate}
\end{theorem}

For simplicity, the notations $O(\cdot)$, $\lesssim$ and $\gtrsim$ are used when the implicit constants may only depend on $M$ and $K$. The proof of the
theorem relies on several lemmas. The following one estimates
$W_f(a,\theta,b)$.
\begin{lemma}
  \label{lem:A}
Suppose \[\Omega=\left\{ (a,\theta):a\in\left(\frac{N}{2M},2MN\right),\exists k\ s.t.\ \left|\theta_{\grad\phi_k(b)}-\theta\right|<\theta_0\right\},\] where $\theta_0=\arcsin((\frac{M}{N})^{t-s})$.
  Under the assumption of the theorem, the following estimation of $W_f(a,\theta,b)$ holds for any $\eps$, when $N$ is sufficiently large.
 \begin{enumerate}[(1)]
  \item If $(a,\theta)\in \Omega$,\[W_f(a,\theta,b)=a^{-\frac{s+t}{2}} \left( \sum_{k:\ |\theta_{\grad\phi_k(b)}-\theta|<\theta_0} f_k(b)\widehat{w}\left(A^{-1}_a R^{-1}_\theta (a\cdot e_\theta-N\grad\phi_k(b))\right)+O(\eps) \right);\]
  \item Otherwise,\[W_f(a,\theta,b)=a^{-\frac{s+t}{2}}O(\eps).\]
  \end{enumerate}

%  \begin{equation}
%    \label{E1}
%    W_f(a,\theta,b)=
%    \begin{cases}
%  a^{-\frac{s+t}{2}} \left( \sum_{k\in B_\theta} e^{-\frac{(\phi_k(b)-c_k)^2}{\sigma^2_k}} \alpha_k(b)e^{2\pi iN \phi_k(b)}\widehat{w}\left(A^{-1}_a R^{-1}_\theta (a\cdot e_\theta-N\grad\phi_k(b))\right)+O(\eps) \right),
%&  (a,\theta)\in \Omega  \\
%   a^{-\frac{s+t}{2}}O(\eps),   
%      & (a,\theta)\notin \Omega .
%    \end{cases}
%  \end{equation}
\end{lemma}

\begin{proof}
   We only need to discuss the case when $K=1$. The result for general $K$ is an easy extension by the linearity of general curvelet transform. Suppose $f(x)$ contains a single banded intrinsic mode function of type $(M,N)$
  \[
  f(x)=e^{-(\phi(x)-c)^2/\sigma^2}\alpha(x)e^{2\pi i N \phi(x)}.
  \]
We claim that when $N$ is large enough, the approximation of $W_f(a,\theta,b)$ holds. By the definition of general curvelet transform, it holds that
  \begin{eqnarray*}
W_f(a,\theta,b)&=& \int_{R^2} f(x)a^{\frac{s+t}{2}}w(A_a R^{-1}_\theta (x-b))e^{-2\pi i a(x-b)\cdot e_\theta}dx\\
    &=&a^{-\frac{s+t}{2}}\int_{R^2}f(b+R_\theta A^{-1}_a y)w(y)e^{-2\pi i a^{1-t} y_1}dy.
  \end{eqnarray*}

\textbf{Step 1:} We start with the proof of $(2)$ first. 

Let $h(y)=w(y)e^{-(\phi(b+R_\theta A^{-1}_a y)-c)^2/\sigma^2}\alpha(b+R_\theta A^{-1}_a y)$ and $g(y)=2\pi(N \phi(b+R_\theta A^{-1}_a y)-a^{1-t}y_1)$, then we have
  \[
W_f(a,\theta,b)=a^{-\frac{s+t}{2}}  \int_{\R^2} h(y) e^{i g(y)} dy, 
  \]
with real smooth functions $h(y)$ and $g(y)$. Consider the
  differential operator
  \[
  L =\frac{1}{i}\frac{\langle\grad g,\grad\rangle}{ |\grad g|^2}.
  \]
  If $|\grad g|$ does not vanish, we have
  \[
  L e^{ig} = \frac{\langle \grad g, i\grad g e^{ig} \rangle}{i |\grad g|^2} = e^{i g}.
  \]
By the definition of $w(y)$, we know $h(y)$ is decaying rapidly at infinity. Then we can apply integration by parts to get 
  \[
  \int_{\R^2} h e^{ig} dy = \int_{\R^2} h (L e^{i g}) dy = -\int_{\R^2} \grad\cdot\bigl( \frac{h\grad g}{i|\grad g|^2} \bigr)e^{ig}dy. 
  \]
Hence, we need to estimate $\left| \grad\cdot\bigl( \frac{h\grad g}{i|\grad g|^2} \bigr)\right|$.
Because 
\[
\grad\cdot\bigl( \frac{h\grad g}{i|\grad g|^2} \bigr) = \frac{1}{i}\biggl( \frac{\grad h\cdot\grad g}{|\grad g|^2} + h\grad\cdot\bigl( \frac{\grad g}{|\grad g|^2}\bigr)\biggr)
\]
and $|h(y)|\lesssim 1$, we only need to estimate $\left|\frac{\grad h\cdot\grad g}{|\grad g|^2}\right|$ and $\left| \frac{\partial^2 g}{\partial y_i\partial y_j} \frac{1}{|\grad g|^2}\right|$ for $i,j=1,2$. 

Let $z=(z_1,z_2)^T= R^{-1}_\theta\grad\phi(b+R_\theta A^{-1}_a y)$, $v_1=NA^{-1}_a R^{-1}_\theta\grad\phi(b+R_\theta A^{-1}_a y)$ and $v_2=(a^{1-t},0)^{T}$, then $\grad g(y)=2\pi (v_1-v_2)=2\pi((Nz_1-a)a^{-t},Na^{-s}z_2)$.

\textbf{Case 1:} $a\notin (\frac{N}{2M},2MN)$.

%Let $z=(z_1,z_2)^T= R^{-1}_\theta\grad\phi(b+R_\theta A^{-1}_a y)$, $v_1=NA^{-1}_a %R^{-1}_\theta\grad\phi(b+R_\theta A^{-1}_a y)$ and $v_2=(a^{1-t},0)^{T}$, then $\grad g(y)=2\pi %(v_1-v_2)=2\pi((Nz_1-a)a^{-t},Na^{-s}z_2)$. 
When $a\geq 2MN$, then
\[|\grad g(y)|\geq a^{1-t}-MNa^{-t}=\frac{a^{1-t}}{2}+(\frac{a}{2}-MN)a^{-t}\geq\frac{a^{1-t}}{2}\gtrsim N^{1-t}.\]

When $a\leq\frac{N}{2M}$, then
\[|\grad g(y)|\gtrsim \frac{Na^{-t}}{M}-a^{1-t} \geq \frac{Na^{-t}}{2M} \gtrsim N^{1-t}.\]

So
\begin{equation}
\label{eqn:case1ieq1}
|\grad g(y)|\gtrsim N^{1-t}
\end{equation}
for $a\notin (\frac{N}{2M},2MN)$.

%By the definition of $w(y)$, we know $h(y)$ is decaying rapidly at infinity, we can apply %integration by parts to get 
%  \[
%  \int_{\R^2} h e^{ig} dy = \int_{\R^2} h (L e^{i g}) dy = -\int_{\R^2} \grad\cdot\bigl( %\frac{h\grad g}{i|\grad g|^2} \bigr)e^{ig}dy. 
%  \]
%We are going to estimate $\left| \grad\cdot\bigl( \frac{h\grad g}{i|\grad g|^2} \bigr)\right|$. %Because 
%\[
%\grad\cdot\bigl( \frac{h\grad g}{i|\grad g|^2} \bigr) = \frac{1}{i}\biggl( \frac{\grad h\cdot\grad %g}{|\grad g|^2} + \frac{h \Delta g}{|\grad g|^2} + h\grad\cdot\bigl( \frac{\grad g}{|\grad g|%^2}\bigr)\biggr)
%\]
%and $|h(y)|\lesssim 1$, we only need to estimate $\left|\frac{\grad h\cdot\grad g}{|\grad g|%^2}\right|$ and $\left| \frac{\partial^2 g}{\partial y_i\partial y_j} \frac{1}{|\grad g|^2}\right|%$ for $i,j=1,2$. 
 If $a\geq 2MN$,  then $\left| \frac{\partial^2 g}{\partial y_i\partial y_j} \right|\lesssim Na^{-2s}\lesssim N^{1-2s}$, implying that 
\[
\left| \frac{\partial^2 g}{\partial y_i\partial y_j} \frac{1}{|\grad g|^2}\right|\lesssim N^{1-2s}/N^{2-2t}=\frac{1}{N^{1-2(t-s)}}.
\]
Since $|z|\geq\frac{1}{M}$, then either $|z_1|\geq\frac{1}{\sqrt{2}M}$ or $|z_2|\geq\frac{1}{\sqrt{2}M}$ holds. If $a\leq \frac{N}{2M}$, then
\begin{eqnarray*}
\left| \frac{\partial^2 g}{\partial y_i\partial y_j} \frac{1}{|\grad g|^2}\right|&\lesssim&\frac{Na^{-2s}}{(Nz_1-a)^2a^{-2t}+N^2a^{-2s}z_2^2}\\
&=&\frac{1}{(z_1-\frac{a}{N})^2Na^{-2(t-s)}+Nz_2^2}\\
&\lesssim&\max\{\frac{1}{N^{1-2(t-s)}},\frac{1}{N}\}.\\
&=&\frac{1}{N^{1-2(t-s)}}.
\end{eqnarray*}
In sum, 
\begin{equation}
\label{eqn:case1ieq2}
 \left| \frac{\partial^2 g}{\partial y_i\partial y_j} \frac{1}{|\grad g|^2}\right|\lesssim \frac{1}{N^{1-2(t-s)}}
\end{equation}
for $a\notin (\frac{N}{2M},2MN)$.

Notice that the dominant term of $\grad h$ is
\[
w(y)\alpha(b+R_\theta A_a^{-1}y)e^{-(\phi(b+R_\theta A_a^{-1}y)-c)^2/\sigma^2}\cdot\frac{-2(\phi(b+R_\theta A_a^{-1}y)-c)}{\sigma^2}A_a^{-1}z
\]
and the other terms are of order $1$. Because $e^{-\frac{x^2}{\sigma^2}}\cdot \frac{|x|}{\sigma^2}\leq e^{-\frac{1}{2}}\cdot \frac{1}{\sigma\sqrt{2}}$, then
\[
\left|\frac{\grad h\cdot\grad g}{|\grad g|^2}\right|\lesssim\frac{1}{\sigma}\left|\frac{(A_a^{-1}z)\cdot \grad g}{|\grad g|^2}\right|+\left|\frac{1}{|\grad g|}\right|\lesssim N^\eta\left|\frac{(A_a^{-1}z)\cdot \grad g}{|\grad g|^2}\right|+\frac{1}{N^{1-t}}.
\]
Recall that $\grad g=2\pi(NA_a^{-1}z-(a^{1-t},0)^T)$, then 
\[
\frac{(A_a^{-1}z)\cdot \grad g}{|\grad g|^2}\approx\frac{(Nz_1-a)a^{-2t}z_1+Na^{-2s}z_2^2}{(Nz_1-a)^2a^{-2t}+N^2a^{-2s}z_2^2}.
\]

If $z_1z_2\neq0$, then $\left|\frac{Na^{-2s}z_2^2}{N^2a^{-2s}z_2^2}\right|=\frac{1}{N}$ and $\left| \frac{(Nz_1-a)a^{-2t}z_1}{(Nz_1-a)^2a^{-2t}}\right|\approx\frac{1}{|Nz_1-a|}\approx\frac{1}{N}$, which implies that $\left|\frac{(A_a^{-1}z)\cdot \grad g}{|\grad g|^2}\right|\lesssim \frac{1}{N}$. If $z_1z_2=0$, then it is easy to check that $\left|\frac{(A_a^{-1}z)\cdot \grad g}{|\grad g|^2}\right|\approx\frac{1}{N}$. Hence,
\begin{equation}
\label{eqn:case1ieq3}
\left|\frac{\grad h\cdot\grad g}{|\grad g|^2}\right|\lesssim N^\eta\left|\frac{(A_a^{-1}z)\cdot \grad g}{|\grad g|^2}\right|+\frac{1}{N^{1-t}}\lesssim \frac{1}{N^{1-\eta}}+\frac{1}{N^{1-t}}\lesssim\frac{1}{N^{1-t}}
\end{equation}
for $a\notin(\frac{N}{2M},2MN)$. 

By \eqref{eqn:case1ieq2} and \eqref{eqn:case1ieq3}, we have
\[
  \left|\int_{\R^2} h e^{ig} dy\right| = \left|\int_{\R^2} \grad\cdot\bigl( \frac{h\grad g}{i|\grad g|^2} \bigr)e^{ig}dy\right| \lesssim\left|\grad\cdot\bigl( \frac{h\grad g}{i|\grad g|^2} \bigr)\right|(||w||_{L^1}+||\grad w||_{L^1})\lesssim \frac{1}{N^{1-t}}
\]
for $a\notin(\frac{N}{2M},2MN)$. So, 
\[
W_f(a,\theta,b)=a^{-\frac{s+t}{2}}O(\eps),
\] 
when $N\gtrsim\eps^{\frac{-1}{1-t}}$ and $a\notin(\frac{N}{2M},2MN)$.

\textbf{Case 2:} $a\in (\frac{N}{2M},2MN)$ and $|\theta_{\grad\phi(b)}-\theta|\geq\theta_0$.

Observing that $\grad g(y)=2\pi A^{-1}_a R^{-1}_\theta (N\grad\phi(b+R_\theta A^{-1}_a y)-a\cdot e_\theta)$, we can expect $|\grad g|$ is large when $\theta_{\grad\phi(b)}$ is far away from $\theta$. Notice that $w(y)$ is in the Schwartz class, then $\exists C_m>0$ such that $|w(y)|\leq\frac{C_m}{y^m}$ for $|y|\geq 1$ and any $m$ large enough. So
\begin{eqnarray*}
 W_f(a,\theta,b)&=&a^{-\frac{s+t}{2}}\biggl(\int_{|y|\lesssim\eps^{-1/m}}f(b+R_\theta A^{-1}_a y)w(y) e^{-2\pi i a^{1-t}y_1} dy+O(\eps)\biggr).
\end{eqnarray*}
Define $D=\{y:|y|\lesssim\eps^{-1/m}\}$ and $D_+=\{y:|y|\lesssim\eps^{-1/m}+1\}$. Suppose $X_D(y)$ is a positive and smooth function compactly supported in $D_+$ such that $X_D(y)=1$ if $y\in D$, $||X_D||_{L^\infty}\leq 1$, then
\[
 W_f(a,\theta,b)=a^{-\frac{s+t}{2}}\biggl( O(\eps)+\int_{D_+}X_D(y)h(y)e^{ig(y)}dy\biggr).
\]
If $|\grad g(y)|$ is not vanishing in $D_+$, then apply the integral by parts to get
  \[
  \int_{D_+} X_Dh e^{ig} dy = \int_{D_+} X_Dh (L e^{i g}) dy =  -\int_{D_+} \grad\cdot\bigl( \frac{X_Dh\grad g}{i|\grad g|^2} \bigr)e^{ig}dy.
  \]
We are going to estimate $|\grad g(y)|$ when $a\in (\frac{N}{2M},2MN)$ and $|\theta_{\grad\phi(b)}-\theta|\geq\theta_0$. By Taylor expansion, 
\[
\grad\phi(b+R_\theta A^{-1}_a y)=\grad\phi(b)+\grad^2\phi(b^*)R_\theta A^{-1}_a y,
\]
where $b^*$ is between $b$ and $b+R_\theta A^{-1}_a y$. Notice that
\[
|\grad^2\phi(b^*)R_\theta A^{-1}_a y|\leq a^{-s}|\grad^2\phi(b^*)||y|\lesssim Ma^{-s}(\eps^{-1/m}+1)\leq\frac{\sin(\theta_0)}{2M},
\]
when $|y|\lesssim \eps^{-1/m}+1$ and $(\frac{2M^2}{\sin(\theta_0)})^{1/s}(\eps^{-1/m}+1)^{1/s}\leq a$. The latter one holds when $N\gtrsim(\eps^{-1/m}+1)^{1/(2s-t)}$ for $a\in(\frac{N}{2M},2MN)$.
So, when these conditions are satisfied, we have
\[
\grad \phi(b+R_\theta A^{-1}_a y)=\grad \phi(b)+v,
\]
with $|v|\leq\frac{\sin(\theta_0)}{2M}$. Recall the fact $|\theta_{\grad\phi(b)}-\theta|\geq \theta_0$, then it holds that
\begin{eqnarray*}
& &|A^{-1}_a R^{-1}_\theta (N\grad\phi(b+R_\theta A^{-1}_a y)-a\cdot e_\theta)|\\
&\geq&|NA^{-1}_a R^{-1}_\theta \grad\phi(b)-(a^{1-t},0)^{T}|-N|A^{-1}_a R^{-1}_\theta v|\\
&\geq&\sqrt{(r\cos\alpha-a)^2a^{-2t}+r^2a^{-2s}\sin^2\alpha}-\frac{N}{2M}\sin\theta_0a^{-s}\\
&\geq&ra^{-s}\sin\theta_0-\frac{N}{2M}\sin\theta_0a^{-s}\\
&\geq&\frac{N}{2M}\sin\theta_0a^{-s}\\
&\gtrsim&N^{1-t},
\end{eqnarray*}
where $\alpha=\theta_{\grad\phi(b)}-\theta$ and $r=|N\grad\phi(b)|\geq\frac{N}{M}$. Hence, we have \begin{equation}
\label{eqn:case2ieq1}
|\grad g(y)|\gtrsim N^{1-t}
\end{equation}
when $a\in (\frac{N}{2M},2MN)$, $|\theta_{\grad\phi(b)}-\theta|\geq\theta_0$, $N\gtrsim (\eps^{-1/m}+1)^{1/(2s-t)}$ and $y\in D_+$.

Next, we move on to estimate $\left|\frac{\grad( X_D h)\cdot\grad g}{|\grad g|^2}\right|$ and $\left| \frac{\partial^2 g}{\partial y_i\partial y_j} \frac{1}{|\grad g|^2}\right|$ for $i,j=1,2$, under the conditions that $a\in (\frac{N}{2M},2MN)$, $|\theta_{\grad\phi(b)}-\theta|\geq\theta_0$, $N\gtrsim (\eps^{-1/m}+1)^{1/(2s-t)}$ and $y\in D_+$. First,
\begin{equation}
\label{eqn:case2ieq2}
\left| \frac{\partial^2 g}{\partial y_i\partial y_j} \frac{1}{|\grad g|^2}\right|\leq \frac{Na^{-2s}}{|\grad g|^2}\leq \frac{N^{1-2s}}{N^{2-2t}}=\frac{1}{N^{1-2(t-s)}}.
\end{equation}
Second, as for $\left|\frac{\grad (X_Dh)\cdot\grad g}{|\grad g|^2}\right|$, we only need to estimate $\left|\frac{(A_a^{-1}z)\cdot \grad g}{|\grad g|^2}\right|$ for the similar reason in the last case.
As we have shown,
\[
\frac{(A_a^{-1}z)\cdot \grad g}{|\grad g|^2}\approx\frac{(Nz_1-a)a^{-2t}z_1+Na^{-2s}z_2^2}{(Nz_1-a)^2a^{-2t}+N^2a^{-2s}z_2^2}.
\]
If $z_1=0$, then $\left|\frac{(A_a^{-1}z)\cdot \grad g}{|\grad g|^2}\right|\approx \frac{1}{N}$.
If $z_1\neq 0$ and $\left|\frac{z_2}{z_1}\right|\gtrsim \frac{a^s}{a^t}$, then $|z_2|\gtrsim\frac{a^s}{Ma^t}$, since $|z|\geq\frac{1}{M}$. Hence,
\begin{eqnarray*}
\left|\frac{(A_a^{-1}z)\cdot \grad g}{|\grad g|^2}\right|&\lesssim&\frac{|(Nz_1-a)a^{-2t}z_1|+|Na^{-2s}z_2^2|}{N^2a^{-2s}z_2^2}\\
&\lesssim&\frac{|Nz_1-a|\cdot|z_1|}{N^2a^{2(t-s)}z_2^2}+\frac{1}{N}\\
&\lesssim&\frac{1}{Na^{t-s}|z_2|}+\frac{1}{N}\\
&\lesssim&\frac{1}{N}.
\end{eqnarray*}
If $z_1\neq 0$ and $\left|\frac{z_2}{z_1}\right|\lesssim \frac{a^s}{a^t}$, then 
\begin{eqnarray*}
\left|\frac{(A_a^{-1}z)\cdot \grad g}{|\grad g|^2}\right|&\leq&\frac{|(Nz_1-a)a^{-2t}z_1|+|Na^{-2s}z_2^2|}{|\grad g|^2}\\
&\lesssim&\frac{(|Nz_1|+a)a^{-2t}|z_1|+Na^{-2t}z_1^2}{N^{2-2t}}\\
&\lesssim&\frac{1}{N}.
\end{eqnarray*}
In sum,
\[
\left|\frac{(A_a^{-1}z)\cdot \grad g}{|\grad g|^2}\right|\lesssim\frac{1}{N},
\]
which implies that
\begin{equation}
\label{eqn:case2ieq3}
\left|\frac{\grad (X_Dh)\cdot \grad g}{|\grad g|^2}\right|\lesssim\frac{1}{N^{1-t}}.
\end{equation}
By \eqref{eqn:case2ieq2} and \eqref{eqn:case2ieq3}, we have
\[
\left|\int_{D_+} \grad\cdot\bigl( \frac{X_Dh\grad g}{i|\grad g|^2} \bigr)e^{ig}dy\right| \lesssim\left|\grad\cdot\bigl( \frac{X_Dh\grad g}{i|\grad g|^2} \bigr)\right|(||X_Dw||_{L^1}+||\grad(X_Dw)||_{L^1})\lesssim \frac{1}{N^{1-t}}
\]
for $a\in (\frac{N}{2M},2MN)$, $|\theta_{\grad\phi(b)}-\theta|\geq\theta_0$ and $N\gtrsim (\eps^{-1/m}+1)^{1/(2s-t)}$. So, 
\[
W_f(a,\theta,b)=a^{-\frac{s+t}{2}}O(\eps),
\] 
when $a\in (\frac{N}{2M},2MN)$, $|\theta_{\grad\phi(b)}-\theta|\geq\theta_0$ and
\[
N\gtrsim\max\{(\eps^{\frac{-1}{m}}+1)^{\frac{1}{2s-t}},\eps^{\frac{-1}{1-t}}\}.
\].

From the discussion in the two cases above, we see that 
\[
W_f(a,\theta,b)=a^{-\frac{s+t}{2}}O(\eps),
\]
if $a\notin(\frac{N}{2M},2MN)$ or $|\theta_{\grad\phi(b)}-\theta|\geq\theta_0$, when $N$ is sufficiently large. Hence, the proof of $(2)$ when $K=1$ is done.

\textbf{Step2:} Henceforth, we move on to prove $(1)$, i.e., to discuss the approximation of $W_f(a,\theta,b)$, when $a\in(\frac{N}{2M},2MN)$ and $|\theta_{\grad \phi(b)}-\theta|<\theta_0$. Recall that 
\begin{eqnarray*}
 W_f(a,\theta,b)&=&a^{-\frac{s+t}{2}}\bigg(\int_{y\in D}f(b+R_\theta A^{-1}_a y)w(y) e^{-2\pi i a^{1-t}y_1} dy+O(\eps)\bigg).
\end{eqnarray*}
Our goal is to get the following estimate
\begin{equation}
\label{main:est3}
 W_f(a,\theta,b)=a^{-\frac{s+t}{2}}\bigg(\int_{y\in D}f(b)w(y) e^{2\pi i (N\grad\phi(b)\cdot (R_\theta A^{-1}_a y)-a^{1-t}y_1)} dy+O(\eps)\bigg),
\end{equation}
for $N$ large enough.

First, we are going to show 
\begin{equation}
\label{main:est1}
 W_f(a,\theta,b)=a^{-\frac{s+t}{2}}\bigg(\int_{y\in D}e^{-\frac{(\phi(b)-c)^2}{\sigma^2}} \alpha(b+R_\theta A^{-1}_a y)w(y) e^{2\pi i (N \phi(b+R_\theta A^{-1}_a y)-a^{1-t}y_1)} dy+O(\eps)\bigg)
\end{equation}
for sufficiently large $N$.
Taylor expansion is applied again to obtain the following three expansions.
\[
\phi(b+R_\theta A^{-1}_a y)=\phi(b)+\grad\phi(b)\cdot(R_\theta A^{-1}_a y)+\frac{1}{2}(R_\theta A^{-1}_a y)^{T}\grad^2\phi(b^*)(R_\theta A^{-1}_a y),
\]
where $b^*$ is between $b$ and $b+R_\theta A^{-1}_a y$.
\begin{eqnarray*}
& &e^{-(\phi(b+R_\theta A^{-1}_a y)-c)^2/\sigma^2}\\
&=&e^{-(\phi(b)+\grad\phi(b)\cdot(R_\theta A^{-1}_a y)+\frac{1}{2}(R_\theta A^{-1}_a y)^{T}\grad^2\phi(b^*)(R_\theta A^{-1}_a y)-c)^2/\sigma^2}\\
&=&e^{-\frac{(\phi(b)-c)^2}{\sigma^2}}+e^{-\frac{(\lambda-c)^2}{\sigma^2}}\cdot\frac{-2(\lambda-c)}{\sigma^2}\big(\grad\phi(b)\cdot(R_\theta A^{-1}_a y)+\frac{1}{2}(R_\theta A^{-1}_a y)^{T}\grad^2\phi(b^*)(R_\theta A^{-1}_a y)\big),
\end{eqnarray*}
where $\lambda\in[\phi(b),\phi(b)+\grad\phi(b)\cdot(R_\theta A^{-1}_a y)+\frac{1}{2}(R_\theta A^{-1}_a y)^{T}\grad^2\phi(b^*)(R_\theta A^{-1}_a y)]$.
\[
\alpha(b+R_\theta A^{-1}_a y)=\alpha(b)+\grad\alpha(b^{**})\cdot(R_\theta A^{-1}_a y),
\]
where $b^{**}$ is between $b$ and $b+R_\theta A^{-1}_a y$.

The above Taylor expansions help us to estimate the effect of phase function $\phi(x)$ in the Gaussian term. We claim two estimates as follows.
\[
I_1=\int_{y\in D}\left|e^{-\frac{(\lambda-c)^2}{\sigma^2}}\cdot\frac{-2(\lambda-c)}{\sigma^2}\grad\phi(b)\cdot(R_\theta A^{-1}_a y)\alpha(b+R_\theta A^{-1}_a y)w(y)\right|dy\leq O(\eps)
\]
and
\[
I_2=\int_{y\in D}\left|e^{-\frac{(\lambda-c)^2}{\sigma^2}}\cdot\frac{-2(\lambda-c)}{\sigma^2} \frac{1}{2}(R_\theta A^{-1}_a y)^{T}\grad^2\phi(b^*)(R_\theta A^{-1}_a y) \alpha(b+R_\theta A^{-1}_a y)w(y)\right|dy\leq O(\eps).
\]
Because $e^{-\frac{x^2}{\sigma^2}}\cdot \frac{|x|}{\sigma^2}\leq e^{-\frac{1}{2}}\cdot \frac{1}{\sigma\sqrt{2}}$, we know
\[
I_2\lesssim\frac{1}{\sigma}\int_{y\in D}|y|^2a^{-2s}dy\lesssim\frac{1}{\sigma}a^{-2s}\eps^{-\frac{4}{m}}<\eps,
\]
if $a\gtrsim\sigma^{-\frac{1}{2s}}\eps^{-\frac{1+\frac{4}{m}}{2s}}$, which is true when 
\begin{equation}
\label{ieq:sgm1}
%N\gtrsim\sigma^{-\frac{1}{2s}}\eps^{-\frac{1+\frac{4}{m}}{2s}}.
N\gtrsim\eps^{-\frac{1+\frac{4}{m}}{2s-\eta}}.
\end{equation}

As for $I_1$, notice that $|\theta_{\grad\phi(b)}-\theta|<\theta_0$, then $|\theta_{R^{-1}_\theta\grad\phi(b)}|<\theta_0$. Let $\tilde{\theta}=\theta_{R^{-1}_\theta\grad\phi(b)}$ and $y=(y_1,y_2)^{T}$, then for $a\in(\frac{N}{2M},2MN)$
\begin{eqnarray*}
I_1&\lesssim&\frac{1}{\sigma}\int_{y\in D}\left|\grad\phi(b)\cdot(R_\theta A^{-1}_a y)\right|dy\\
&\lesssim&\frac{M}{\sigma}\int_{y\in D}\left|\frac{y_1}{a^t}\cos\tilde{\theta}+\frac{y_2}{a^s}\sin\tilde{\theta}\right|dy\\
&\lesssim&\frac{Md}{\sigma}\int_{y\in D}\max_{\gamma\in[0,2\pi)}\left|\frac{\cos\gamma\cos\tilde{\theta}}{a^t}+\frac{\sin\gamma\sin\tilde{\theta}}{a^s}\right|dy\\
&\lesssim&\frac{Md^3L}{\sigma},
\end{eqnarray*}
where $d\approx\eps^{-\frac{1}{m}}$ is the radius of $D$ and 
\[
L=\sqrt{\frac{\cos^2\tilde{\theta}}{a^{2t}}+\frac{\sin^2\tilde{\theta}}{a^{2s}}} \leq\sqrt{\frac{1}{a^{2t}}+\frac{\sin^2\theta_0}{a^{2s}}}\lesssim\max\{\frac{1}{a^t},\frac{|\sin\theta_0|}{a^s}\}\lesssim N^{-t}.
\]
So 
\[
I_1\lesssim\frac{Md^3L}{\sigma}\lesssim\frac{Md^3N^{-t}}{\sigma}\lesssim O(\eps),
\]
if 
\begin{equation}
\label{ieq:sgm2}
N\gtrsim \eps^{-\frac{1+\frac{3}{m}}{t-\eta}}.
\end{equation}
A direct result of the estimate of $I_1$ and $I_2$ is \eqref{main:est1} for  
\begin{equation}
\label{main:bound1}
N\gtrsim \max\{\eps^{-\frac{1+\frac{3}{m}}{t-\eta}},\eps^{-\frac{1+\frac{4}{m}}{2s-\eta}} \}.
\end{equation}

Second, we need to show 
\begin{equation}
\label{main:est2}
 W_f(a,\theta,b)=a^{-\frac{s+t}{2}}\bigg(\int_{y\in D}e^{-\frac{(\phi(b)-c)^2}{\sigma^2}} \alpha(b)w(y) e^{2\pi i (N \phi(b+R_\theta A^{-1}_a y)-a^{1-t}y_1)} dy+O(\eps)\bigg),
\end{equation}
which relies on the analysis of the effect of $\phi(x)$ on $\alpha(x)$ as follows. Since $a\in (\frac{N}{2M},2MN)$, then
\begin{eqnarray*}
I_3&=&\int_{y\in D}e^{-\frac{(\phi(b)-c)^2}{\sigma^2}}\left|\grad\alpha\cdot(R_\theta A^{-1}_a y)w(y)\right|dy\\
&\lesssim&\int_{y\in D}\left|R_\theta A^{-1}_a y\right|dy\\
&\lesssim&a^{-s}\eps^{-\frac{3}{m}}\\
&\lesssim&O(\eps)
\end{eqnarray*}
holds when 
\[
N\gtrsim\eps^{-\frac{1+\frac{3}{m}}{s}}.
\]
Then we derive \eqref{main:est2} by the estimate of $I_3$ and \eqref{main:est1} for $N\gtrsim\eps^{-\frac{1+\frac{3}{m}}{s}}$.

Finally, we should estimate the non-linear effect of $\phi(x)$ on the oscillatory pattern and show \eqref{main:est3} for sufficiently large $N$. If 
\[
N\gtrsim\eps^{-\frac{(1+\frac{4}{m})}{2s-1}},
\] 
then
\begin{eqnarray*}
I_4&=&\int_{y\in D}\left|e^{2\pi i(N\phi(b)+N\grad\phi(b)\cdot(R_\theta A^{-1}_a y)-a^{1-t}y_1)}\right|\cdot \left|e^{2\pi i \frac{N}{2}(R_\theta A^{-1}_a y)^{T}\grad^2\phi(R_\theta A^{-1}_a y)}-1\right|dy\\
&\lesssim&\int_{y\in D}\left|N(R_\theta A^{-1}_a y)^{T}\grad^2\phi(R_\theta A^{-1}_a y)\right|dy\\
&\lesssim&\int_{y\in D}Na^{-2s}|y|^2dy\\
&\lesssim&Na^{-2s}\eps^{-\frac{4}{m}}\\
&\lesssim&O(\eps)
\end{eqnarray*}
holds by the fact that $|e^{ix}-1|\leq|x|$ and $a\in(\frac{N}{2M},2MN)$. Then by \eqref{main:est2} and $I_4$, we have
\begin{eqnarray*}
W_f(a,\theta,b)&=&a^{-\frac{s+t}{2}}\bigg(f(b)\int_{y\in D}w(y) e^{2\pi i (N\grad\phi(b)\cdot (R_\theta A^{-1}_a y)-a^{1-t}y_1)} dy+O(\eps)\bigg)\\
&=&a^{-\frac{s+t}{2}}\bigg(f(b)\int_{R^2}w(y) e^{2\pi i (NA^{-1}_a R^{-1}_\theta\grad\phi(b)-(a^{1-t},0)^{T})\cdot y} dy+O(\eps)\bigg)\\
&=&a^{-\frac{s+t}{2}}\bigg(f(b)\widehat{w}\big(A^{-1}_a R^{-1}_\theta(a\cdot e_\theta-N\grad\phi(b))\big)+O(\eps)\bigg),
\end{eqnarray*}
for $a\in(\frac{N}{2M},2MN)$ and $|\theta_{\grad\phi(b)}-\theta|<\theta_0$, if $N$ is sufficiently large. This complete the proof of $(1)$ when $K=1$.

In sum, we have proved this lemma when $K=1$. The conclusion is also true for general $K$ by the linearity of general curvelet transform.
%In the general case where the function $f$ has $K$ components,
%\[
%f(x)=\sum^K_{k=1}f_k(x)=\sum^K_{k=1}e^{-\frac{(\phi_k(x)-c_k)^2}{\sigma_k^2}}\alpha_k(x)e^{2\pi iN\phi_k(x)},
%\]
%by the approximation above and linearity of curvelet transform, we conclude that for sufficiently large $N$, we have
%\[
%W_f(a,\theta,b)=a^{-\frac{s+t}{2}} \left( \sum_{k\in B_\theta} f_k(b)\widehat{w}\left(A^{-1}_a R^{-1}_\theta (a\cdot e_\theta-N\grad\phi_k(b))\right)+O(\eps) \right),
%\]
%if $(a,\theta)\in \Omega$. Otherwise, 
%\[
%W_f(a,\theta,b)=a^{-\frac{s+t}{2}}O(\eps).
%\]
%%  \begin{equation}
%%    \label{EE1}
%%    W_f(a,\theta,b)=
%%    \begin{cases}
%%  a^{-\frac{s+t}{2}} \left( \sum_{k\in B_\theta} e^{-\frac{(\phi_k(b)-c_k)^2}{\sigma^2_k}} \alpha_k(b)e^{2\pi iN \phi_k(b)}\widehat{w}\left(A^{-1}_a R^{-1}_\theta (a\cdot e_\theta-N\grad\phi_k(b))\right)+O(\eps) \right),
%%&  (a,\theta)\in \Omega  \\
%%   a^{-\frac{s+t}{2}}O(\eps),   
%%      & (a,\theta)\notin \Omega .
%%    \end{cases}
%%  \end{equation}
%Here $\Omega=\{(a,\theta):a\in(\frac{N}{2M},2MN),\theta\in\{\theta|B_\theta\neq\emptyset\}\}$, $B_\theta=\{k:|\theta_{\grad\phi_k(b)}-\theta|<\theta_0\}$, and $\theta_0=\arcsin((\frac{M}{N})^{t-s})$.
\end{proof}
%%Actually, we can arrive at a more precise estimate considering the support of $\widehat{w}$ as %follows.
%%\begin{lemma}
%%  \label{lem:B}
%%  Under the assumption of the theorem, we have the following estimation of $W_f(a,\theta,b)$ for %any $\eps$, when $N$ is sufficiently large.
%%  \begin{equation}
%%    \label{EEE1}
%%    W_f(a,\theta,b)=
%%    \begin{cases}
%%  a^{-\frac{s+t}{2}} \left( \sum_{k\in B_{a,\theta}} e^{-\frac{(\phi_k(b)-c_k)^2}{\sigma^2_k}} %\alpha_k(b)e^{2\pi iN \phi_k(b)}\widehat{w}\left(A^{-1}_a R^{-1}_\theta (a\cdot e_\theta-N\grad\phi_k(b))\right)+O(\eps) %%\right),
%%&  (a,\theta)\in \cup^K_{k=1}Z_{f,k} \\
%%   a^{-\frac{s+t}{2}}O(\eps),   
%%     & otherwise .
%%    \end{cases}
%%  \end{equation}
%%where $B_{a,\theta}=\{k:(a,\theta)\in Z_{f,k}\}$.
%%\end{lemma}

To prove Theorem \ref{thm:main}, we need one more lemma which estimates $\grad_bW_f(a,\theta,b)$.

\begin{lemma}
  \label{lem:C}
  Under the assumption of the theorem, we have
  \begin{eqnarray*}
    \grad_b W_f(a,\theta,b)=a^{-\frac{s+t}{2}}\left(2\pi i N \sum_{k:\ |\theta_{\grad\phi_k(b)}-\theta|<\theta_0}\grad\phi_k(b)f_k(b)\widehat{w}\left(A^{-1}_a R^{-1}_\theta(a\cdot e_\theta-N\grad\phi(b))\right) +O(\eps)\right),
    \label{E2}
  \end{eqnarray*}
when \[(a,\theta)\in\Omega=\left\{ (a,\theta):a\in\left(\frac{N}{2M},2MN\right),\exists k\ s.t.\ \left|\theta_{\grad\phi_k(b)}-\theta\right|<\theta_0\right\}.\]
\end{lemma}

\begin{proof}
  The proof is similar to the one of Lemma \ref{lem:A}. We only need to discuss the case $K=1$ and the case $K>1$ holds by the linearity of general curvelet transform. Suppose
\[
f(x)=e^{-\frac{(\phi(x)-c)^2}{\sigma^2}}\alpha(x)e^{2\pi iN\phi(x)},
\]
we have 
  \begin{eqnarray*}
   & & \grad_b W_f(a,\theta,b)\\
&=&\int_{\R^2}f(x) a^{\frac{s+t}{2}}\biggl( (-R_\theta A_a)\grad w(A_a R^{-1}_\theta (x-b)) + 2\pi ia e_\theta w(A_a R^{-1}_\theta (x-b))\biggr) e^{-2\pi ia(x-b)\cdot e_\theta} dx\\
&=& \int_{\R^2}f(b+R_\theta A^{-1}_a y)a^{-\frac{s+t}{2}}\biggl((-R_\theta A_a)\grad w(y)+2\pi iae_\theta w(y)\biggr)e^{-2\pi ia^{1-t} y_1}dy\\
&=&a^{-\frac{s+t}{2}}\bigg(f(b)\int_{\R^2}\bigl( (-R_\theta A_a)\grad w(y) +2\pi iae_\theta w(y) \bigr)e^{-2\pi i((a^{1-t},0)^{T}-NA^{-1}_a R^{-1}_\theta \grad\phi(b))\cdot y}dy+O(\eps)\bigg)\\
%&=&a^{-\frac{s+t}{2}}\bigg(f(b)(-R_\theta A_a)2\pi iA^{-1}_a R^{-1}_\theta(a\cdot e_\theta-N\grad\phi(b))\widehat{w}(A^{-1}_a R^{-1}_\theta(a\cdot e_\theta-N\grad\phi(b)))+O(\eps)\bigg)\\
%& &+a^{-\frac{s+t}{2}}\bigg(2\pi iae_\theta f(b)\widehat{w}(A^{-1}_a R^{-1}_\theta(a\cdot e_\theta-N\grad\phi(b)))+O(\eps)\bigg)\\
&=&a^{-\frac{s+t}{2}}\bigg(2\pi iN\grad\phi(b)f(b)\widehat{w}\left(A^{-1}_a R^{-1}_\theta(a\cdot e_\theta-N\grad\phi(b))\right)+O(\eps)\bigg)
  \end{eqnarray*}
for $a\in(\frac{N}{2M},2MN)$ and $|\theta_{\grad\phi(b)}-\theta|<\theta_0$, if $N$ satisfies the condition in Lemma \ref{lem:A}. Therefore, if $f$ has $K$ components, we know
\[
\grad_b W_f(a,\theta,b)=a^{-\frac{s+t}{2}}\bigg(\sum_{k:\ |\theta_{\grad\phi_k(b)}-\theta|<\theta_0}2\pi iN\grad\phi_k(b)f_k(b)\widehat{w}\left(A^{-1}_a R^{-1}_\theta(a\cdot e_\theta-N\grad\phi_k(b))\right)+O(\eps)\bigg),
\]
for $(a,\theta)\in\Omega$ and $N$ large enough.
\end{proof}

%\begin{lemma}
%  \label{lem:D}
%  Under the assumption of the theorem, we have
%  \begin{equation}
%    \grad_b W_f(a,\theta,b)=
%    a^{-\frac{s+t}{2}}
%    \left(2\pi i N \sum_{k\in B_{a,\theta}}\grad\phi_k(b)e^{-\frac{(\phi_k(b)-c_k)^2}{\sigma^2_k}}\alpha_k(b)e^{2\pi i N \phi_k(b)}\widehat{w}(A^{-1}_a R^{-1}_\theta(a\cdot e_\theta-N\grad\phi(b))) +O(\eps)\right)
%    \label{E2}
%  \end{equation}
%for $(a,\theta)\in \cup^K_{k=1}Z_{f,k}$, where $B_{a,\theta}=\{k:(a,\theta)\in Z_{f,k}\}$.
%\end{lemma}

With the above two lemmas proved, it is enough to prove Theorem \ref{thm:main}.
\begin{proof}
We shall start from $(i)$. $\{Z_{f,k}:1\leq k \leq K\}$ are disjoint as soon as $f(x)$ is a superposition of well-separated components. Let $(a,\theta,b)\in R_{f,\eps}$. By Lemma \ref{lem:A}, $(a,\theta)\in \Omega$. So, we have
\[
W_f(a,\theta,b)=a^{-\frac{s+t}{2}} \left( \sum_{k:\ |\theta_{\grad\phi_k(b)}-\theta|<\theta_0} f_k(b)\widehat{w}\left(A^{-1}_a R^{-1}_\theta (a\cdot e_\theta-N\grad\phi_k(b))\right)+O(\eps) \right).
\]
Therefore, $\exists k$ such that $\widehat{w}\left(A^{-1}_a R^{-1}_\theta (a\cdot e_\theta-N\grad\phi_k(b))\right)\neq 0$. By the definition of $Z_{f,k}$, we see that $(a,\theta,b)\in Z_{f,k}$. Hence, $R_{f,\eps}\subset\cup^K_{k=1}Z_{f,k}$.

To show $(ii)$, notice that $(a,\theta,b)\in R_{f,\eps}\cup Z_{f,k}$, then
\[
W_f(a,\theta,b)=a^{-\frac{s+t}{2}} \biggl( f_k(b)\widehat{w}\left(A^{-1}_a R^{-1}_\theta (a\cdot e_\theta-N\grad\phi_k(b))\right)+O(\eps) \biggr),
\]
and 
\[
 \grad_b W_f(a,\theta,b)=a^{-\frac{s+t}{2}}\biggl(2\pi i N \grad\phi_k(b)f_k(b)\widehat{w}\left(A^{-1}_a R^{-1}_\theta(a\cdot e_\theta-N\grad\phi(b))\right) +O(\eps)\biggr).
\]
%So 
%\[
%v_f(a,\theta,b)=\frac{N \grad\phi_k(b)e^{-\frac{(\phi_k(b)-c_k)^2}{\sigma^2_k}}\alpha_k(b)e^{2\pi i N \phi_k(b)}\widehat{w}(A^{-1}_a R^{-1}_\theta(a\cdot e_\theta-N\grad\phi(b))) +O(\eps)}{e^{-\frac{(\phi_k(b)-c_k)^2}%{\sigma^2_k}} \alpha_k(b)e^{2\pi iN \phi_k(b)}\widehat{w}\left(A^{-1}_a R^{-1}_\theta (a\cdot e_\theta-N\grad\phi_k(b))\right)+O(\eps) }.
%\]
Let $g=f_k(b)\widehat{w}\left(A^{-1}_a R^{-1}_\theta(a\cdot e_\theta-N\grad\phi(b))\right)$, then
\[
v_f(a,\theta,b)=\frac{N\grad\phi_k(b)g+O(\eps)}{g+O(\eps)}.
\]
Since $|W_f(a,\theta,b)|\gtrsim\sqrt{\eps}$ for $(a,\theta,b)\in R_{f,\eps}$, then $|g|\gtrsim\sqrt{\eps}$. So 
\[
\frac{|v_f(a,\theta,b)-N\grad\phi_k(b)|}{|N\grad\phi_k(b)|}\lesssim \left|\frac{O(\eps)}{g+O(\eps)}\right|\lesssim\sqrt{\eps}.
\]
\end{proof}

The assumption $\frac{1}{2}<s<t<1$ and $\eta<t$ are essential to the proof. However, we have not arrived to a clear opinion on the optimal values of these parameters. The difference $t-s$ allows us to construct directional needle-like curvelets in order to approximate banded wave-like components or wavefronts and capture the oscillatory behavior better. When $t$ and $\eta$ approach to $1$, and $s$ gets close to $\frac{1}{2}$, we can expect that the synchrosqueezed curvelet transform can separate banded components of width approximately $O(N^{-1})$, if $m$ is large enough. On the other hand, the lower bound $s>1/2$ ensures that the support of each curvelet is sufficiently small in space so that the second order properties of the phase function (such as the curvature of wavefronts) do not affect the estimate of local wave-vectors. The upper bound $t<1$ guarantees sufficient resolution to detect different components with large wavenumbers. 

In Theorem \ref{thm:main}, although the lower bound of $N$ could be optimized, $N$ is required to be sufficiently large so that the local wave-vector can be precisely captured by synchrosqueezing. On the other hand, the local wave-vector is not well defined for low frequency component. In fact, in the presence of such component, each high oscillatory component is still squeezed into a well-separated sharpened representation in the high frequency part of Fourier domain. Therefore, the low frequency component would be identified precisely by subtracting high frequency components.

\section{Implementation of the transform}
\label{sec:algo}

In this section, we describe the discrete synchrosqueezed
curvelet transform and the mode decomposition in detail. Subsection \ref{sub:mode} has discussed the key ideas of mode decomposition by SSCT. Let us describe the whole framework now.
Suppose $f(x)$ is a superposition of several well-separated components, the mode decomposition by SSCT consists of the following steps:
\begin{enumerate}[(i)]
\item Apply the general curvelet transform to obtain $W_f(a,\theta,b)$
  and the gradient $\grad_b W_f(a,\theta,b)$;
\item Compute the local wave-vector estimate $v_f(a,\theta,b)$
  and concentrate the energy around it to get $T_f(v,b)$;
\item Separate the essential supports of the concentrated phase space energy distribution $T_f(v,b)$ into several components by clustering techniques;
\item Restrict $W_f(a,\theta,b)$ to each resulting component and reconstruct corresponding intrinsic mode functions using the dual frame.
\end{enumerate}
We first introduce a discrete implementation of the general curvelet transform in
Section \ref{sub:sct} for Step (i) and Step (iv). Clustering methods will be discussed later in Section \ref{sub:cluster}. The full discrete algorithm will then be
summarized in Section \ref{sub:full}.

%---------------------------------------------------------
\subsection{Discrete general curvelet transforms}
\label{sub:sct}

For simplicity, we consider functions that are periodic over the unit
square $[0,1)^2$ in 2D. If it is not the case, the functions will be periodized by multiplying a smooth decaying function near the boundary of $[0,1)^2$. Let
\[
X = \{ (n_1/L,n_2/L): 0 \le n_1, n_2,<L, n_1,n_2 \in \Z\}
\]
be the $L\times L$ spatial grid at which these functions are sampled.
The corresponding $L\times L$ Fourier grid is
\[
\Xi = \{ (\xi_1,\xi_2): -L/2\le \xi_1,\xi_2<L/2, \xi_1,\xi_2 \in \Z\}.
\]
For a function $f(x)\in \ell^2(X)$, the discrete forward Fourier
transform is defined by
\[
\widehat{f}(\xi) = \frac{1}{L} \sum_{x\in X} e^{-2\pi i x\cdot \xi} f(x).
\]
For a function $g(\xi)\in \ell^2(\Xi)$, the discrete inverse Fourier
transform is 
\[
\check{g}(x) = \frac{1}{L} \sum_{\xi\in \Xi} e^{2\pi i  x\cdot \xi} g(\xi).
\]
In both transforms, the factor $1/L$ ensures that these discrete transforms
are isometric between $\ell_2(X)$ and $\ell_2(\Xi)$.

\begin{figure}[ht!]
  \begin{center}
    \begin{tabular}{cc}
     \includegraphics[height=2.4in]{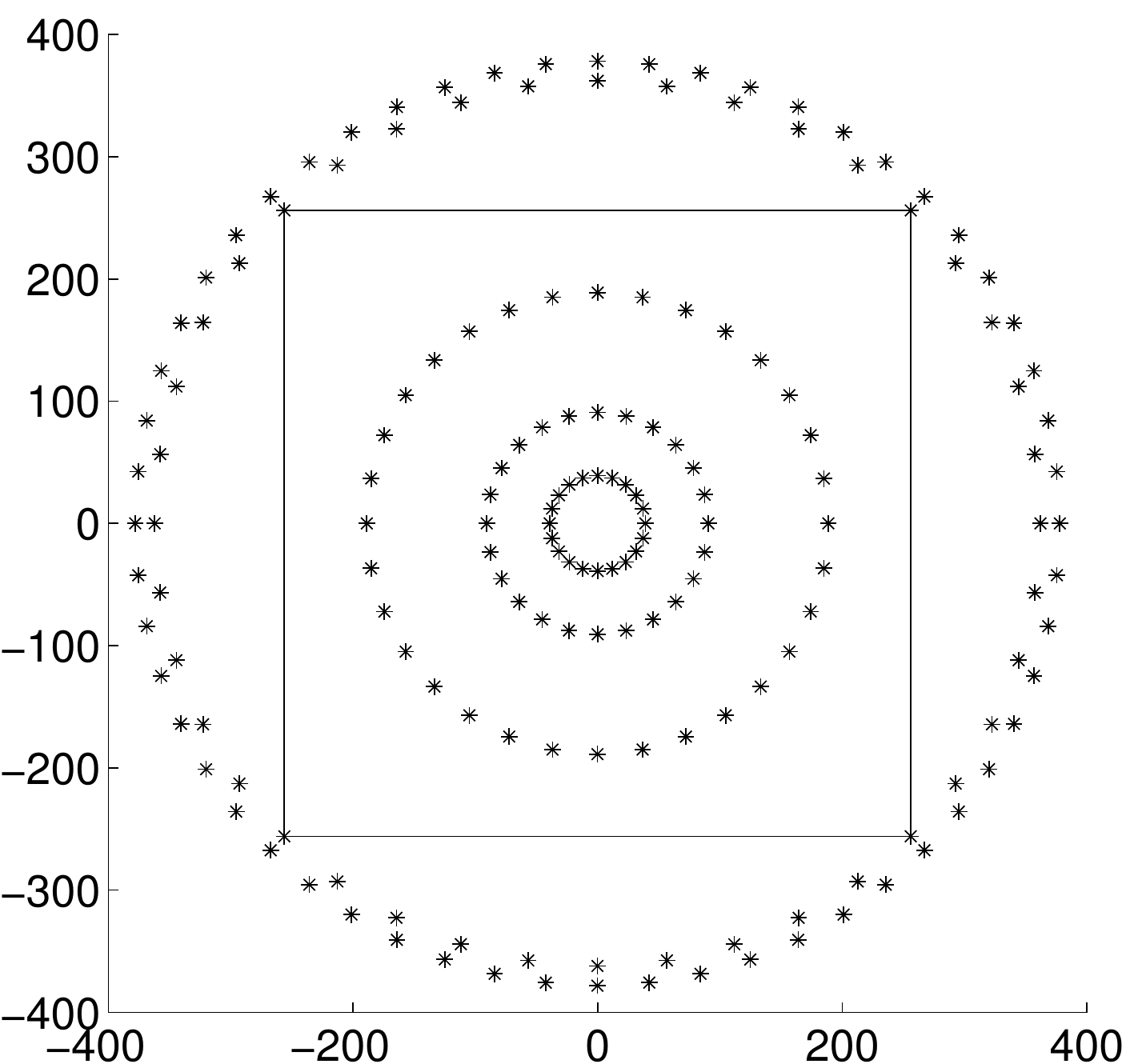} & \includegraphics[height=2.4in]{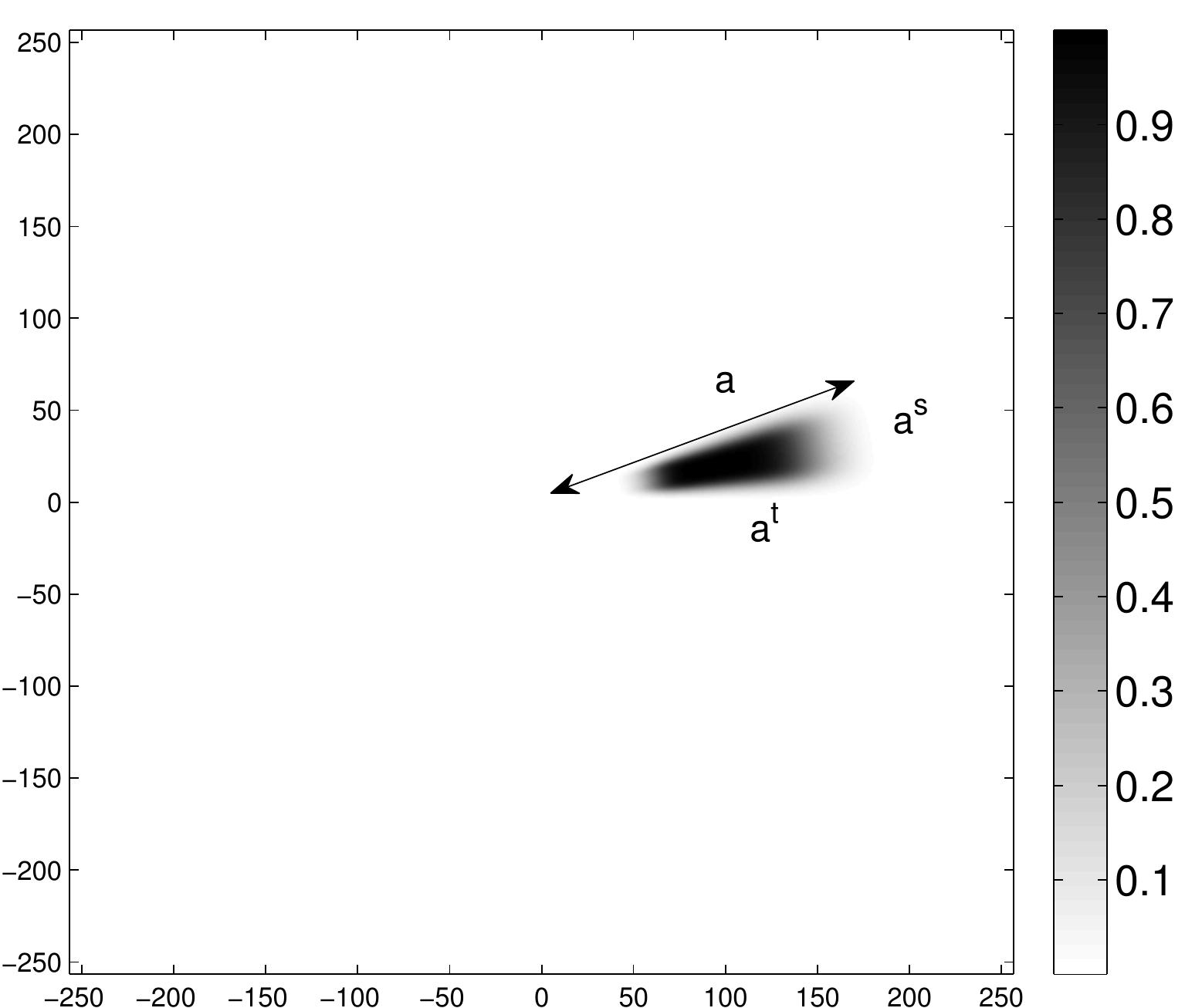}\\
    \end{tabular}
  \end{center}
  \caption{Left: Sampled point set $P$ in Fourier domain for an image of size $512\times512$. Each point represents the center of the support of a window function. The window function centered at the origin is supported on a disk and is not indicated in this picture. The size of finest scale is set to be small (e.g. $16$) in order to save memory. Right: An example of fan-shaped window function $g_{a,\theta}(\xi)$.}
  \label{fig:tiling}
\end{figure}

In order to design a discrete curvelet transform, we need to
specify how to decimate the Fourier domain $(a,\theta)$ and the position space
$b$. Let us first consider the Fourier domain $(a,\theta)$.  In the continuous
setting, the Fourier transform $\widehat{w_{a\theta b}}(\xi)$ for a fixed $(a,\theta)$ value have the profile
\begin{equation}
  \label{eqn:wnd}
  a^{-\frac{s+t}{2}} \widehat{w}(A_a^{-1}R^{-1}_\theta(\xi-a\cdot e_\theta)),
\end{equation}
modulo complex modulation. In the discrete setting, we sample the
Fourier domain $[-L/2,L/2)^2$ with a set of points $P$ (Figure \ref{fig:tiling} left) and associate
with each $(a,\theta)\in P$ a window function $g_{a,\theta}(\xi)$ (Figure \ref{fig:tiling} right) that behaves
qualitatively as $\widehat{w}(A_a^{-1}R^{-1}_\theta(\xi-a\cdot e_\theta))$. More precisely,
$g_{a,\theta}(\xi)$ is required to satisfy the following conditions:
\begin{itemize}
\item $g_{a,\theta}(\xi)$ is non-negative and centered at $a\cdot e_\theta$ with a compact fan-shaped support of length $O(a^t)$ and width $O(a^{s})$, which is approximately a directional elliptical support $\{\xi : |A_a^{-1}R^{-1}_\theta(\xi-a\cdot e_\theta)|\leq 1\}$.
\item $g_{a,\theta}(R_\theta A_a\tau + a\cdot e_\theta)$ is a sufficiently smooth function of
  $\tau$, thus making the discrete curvelets to decay rapidly in the
  spatial domain;
\item $C_1 \le \int |g_{a,\theta}(R_\theta A_a\tau + a\cdot e_\theta)|^2 d\tau \le C_2$ for positive
  constants $C_1$ and $C_2$, independent of $(a,\theta)$;
\item In addition, for any $\xi\in [-L/2,L/2)^2$, $\sum_{(a,\theta)\in P}
  |g_{a,\theta}(\xi)|^2 = 1$.
\end{itemize}
We follow the discretization and construction of frames in \cite{Candes2004} to specify the set $P$ and window functions, and refer to \cite{Candes2006} for detail implementation. The difference here is that, we do not restrict angular scaling parameter to $s=\frac{1}{2}$ and radial scaling parameter to $t=1$. This allows us to adaptively adjust the size of tiles according to data structure. In the construction of the tiling in this article, the scaling parameters $s$ and $t$ remain constant as the scale changes.

The decimation of the position space $b$ is much easier; we simply
discretize it with an $L_B\times L_B$ uniform grid as follows:
\[
B = \{ (n_1/L_B,n_2/L_B): 0 \le n_1, n_2<L_B, n_1,n_2 \in \Z\}.
\]
The only requirement is that $L_B$ is large enough so that a sampling grid of size $L_B\times L_B$ can cover the supports of all window functions.

For each fixed $(a,\theta)\in P$ and $b\in B$, the discrete curvelet, still
denoted by $w_{a\theta b}(x)$ without causing much confusion, is defined
through its Fourier transform as
\[
\widehat{w_{a\theta b}}(\xi) = \frac{1}{L_a} e^{-2\pi i b\cdot \xi} g_{a,\theta}(\xi)
\]
for $\xi \in \Xi$ with $L_a=a^{\frac{s+t}{2}}$. Applying the discrete inverse Fourier transform
provides its spatial description
\[
w_{a\theta b}(x) = \frac{1}{L\cdot L_a} \sum_{\xi\in \Xi} e^{2\pi i
  (x-b)\cdot\xi} g_{a,\theta}(\xi).
\]
For a function $f(x)$ defined on $x\in X$, the discrete curvelet
transform is a map from $\ell_2(X)$ to $\ell_2(P\times B)$, defined by
\begin{equation}
  W_f(a,\theta,b) = \langle w_{a\theta b}, f \rangle =  \langle \widehat{w_{a\theta b}}, \widehat{f} \rangle = 
  \frac{1}{L_a} \sum_{\xi\in\Xi} e^{2\pi i b\cdot\xi} g_{a,\theta}(\xi) \widehat{f}(\xi).
  \label{eqn:D1}
\end{equation}
We can introduce an inner product on the space $\ell_2(P\times B)$ as
follows: for any two functions $g(a,\theta,b)$ and $h(a,\theta,b)$,
\[
\langle g, h \rangle = \sum_{(a,\theta)\in P,b\in B} \overline{g(a,\theta,b)} h(a,\theta,b)
 .
\]
The following result shows that $\{w_{a\theta b}: (a,\theta,b)\in P\times B\}$ forms
a tight frame when equipped with this inner product.
\begin{proposition}
  For any function $f(x)$ for $x\in X$, we have
  \[
  \sum_{(a,\theta) \in P,b\in B} |W_f(a,\theta,b)|^2 \left(L_a/L_B\right)^2   = \|f\|_2^2.
  \]
\end{proposition}
\begin{proof}
  From the definition of the curvelet transform, we have
  \begin{align*}
    \sum_{(a,\theta) \in P,b\in B} |W_f(a,\theta,b)|^2 \left(L_a/L_B\right)^2   &= \sum_{(a,\theta) \in P,b\in B} \left| \sum_{\xi\in\Xi} \frac{1}{L_a}  e^{2\pi ib\cdot\xi} g_{a,\theta}(\xi) \widehat{f}(\xi) \right|^2  \left(L_a/L_B\right)^2\\
    &= \sum_{(a,\theta) \in P} \sum_{\xi\in\Xi} \left| g_{a,\theta}(\xi) \widehat{f}(\xi) \right|^2 \\
    &= \sum_{\xi\in\Xi} |\widehat{f}(\xi)|^2.
  \end{align*}
\end{proof}
For a function $h(a,\theta,b)$ in $\ell_2(P\times B)$, the transpose of the
curvelet transform is given by
\begin{equation}
  W^t_h(x) := \sum_{(a,\theta) \in P,b\in B} h(a,\theta,b) w_{a\theta b}(x)   \left(L_a/L_B\right)^2.
  \label{eqn:D2}
\end{equation}
The next result shows that this transpose operator allows us to
reconstruct $f(x),x\in X$ from its curvelet transform
$W_f(a,\theta,b),(a,\theta,b)\in P\times B$.
\begin{proposition}
  For any function $f(x)$ with $x\in X$,
  \[
  f(x) = \sum_{(a,\theta) \in P,b\in B} W_f(a,\theta,b) w_{a\theta b}(x) \left(L_a/L_B\right)^2 .
  \]
\end{proposition}
\begin{proof}
  Let us consider the Fourier transform of the right hand side. It is equal to
  \begin{align*}
    & \sum_{(a,\theta) \in P,b\in B} \left( \sum_{\eta\in\Xi} \frac{1}{L_a} e^{2\pi i b\cdot\eta} g_{a,\theta}(\eta)\widehat{f}(\eta)\right) 
    \cdot\frac{1}{L_a} e^{-2\pi i b\cdot\xi} g_{a,\theta}(\xi)    \left(L_a/L_B\right)^2   \\
    =& \sum_{(a,\theta) \in P} \left(
    \sum_{\eta\in \Xi} \frac{1}{L_B^2} \left(\sum_{b\in B} e^{2\pi i b \cdot(\eta-\xi)} g_{a,\theta}(\eta) \widehat{f}(\eta)\right)
    \right) g_{a,\theta}(\xi)  \\
    =& \sum_{(a,\theta) \in P} (g_{a,\theta}(\xi))^2 \widehat{f}(\xi) = \widehat{f}(\xi),
  \end{align*}
  where the second step uses the fact that in the $\eta$ sum only the
  term with $\eta=\xi$ yields a nonzero contribution.
\end{proof}

Let us now turn to the discrete approximation of $\grad_b
W_f(a,\theta,b)$. From the continuous definition \ref{def:GCT}, we have
\[
\grad_b W_f(a,\theta,b) = \grad_b \langle \widehat{w}_{a\theta b},\widehat{f}\rangle = 
\langle -2\pi i \xi\widehat{w}_{a\theta b}(\xi),\widehat{f}(\xi) \rangle.
\]
Therefore, we define the discrete gradient $\grad_b W_f(a,\theta,b)$
in a similar way
\begin{equation}
  \grad_b W_f(a,\theta,b) = \sum_{\xi\in\Xi} \frac{1}{L_a} 2\pi i \xi e^{2\pi
    i b\cdot\xi} g_{a,\theta}(\xi) \widehat{f}(\xi).
  \label{eqn:D3}
\end{equation}

The above definitions give rise to fast algorithms for computing the
forward general curvelet transform, its transpose, and the discrete
gradient operator. All three algorithms heavily rely on the fast
Fourier transform (FFT). The detailed implementation of these fast algorithms has been discussed in \cite{SSWPT}. The computational cost of all three algorithms is $O(L^2\log L + L^{2-s-t} L_B^2 \log L_B)$ with $L_B$ large enough so that a grid of size $L_B\times L_B$ can cover the supports of all window functions. If we choose $L_B$ to be of the same order as $L^t$, the complexity of these algorithms is $O(L^{2+t-s} \log L)$.

%---------------------------------------------------------
\subsection{Clustering in the phase space}
\label{sub:cluster}
In the proof of Theorem \ref{thm:main}, the radial separation and angular separation conditions play an important role in describing the well-separated condition. Therefore, the polar coordinate is used to quantify distance in the Fourier domain, which motivates the following clustering method used in the numerical examples of this article. Before introducing the algorithm, some notations are defined below.
\begin{enumerate}
\item We associate any point $p$ in the 4D phase space with $(x_p,a_p,\theta_p)$, where $x_p$ is the projection of $p$ in the 2D spatial domain and $(a_p\cos\theta_p,a_p\sin\theta_p)$ is the projection of $p$ in the 2D Fourier domain.
\item We say that $(p,q)$ is a pair of adjacent points with parameter $(d_0,\theta_0,R_0)$, if
\begin{itemize}
\item $|x_p-x_q| \leq d_0$.
\item $|a_p-a_q|\leq R_0$. 
\item $\min\{|\theta_p-\theta_q|,2\pi-|\theta_p-\theta_q|\}\leq \theta_0$.
\end{itemize}
\item We say that a point set $S$ is a cluster with parameter $(d_0,\theta_0,R_0)$, if $\forall p_1,p_2\in S$, $\exists q_i\in S$ $i=1,\dots,n$ such that $(p_1,q_1)$, $(q_n,p_2)$ and $(q_i,q_{i+1})$ are pairs of adjacent points with parameter $(d_0,\theta_0,R_0)$ for $i=1,\dots,n-1$.
\item Two point sets $S_1$ and $S_2$ are defined to be separated with parameter $(d_0,\theta_0,R_0)$, if $\forall p\in S_1$ and $\forall q\in S_2$, $(p,q)$ is not a pair of adjacent points with parameter $(d_0,\theta_0,R_0)$.
\end{enumerate}
With the notations above, we are ready to state the polar clustering algorithm. 
%As we have discussed, it is efficient enough to split the set $S$ only once to obtain clusters $U_1,\dots,U_K$.

\begin{algo}\label{alg:polar}
Polar clustering algorithm
\begin{algorithmic}[1]
\State Input: $S$ is the set of points to be separated. Set up a threshold distance $d_0$ for 2D spatial domain, a threshold angle $\theta_0$ and a threshold radius $R_0$ in the 2D Fourier domain. 
\State Output: Clustered point sets $S_1,\dots,S_n$.
\Function {}{} {\tt PolarCluster}{$(S,d_0,\theta_0,R_0)$}
\State Separate $S$ into $n$ clusters $S_1,\dots,S_n$ s.t. 
\State each $S_i$ is a cluster with $(d_0,\theta_0,R_0)$, 
\State and $S_i$ and $S_j$ are separated with $(d_0,\theta_0,R_0)$ for $i\neq j$.
\State \Return $\{S_1,\dots,S_n\}$
\EndFunction
\end{algorithmic}
\end{algo}

The cost of computation and memory of Algorithm \ref{alg:polar} is extremely high. Suppose the size of given data $f(x)$ is $L\times L$ and there is $K$ components with wavenumbers of $O(L)$. By Theorem \ref{thm:main}, each synchrosqueezed energy distribution $T_{f_k}(v,b)$ is surrounding its 2D wave-vector surface within a distance of $O(L\sqrt{\epsilon})$. Hence, the total number of nonzero grid points in the 4D phase space s.t. $T_f(v,b)\ge \delta$ is of order $KL^4\epsilon$, which is an impractical number for clustering. To reduce the cost, we should apply similar clustering methods first in the 2D Fourier domain at each location, which results in $O(K)$ clusters at each location. Afterward, a clustering method is applied to the point set of reduced size of $O(KL^2)$ in 4D phase space.

%---------------------------------------------------------
\subsection{Description of the full algorithm}
\label{sub:full}

With the fast discrete synchrosqueezed transforms and clustering algorithms available, we now go through the steps of the synchrosqueezed curvelet transform.

For a given function $f(x)$ defined on $x\in X$, we apply fast algorithms to compute $W_f(a,\theta,b)$ and $\grad_bW_f(a,\theta,b)$. Then the local wave-vector estimate $v_f(a,\theta,b)$ is computed by
\[
v_f(a,\theta,b) = \frac{\grad_b W_f(a,\theta,b)}{2\pi i W_f(a,\theta,b)}
\]
for $(a,\theta) \in P,b\in B$ with $W_f(a,\theta,b)\neq 0$ (indeed, $|W_f(a,\theta,b)| \ge \sqrt{\epsilon}$ in the numerical implementation).

The energy resulting in $\Re v_f(a,\theta,b)$ should be stacked up to obtain $T_f(\Re v_f(a,\theta,b),b)$. To realize this step, a two dimensional Cartesian grid of step size $\Delta$ is generated to discretize the Fourier domain of $T_f(v,b)$ in variable $v$ as follows:
\[
V = \{(n_1\Delta,n_2\Delta): n_1,n_2\in \Z\}.
\]
At each $v = (n_1\Delta,n_2\Delta) \in V$, we associate a cell $D_v$
centered at $v$
\[
D_v = 
\left[(n_1-\frac{1}{2})\Delta,(n_1+\frac{1}{2})\Delta\right) \times
  \left[(n_2-\frac{1}{2})\Delta,(n_2+\frac{1}{2})\Delta\right).
\]
Then $T_f(v,b)$ is estimated by
\[
T_f(v,b) = \sum_{(a,\theta,b): \Re v_f(a,\theta,b)\in D_v }| W_f(a,\theta,b)|^2  \left(L_a/L_B\right)^2.
\]
%Without any thresholding, one can check that 
%\[
%\sum_{v\in V,b\in B} T_f(v,b) = \sum_{(a,\theta) \in P,b\in B} |W_f(a,\theta,b)|^2
%  = \|f\|_2^2.
%\]

Suppose that $f(x)$ is a superposition of $K$ well-separated banded intrinsic
mode functions:
\[
f(x) = \sum_{k=1}^K f_k(x) = \sum_{k=1}^K e^{-(\phi_k(x)-c_k)^2/\sigma_k^2}\alpha_k(x) e^{2\pi i N\phi_k(x)}.
\]
In the discrete implementation, we choose a threshold parameter $\delta>0$ and define the set $S$ to be
\[
\{ (v,b): v\in V, b\in B, T_f(v,b) \ge \delta \}.
\]
After synchrosqueezing, $T_f(v,b)$ is essentially supported in the phase space near $K$ ``discrete'' surfaces $\{(N\phi_k(b),b), b\in B\}$.  Hence, under the separation condition given by Theorem \ref{thm:main}, $S$ will have $K$ well-separated clusters $U_1,\ldots,U_K$, and they would be identified by clustering methods in the last subsection. 

Once we discover $U_1,\dots,U_K$, we can define $W_{f_k}(a,\theta,b)$ by restricting $W_f(a,\theta,b)$ to
the set $\{(a,\theta,b): \Re v_f(a,\theta,b)\in U_k\}$. Then, we can recover each intrinsic mode function efficiently using the fast algorithm discussed to compute
\[
f_k(x) = \sum_{(a,\theta)\in P,b\in B} W_{f_k}(a,\theta,b) w_{a\theta b}(x)  \left(L_a/L_B\right)^2.
\]

%----------------------------------------------------------
\section{Numerical Results}
\label{sec:results}

In this section, we start with error analysis of local wave-vector estimation using synchrosqueezed curvelet transform, and compare it with synchrosqueezed wave packet transform. Afterward, some mode decomposition examples of synthetic and real data will be presented to illustrate the efficiency of proposed synchrosqueezed curvelet transform. For all the synthetic examples in this section, the size $L$ of the Cartesian grid $X$ of the discrete algorithm is $512$, the threshold value $\epsilon=10^{-4}$ for $W_f(a,\theta,b)$. The scaling parameters of synchrosqueezed curvelet transform are $t=1-\frac{1}{8}$ and $s=\frac{1}{2}+\frac{1}{8}$, as an appropriate balance as discussed previously. In the meantime, we chose $t=s=\frac{1}{2}+\frac{1}{8}$ to construct discrete synchrosqueezed wave packet transform for a reasonable comparison. In all the decomposition problems, Algorithm \ref{alg:polar} with application dependent parameters is applied and it provides desired solutions. We will only present relevant recovered components to save space.

\subsection{Instantaneous wave-vector estimation}
\label{sub:iwe}

In Theorem \ref{thm:main}, we have seen that the estimate $v_f(a,\theta,b)$ approximates the local wave-vector at $b$, if $|W_f(a,\theta,b)| \ge a^{-\frac{s+t}{2}} \sqrt{\eps}$. Since $a\ge 1$ as we discussed after the Definition \ref{def:GC2D}, it is useful to consider a simple and universal threshold criteria $|W_f(a,\theta,b)| \ge \sqrt{\eps}$, which amounts to a smaller region of the essential support of $W_f(a,\theta,b)$. In such region, though $v_f(a,\theta,b)$ provides an accurate estimate of the local wave-vector at each $b$, it is more rational to average them up to obtain a unique local wave-vector estimate for each fixed $b$. By the definition of synchrosqueezed energy distribution, $T_f(\Re v_f(a,\theta,b),b)$ truly reflects a natural weight of $v_f(a,\theta,b)$ in variables $a$ and $\theta$. Hence, we define the ${\it mean}$ local wave-vector estimate at $b$ to be
\[
v_f^m(b)=\frac{\sum_{(a,\theta)}|T_f(\Re v_f(a,\theta,b),b)|v_f(a,\theta,b)}{\sum_{(a,\theta)}|T_f(\Re v_f(a,\theta,b),b)|}.
\]
In the presence of noise, a threshold $\delta$ proportional to noise level is set up for $T_f(\Re v_f(a,\theta,b),b)$ to uncover the dominant estimate. Correspondingly, we define the thresholded ${\it mean}$ local waveform estimate as
\[
v_f^{m,\delta}(b)=\frac{\sum_{(a,\theta)\in \Omega_\delta(b)}|T_f(\Re v_f(a,\theta,b),b)|v_f(a,\theta,b)}{\sum_{(a,\theta)\in \Omega_\delta(b)}|T_f(\Re v_f(a,\theta,b),b)|},
\]
where $\Omega_\delta(b)=\{ (a,\theta):|T_f(\Re v_f(a,\theta,b),b)|\geq \delta\}$. In a noiseless case, $v_f^m(b) = v_f^{m,0}(b)$. Using this estimate, we can define the relative error $R_\delta(b)$ between $v_f^{m,\delta}(b)$ and the exact
local wave-vector $N \grad\phi(b)$ as
\[
R_\delta(b)=\frac{|v_f^{m,\delta}(b)-N\grad\phi(b)|}{|N\grad\phi(b)|}.
\]

\paragraph{Example 1.} We test the accuracy for a noise free deformed plane wave $f(x)=\alpha(x)e^{2\pi i N\phi(x)}$ with $\alpha(x)=1$, $\phi(x)=\phi(x_1,x_2) = x_1 + (1-x_2) +
0.1 \sin(2\pi x_1) + 0.1 \sin(2\pi (1-x_2))$, and $N=135$ (see Figure \ref{fig:err_full} left). It is a special case in Definition \ref{def:IMTF} with banded parameter $\sigma=\infty$. The relative error $R^0(b)$ of SSCT shown in Figure \ref{fig:err_full} (middle) is of order $10^{-2}$, which agrees with Theorem \ref{thm:main} on that the relative approximation error is of order $O(\sqrt{\eps})$. The synchrosqueezed wave packet transform and the synchrosqueezed curvelet transform share the same accuracy in this case shown by Figure \ref{fig:err_full} middle and right.

\begin{figure}[ht!]
  \begin{center}
    \begin{tabular}{ccc}
      \includegraphics[height=1.6in]{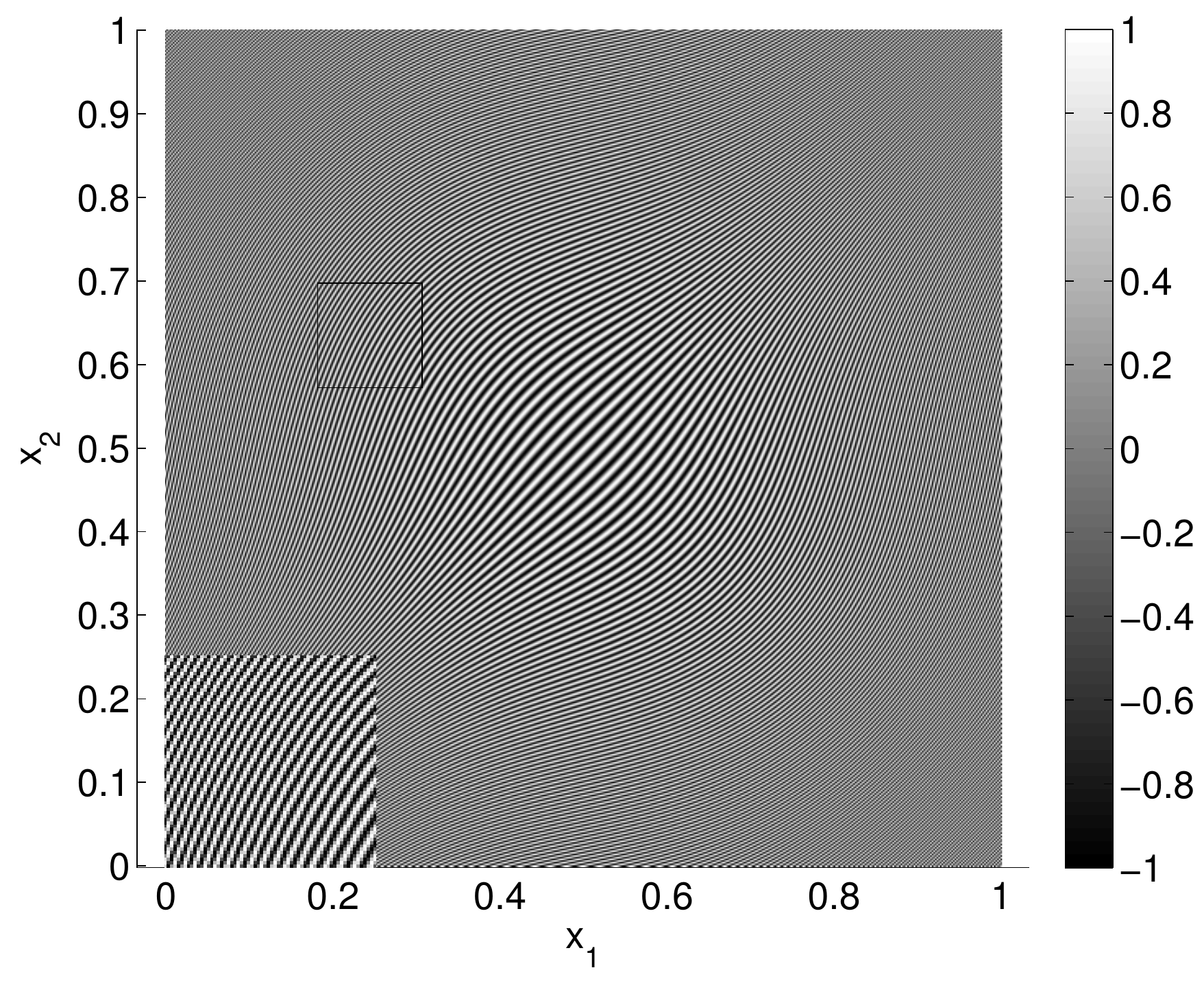} & \includegraphics[height=1.6in]{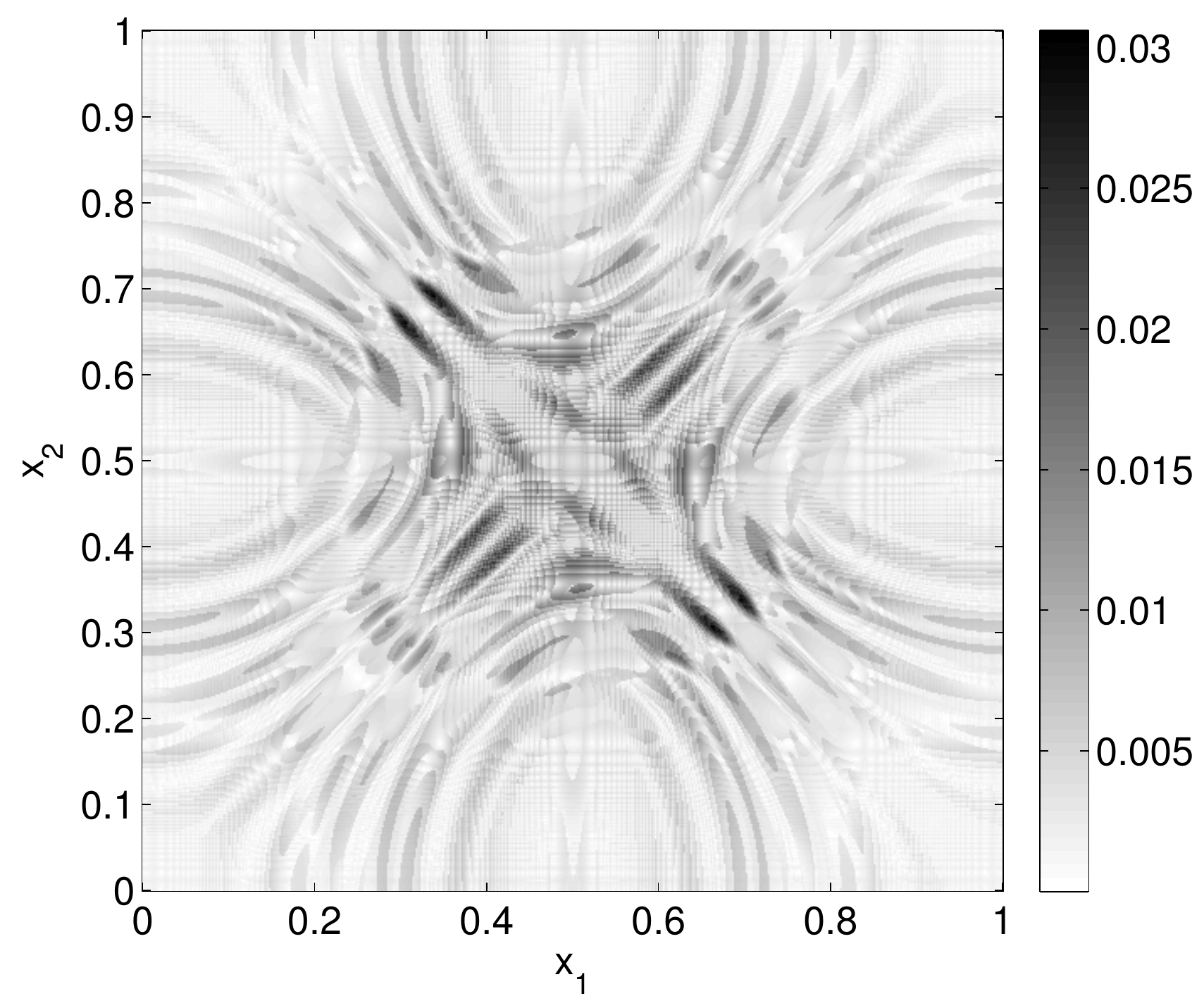} & \includegraphics[height=1.6in]{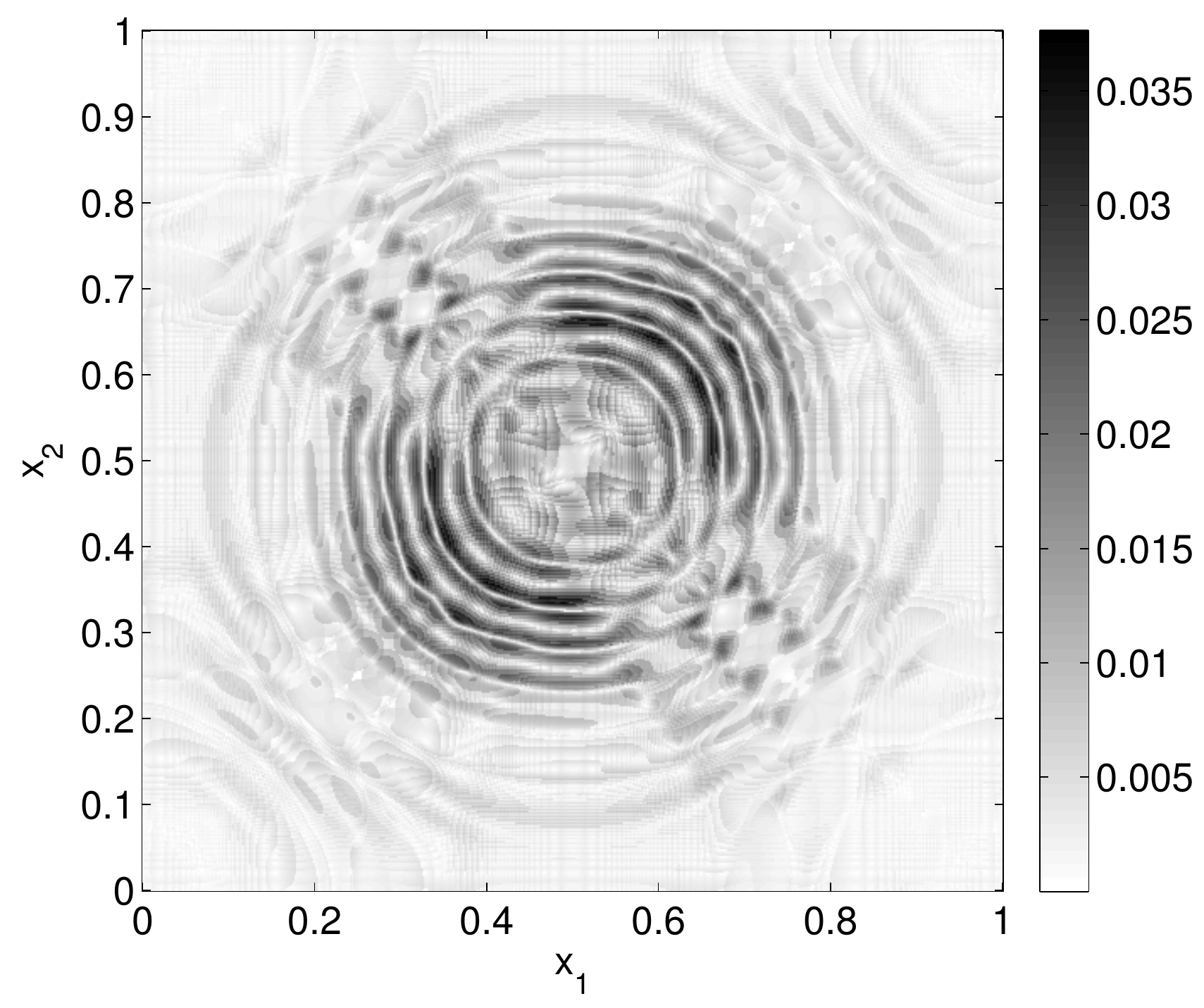}\\
    \end{tabular}
  \end{center}
  \caption{Left: A deformed plane wave propagating in the full space with zoomed-in data indicated by a rectangle. Middle: Relative error $R^0(b)$ of local
    wave-vector estimation using SSCT. Right: Relative error $R^0(b)$ of local wave-vector estimation given by SSWPT.}
  \label{fig:err_full}
\end{figure}

We compare the efficiency of SSCT and SSWPT in a noiseless case of a banded deformed plane wave $f(x)=e^{-(\phi(x)-c)^2/\sigma^2}\alpha(x)e^{2\pi i N\phi(x)}$ with the same parameters in last example and two more parameters $c=0.7$ and $\sigma=\frac{4}{135}$. As we discussed at the beginning of this subsection, $v_f(a,\theta,b)$ is only computed in the relevant region $|W_f(a,\theta,b)| \ge \sqrt{\eps}$. So, the relative error will be set to be zero elsewhere. The numerical result matches well with our theoretical prediction, showing that SSCT estimates local wave-vectors of this banded wave-like component within a relative error of order $O(\sqrt{\eps})$. However, SSWPT fails the truth as we discussed in the section of introduction.

\begin{figure}[ht!]
  \begin{center}
    \begin{tabular}{ccc}
      \includegraphics[height=1.6in]{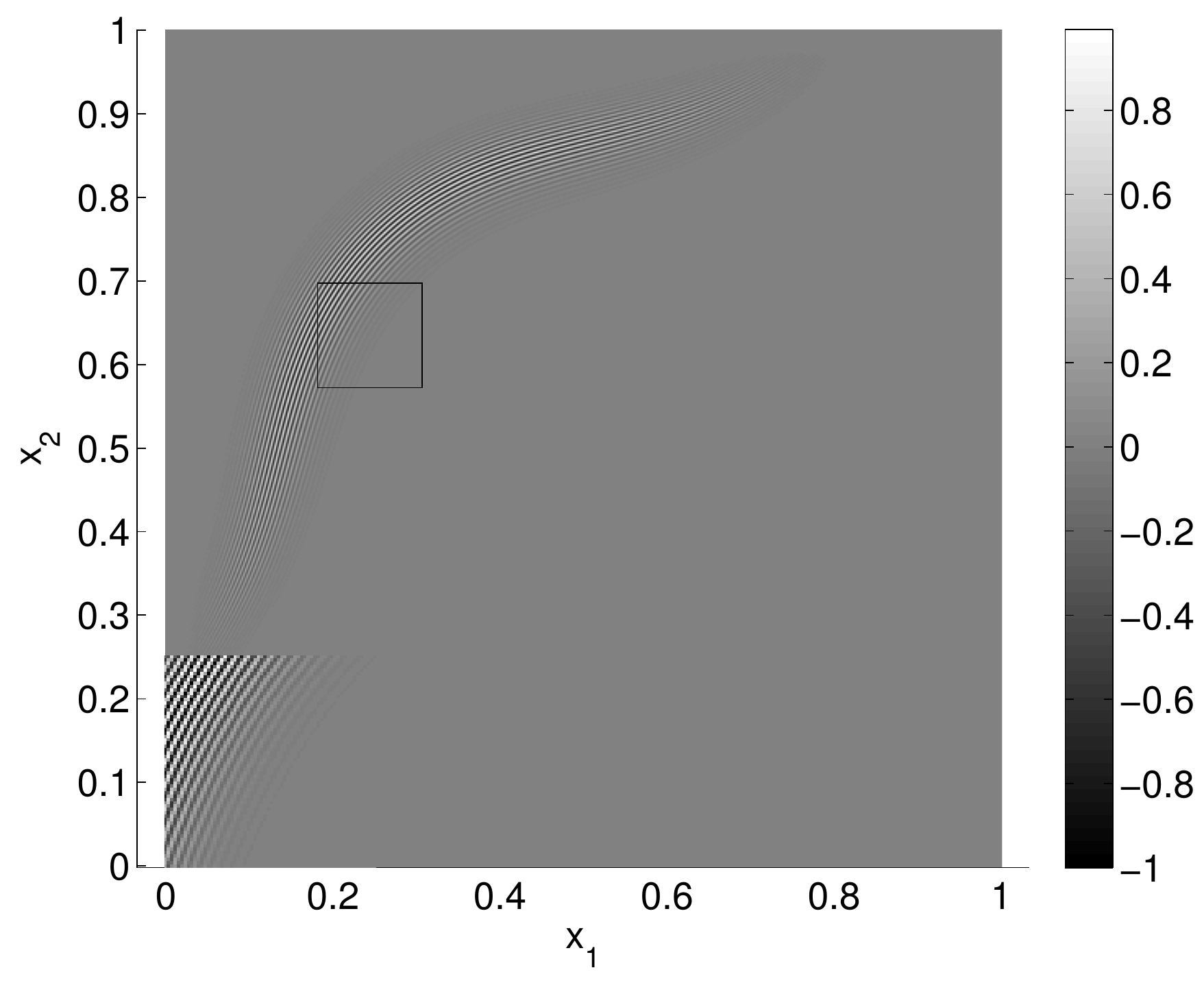} & \includegraphics[height=1.6in]{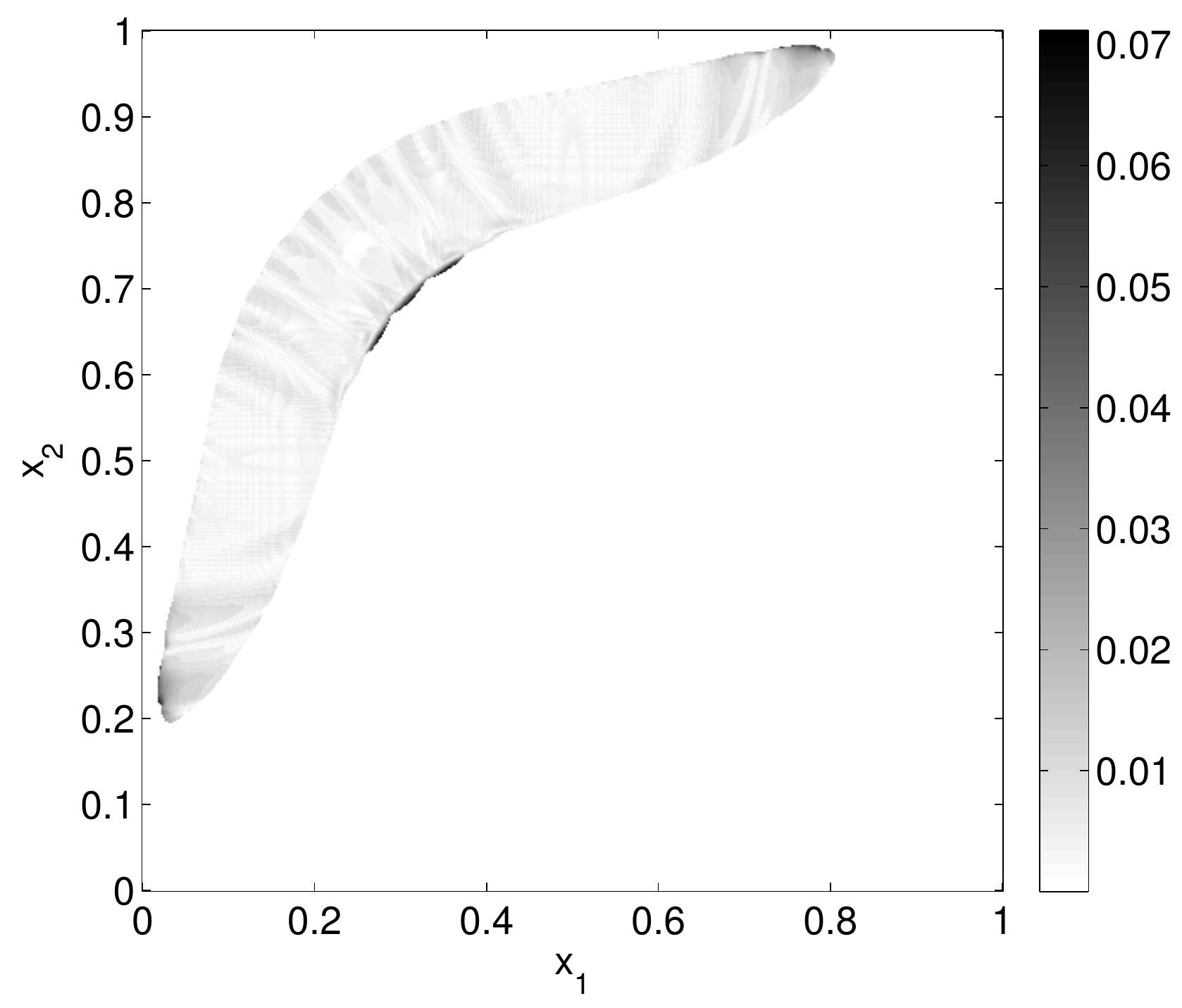} & \includegraphics[height=1.6in]{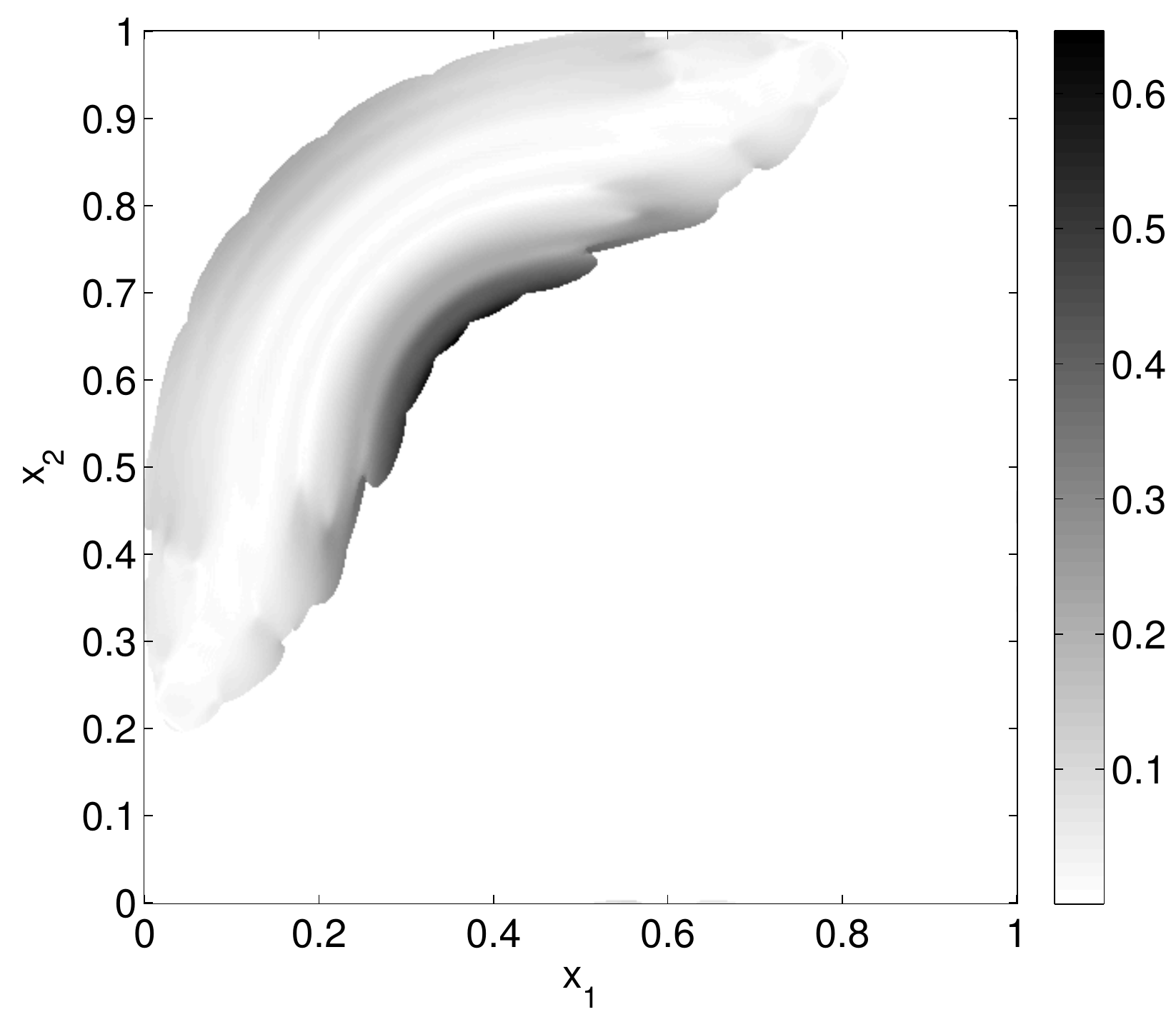}\\
    \end{tabular}
  \end{center}
  \caption{Left: A banded deformed plane wave. The zoomed-in data comes from the small rectangle. Middle: Relative error $R^0(b)$ of local
    wave-vector estimation using SSCT. Right: Relative error $R^0(b)$ of local wave-vector estimation given by SSWPT.}
  \label{fig:err_band}
\end{figure}

To quantitatively demonstrate the robustness against noise, we provide a series of tests of the above banded deformed plane wave with increasing noise levels. As usual, the noise level is described by the Signal-to-Noise Ratio ($\SNR$) defined by
\[
\SNR[dB]=10\log_{10}\bigg(\frac{\VAR f}{\sigma^2}\bigg).
\]
Suppose $n(x)$ is an isotropic complex Gaussian random noise with zero mean. We consider the noisy data  
\begin{equation}
f(x)=e^{-(\phi(x)-c)^2/\sigma^2}\alpha(x)e^{2\pi i N\phi(x)}+n(x),
\label{eq:f}
\end{equation}
with the same parameters in previous noiseless banded example. Table \ref{tab:snr} summarizes the results. The first row shows different noise levels and the second row records the threshold $\delta$ for $T_f(a,\theta,b)$. We observe that the threshold $\delta$ successfully reduces the influence of noise and keeps the local wave-vector estimate accurate and stable. 

\begin{table}
  \begin{center}
    \begin{tabular}{c||r|r|r|r|r}
      $\SNR$ & $\infty$ & 3 & 0 & -3 & -6 \\ \hline 
      $\delta$ & 0 & 3.5 & 4 & 4.5 & 5 \\ \hline 
      $|R_\delta(b)|_{\ell^\infty}$ & 0.03 & 0.03 & 0.03 & 0.045 & 0.06\\
    \end{tabular}
  \end{center}
  \caption{Maximum relative error of $R_\delta(b)$ with different $\SNR$.}
  \label{tab:snr}
\end{table}

\subsection{Intrinsic mode decomposition for synthetic data}
\paragraph{Example 2.} In many applications, it is desired to extract each component from a superposition. To show that our algorithm may provide a solution, we present some numerical examples of mode decomposition for highly oscillatory synthetic seismic data in noiseless and noisy cases (see Figure \ref{fig:EX2} top). Figure \ref{fig:EX2} shows the results of the application of our algorithm described in Section \ref{sub:full}. On the left is a noiseless example and the example on the right has some noise ($\SNR$ is $-3.07$ dB). Each mode of given data is accurately recovered in the noiseless case. In the noisy case, different modes with different propagation characters are completely separated. Each recovered mode practically reflects the curvature of corresponding mode in the original data, though there is some energy loss due to threshold $\delta$  to remove noise.

\begin{figure}[ht!]
  \begin{center}
    \begin{tabular}{cc}
      \includegraphics[height=2.4in]{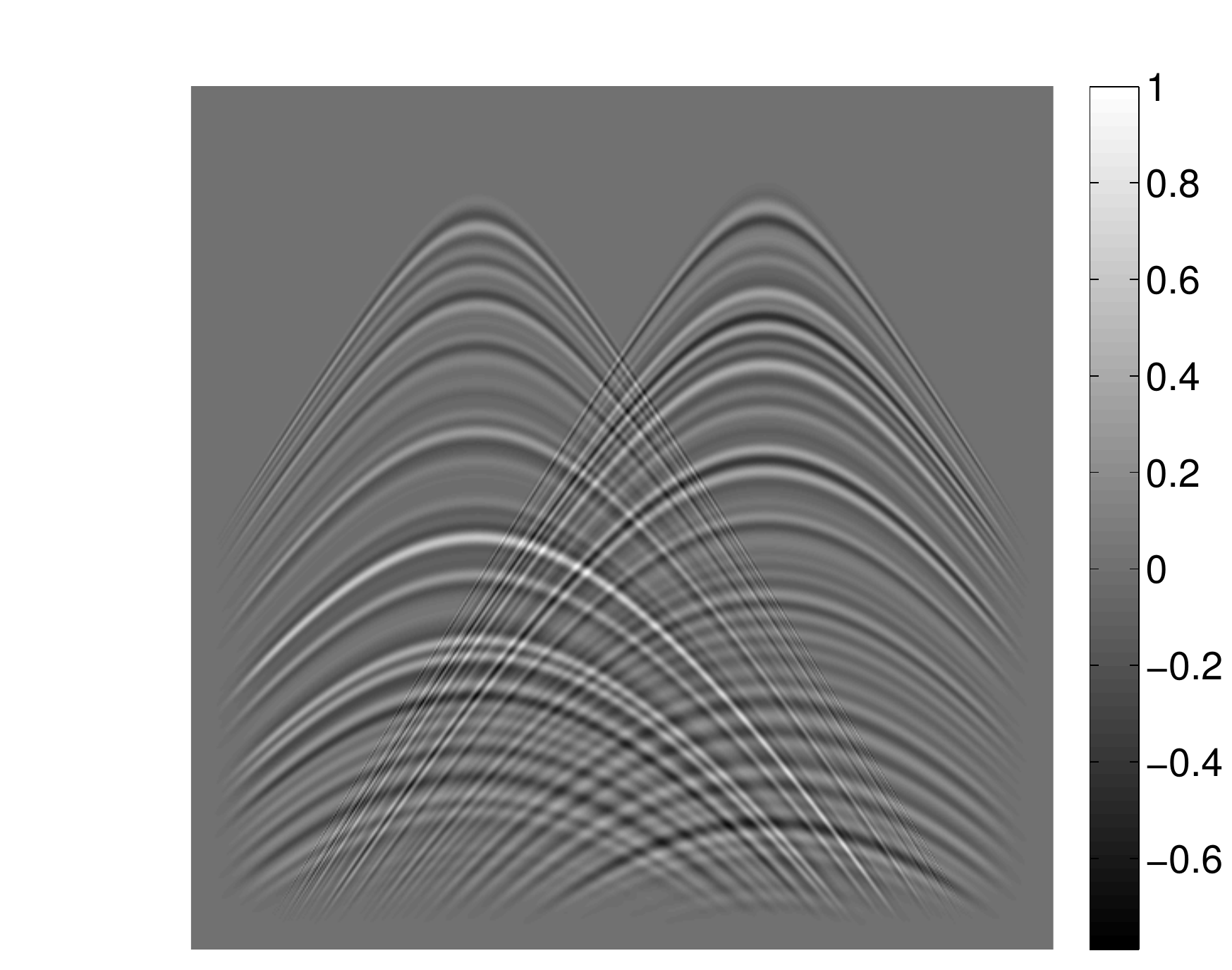}&        \includegraphics[height=2.4in]{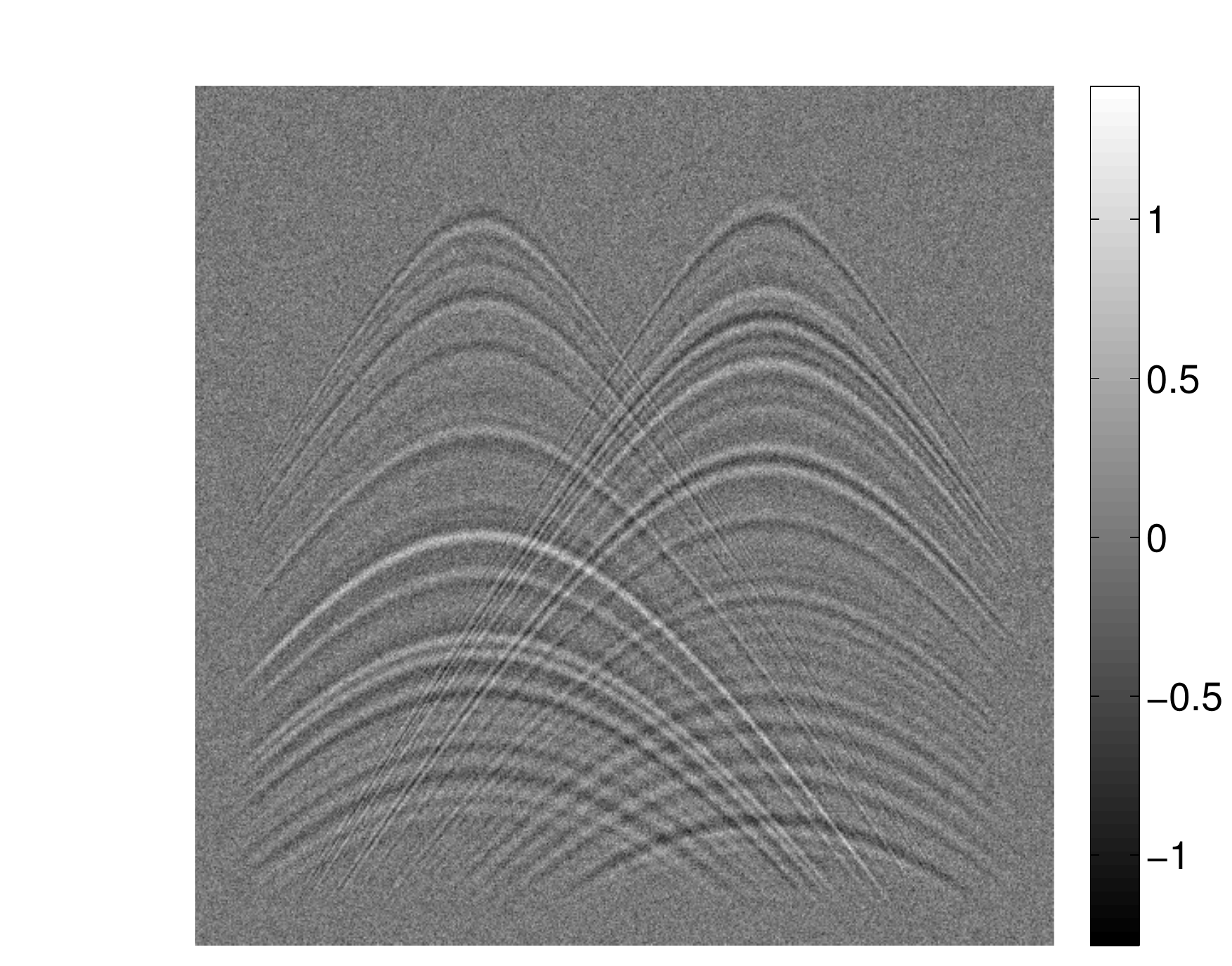}\\
      \includegraphics[height=2.4in]{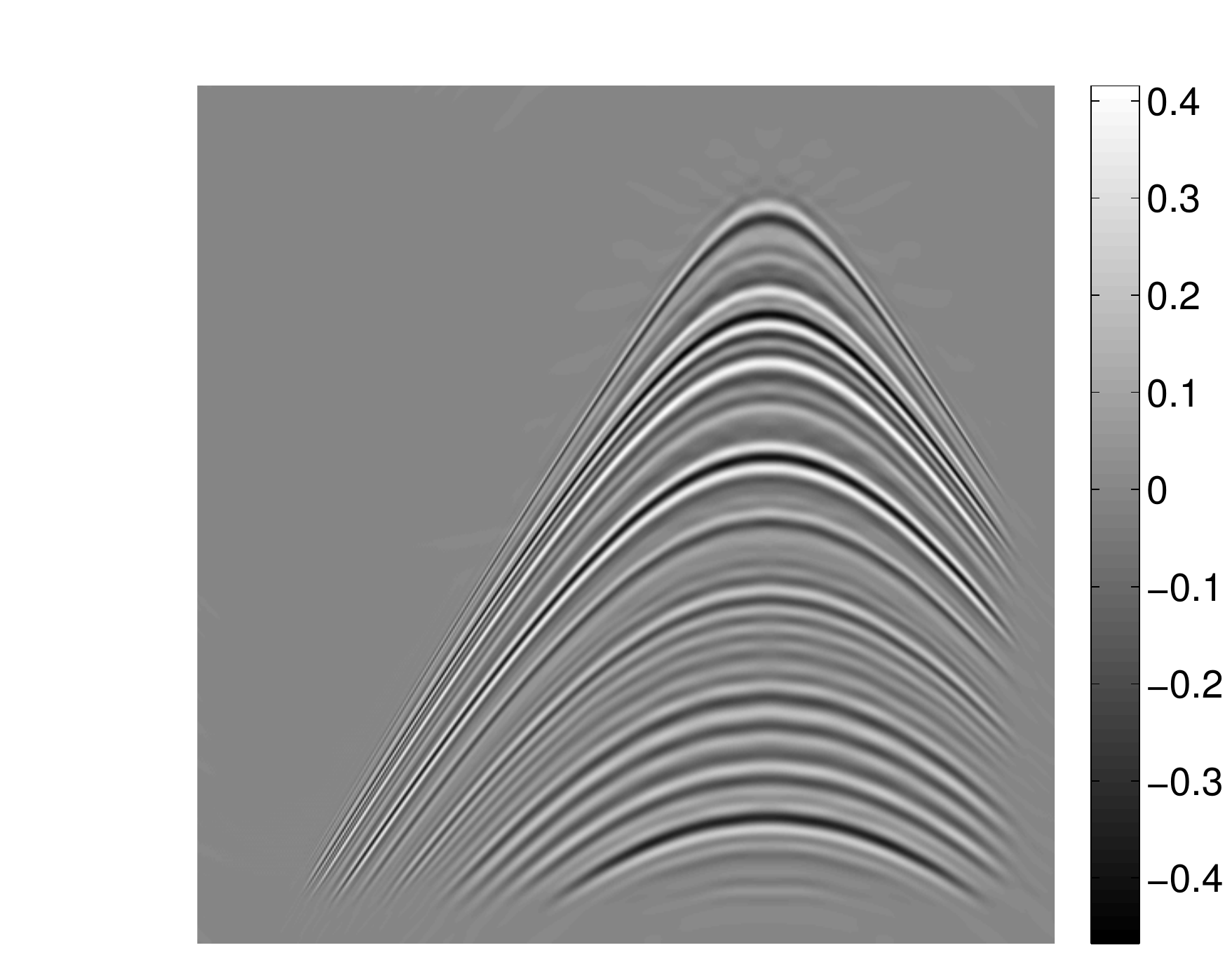}&        \includegraphics[height=2.4in]{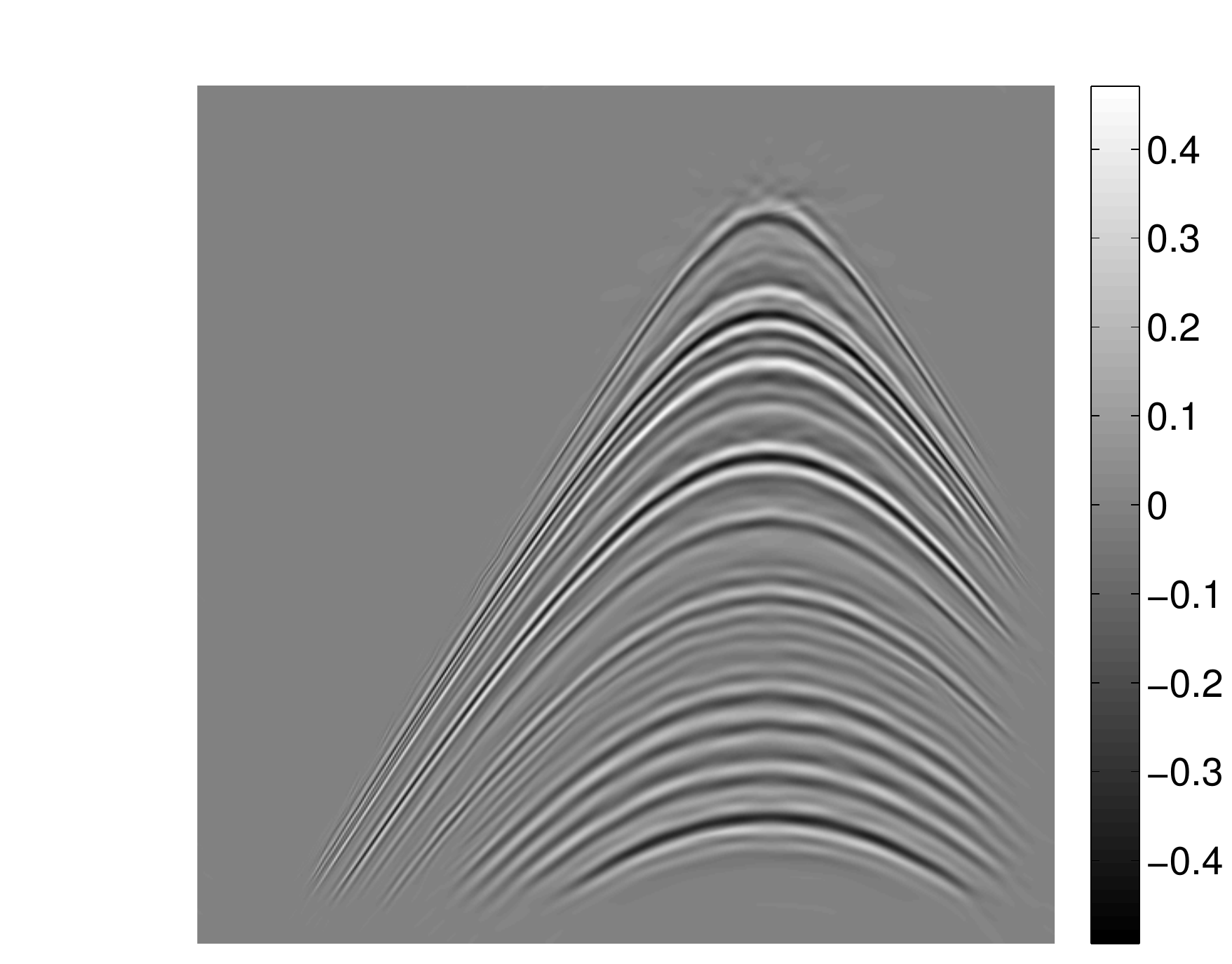}\\
      \includegraphics[height=2.4in]{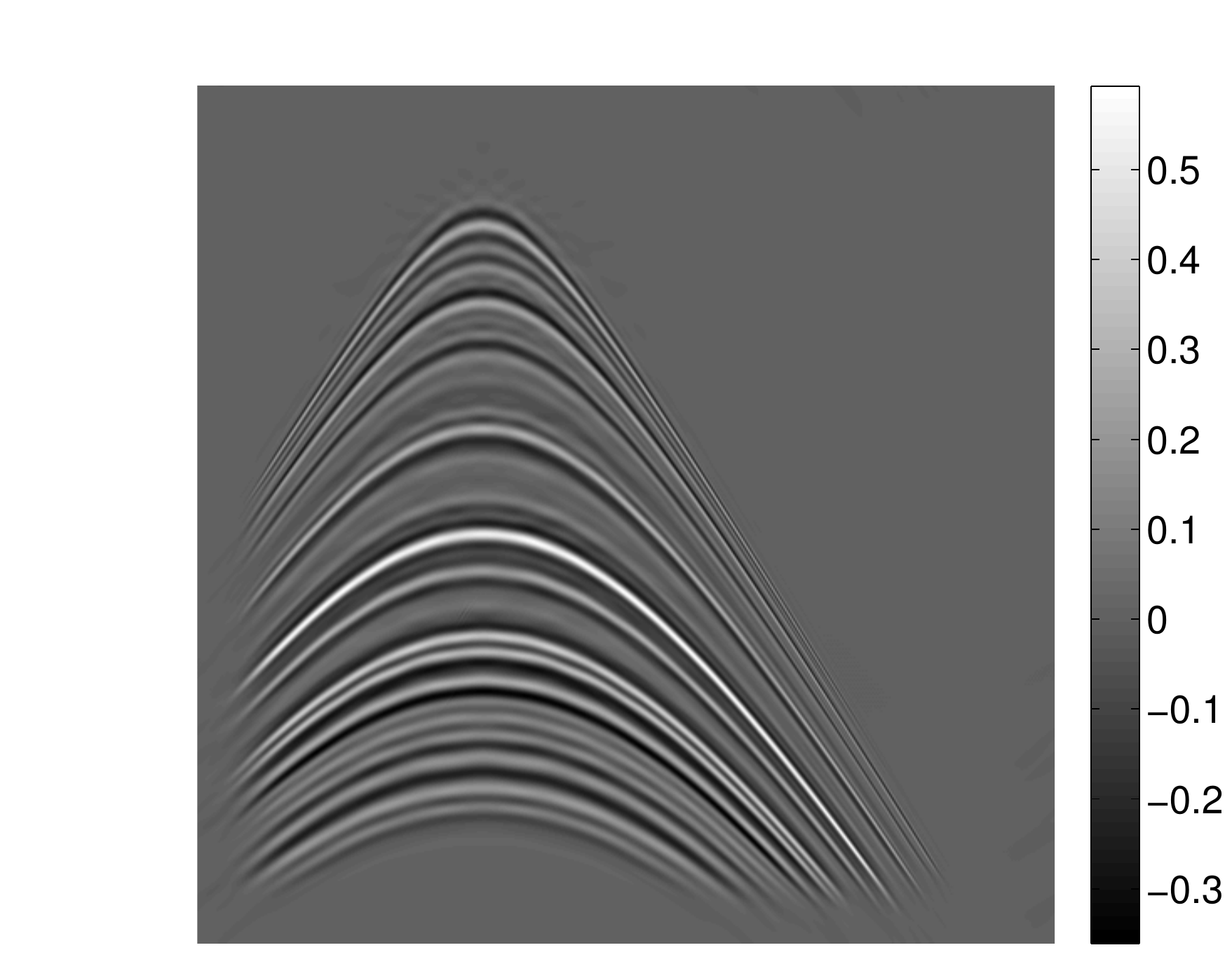}&        \includegraphics[height=2.4in]{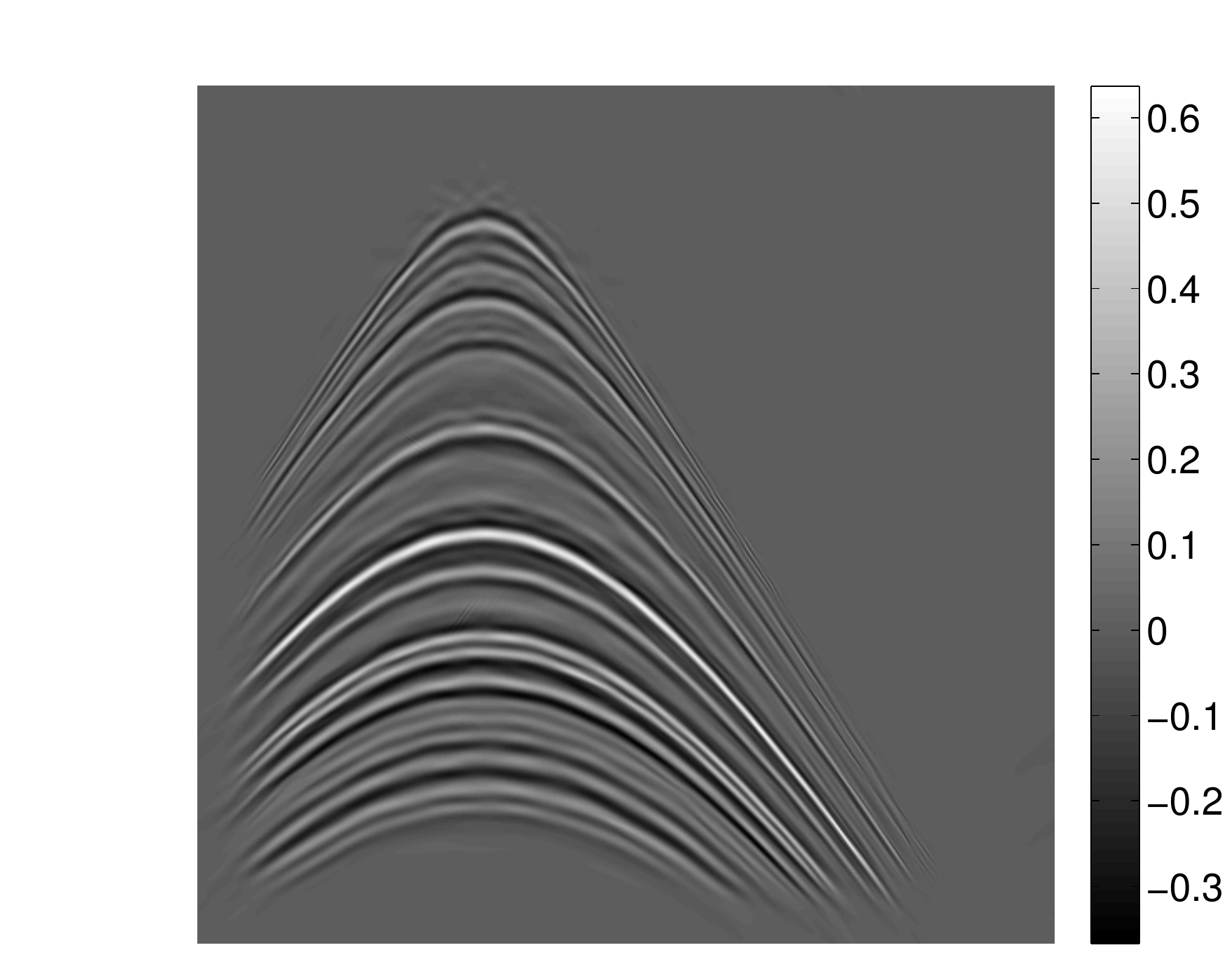}
    \end{tabular}
  \end{center}
  \caption{Example 2. Left: Mode decomposition without noise.  Right: Mode decomposition with noise ($\SNR=-3.07$). Top: A
    superposition of two components. Second row: The first recovered relevant mode. Third row: The second recovered relevant mode.}
  \label{fig:EX2}
\end{figure}

\paragraph{Example 3.} In some other applications, one component might be disrupted (e.g. randomly shifted in this example), and it is required to remove such component and recover others. Here we randomly shift the first mode in Example 2 in the vertical direction and apply our algorithm to recover the second mode. The numerical results summarized in Figure \ref{fig:EX3} show the capability of our algorithm to solve such a problem with or without noise. In this problem, the disrupted component can be considered as noise with high energy, i.e., this is a problem with very small $\SNR$. It is even more problematic that random shifting may create some texture similar to the mode to be recovered in some region. Fortunately, the synchrosqueezed representation is so concentrated that the resolution is still good enough to separate the mode from such similar texture by appropriately thresholding $T_f(a,\theta,b)$. 

The left example in Figure \ref{fig:EX3} shows the result of noiseless data. The recovered mode looks almost the same as the one recovered in noiseless Example 2 (Figure \ref{fig:EX2} bottom left), except some energy loss due to thresholding. It is of interest to add some background noise to see how well our algorithm is performing. Figure \ref{fig:EX3} right shows the result of noisy case. $\SNR$ is  $-0.90$, if we consider the energy of disrupted component as part of data energy. The result (see Figure \ref{fig:EX3} bottom right) is almost identical with the recovered mode in Figure \ref{fig:EX2} bottom left.

\begin{figure}[ht!]
  \begin{center}
    \begin{tabular}{cc}
      \includegraphics[height=2.4in]{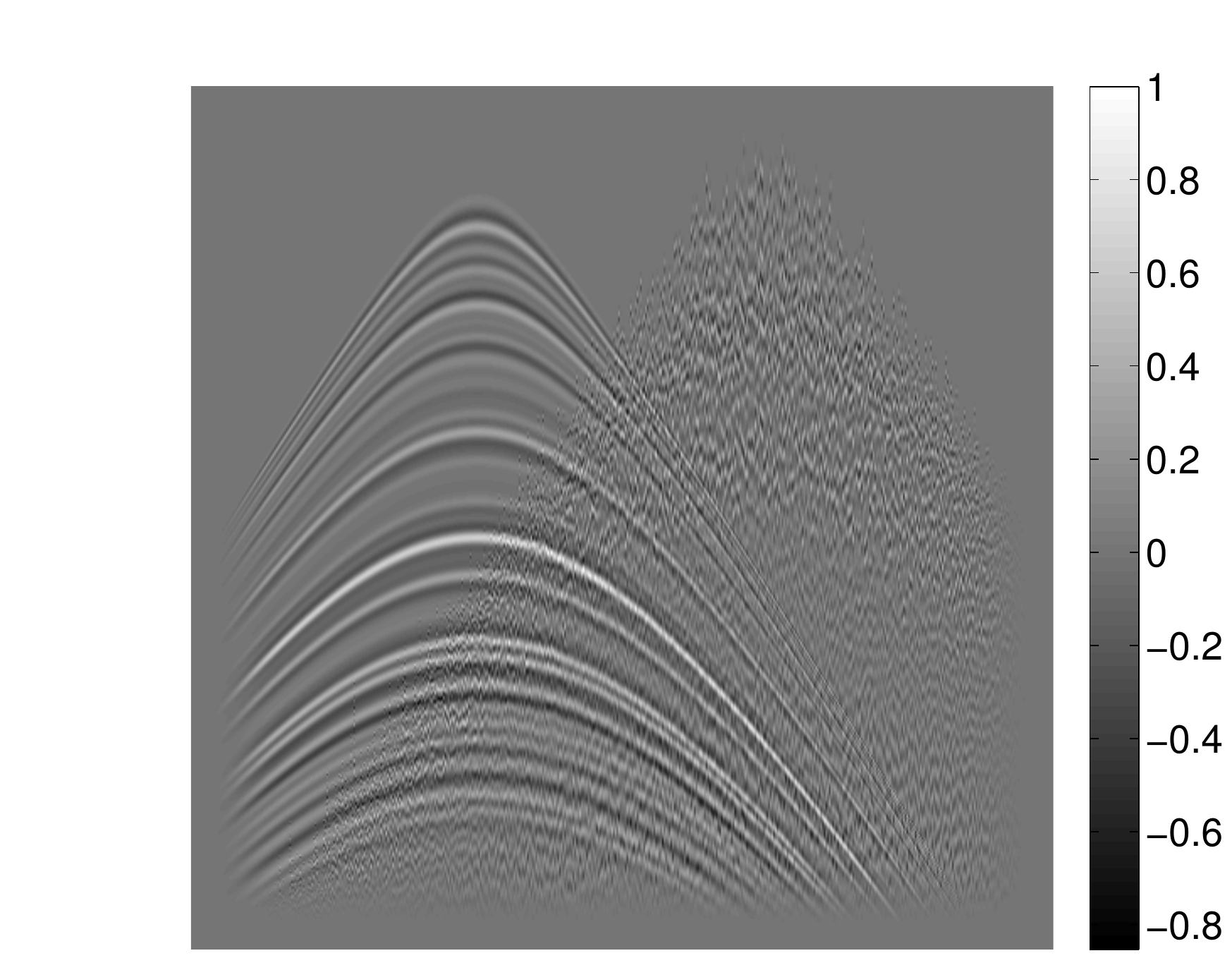}&        \includegraphics[height=2.4in]{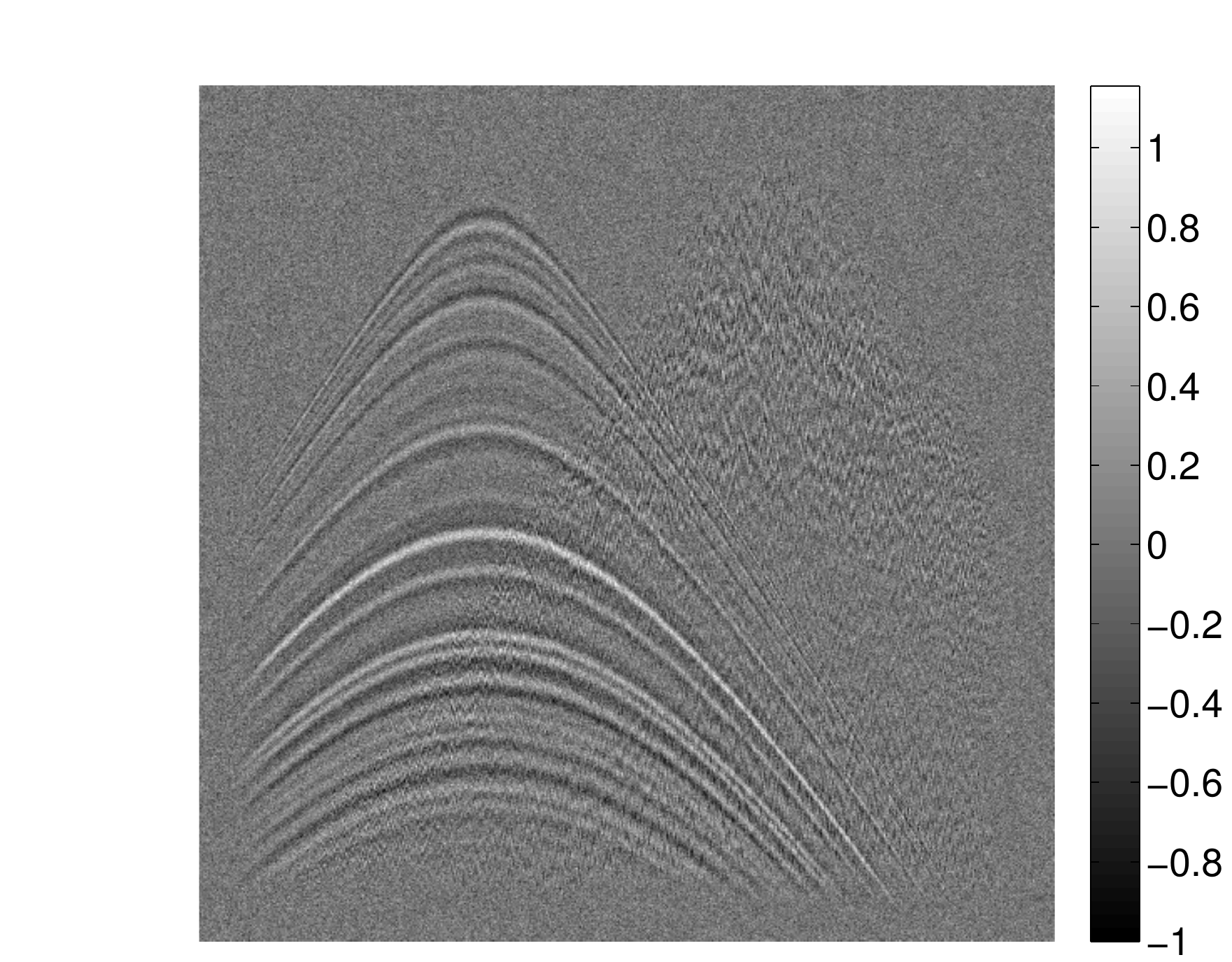}\\
      \includegraphics[height=2.4in]{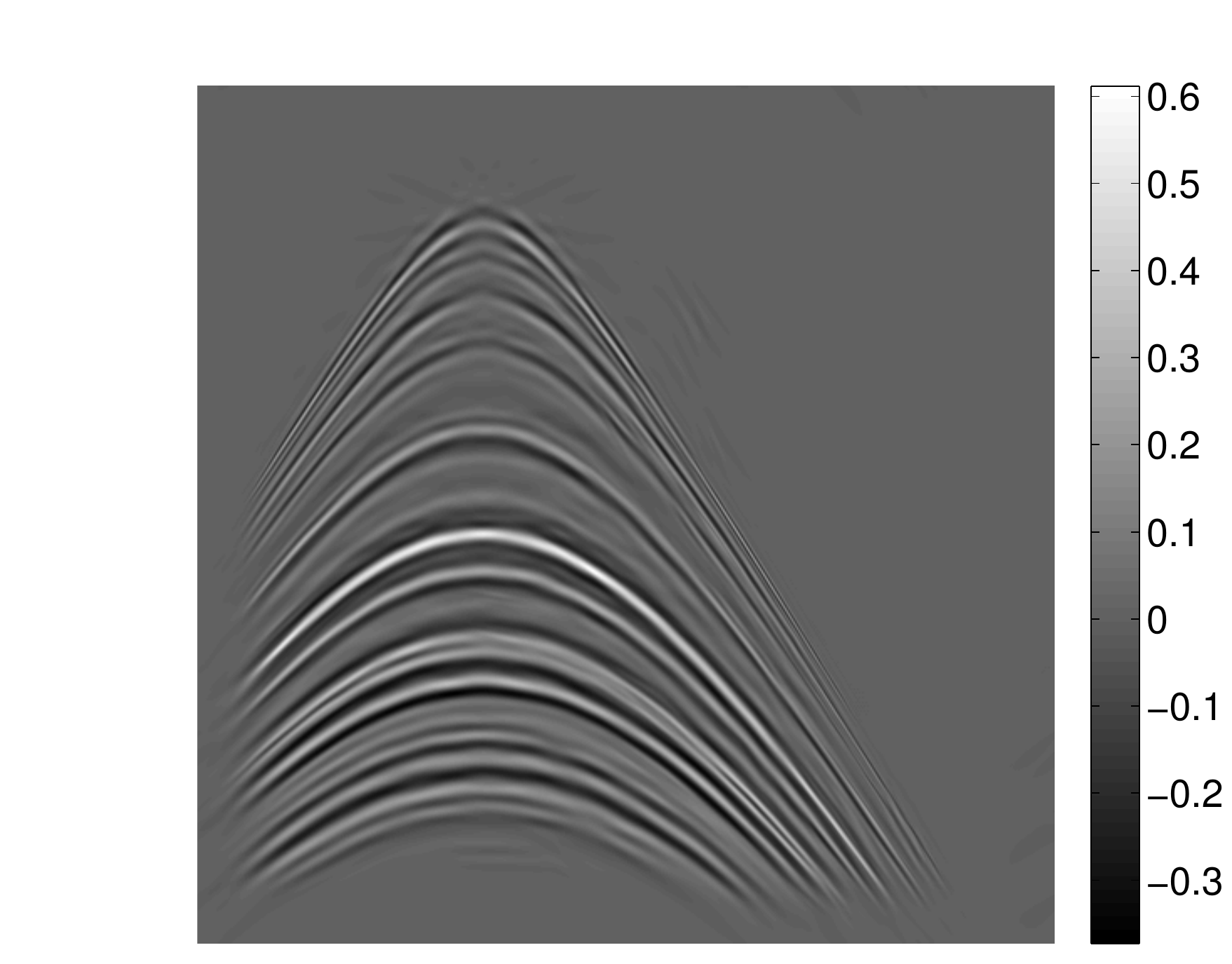}&        \includegraphics[height=2.4in]{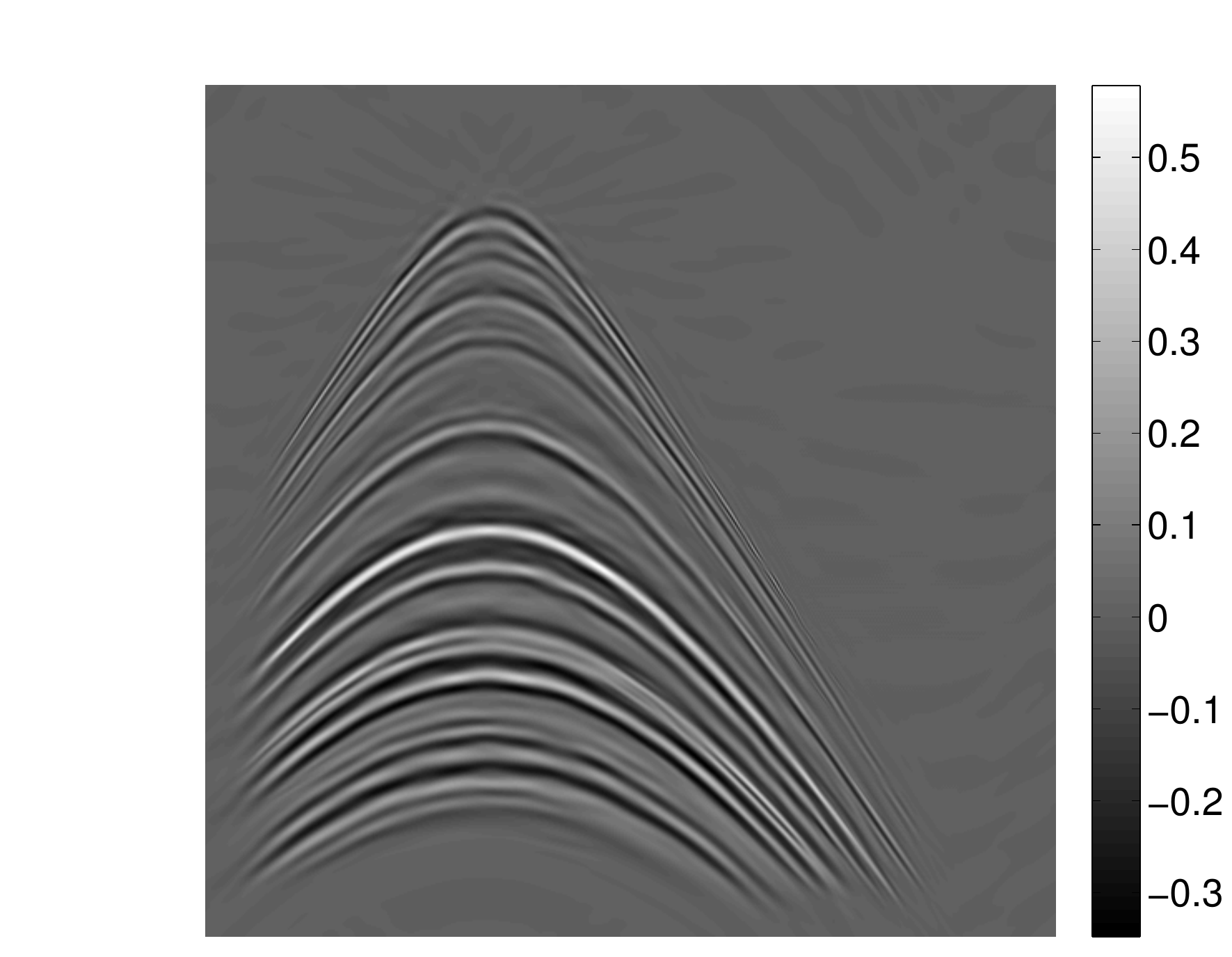}
    \end{tabular}
  \end{center}
  \caption{Example 3. Left: Mode identification without noise.  Right: Mode identification with noise ($\SNR=-0.90$). Top: A
    superposition of two components, one of which is disrupted by random shifting and need to be removed. Second row: The recovered relevant mode.}
  \label{fig:EX3}
\end{figure}

\subsection{Intrinsic mode decomposition for real data}
So far, the experiments shown are idealized, e.g., the boundary of each component is clear and smooth, and the amplitudes of each component are of the same level. In this subsection, we apply the synchrosqueezed curvelet transform to real seismic data and illustrate its good performance in complicated circumstance.

\paragraph{Example 4.} This is real seismic data with four main components and a band of energy loss near the bottom. The centered component is overlapping with others. Components in the bottom left and bottom right corners have irregular boundaries and not well aligned textures. The component on top has obviously weaker energy than others. These characters cause large difficulty in identifying all these components accurately. As shown in Figure \ref{fig:EX4}, the main textures and oscillatory patterns are recognized and recovered by our algorithm, though there is some loss of energy on the boundary of each component caused by thresholding.

\begin{figure}[ht!]
  \begin{center}
    \begin{tabular}{c}
      \includegraphics[height=3.5in]{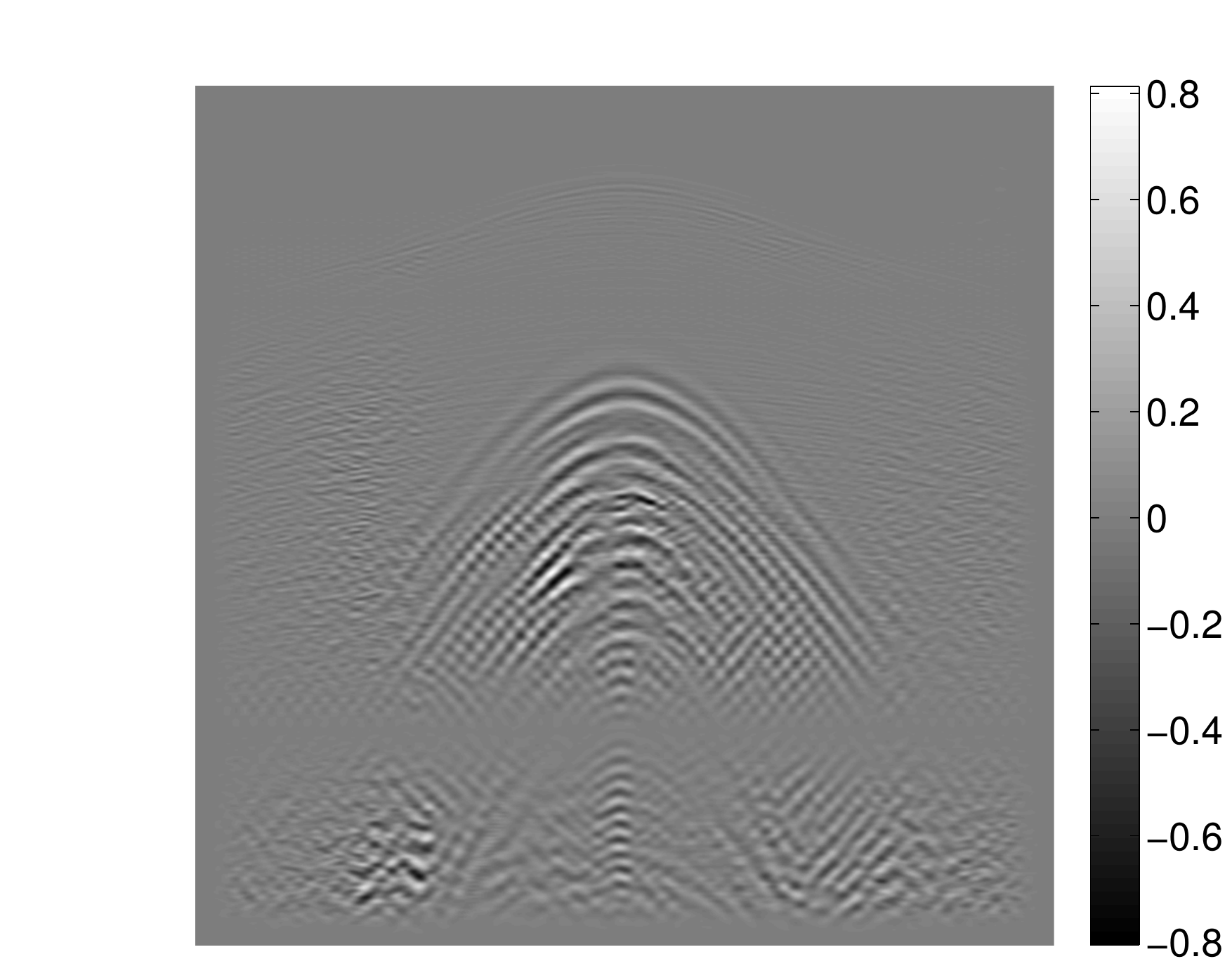}
    \end{tabular}
    \begin{tabular}{cc}
      \includegraphics[height=2.4in]{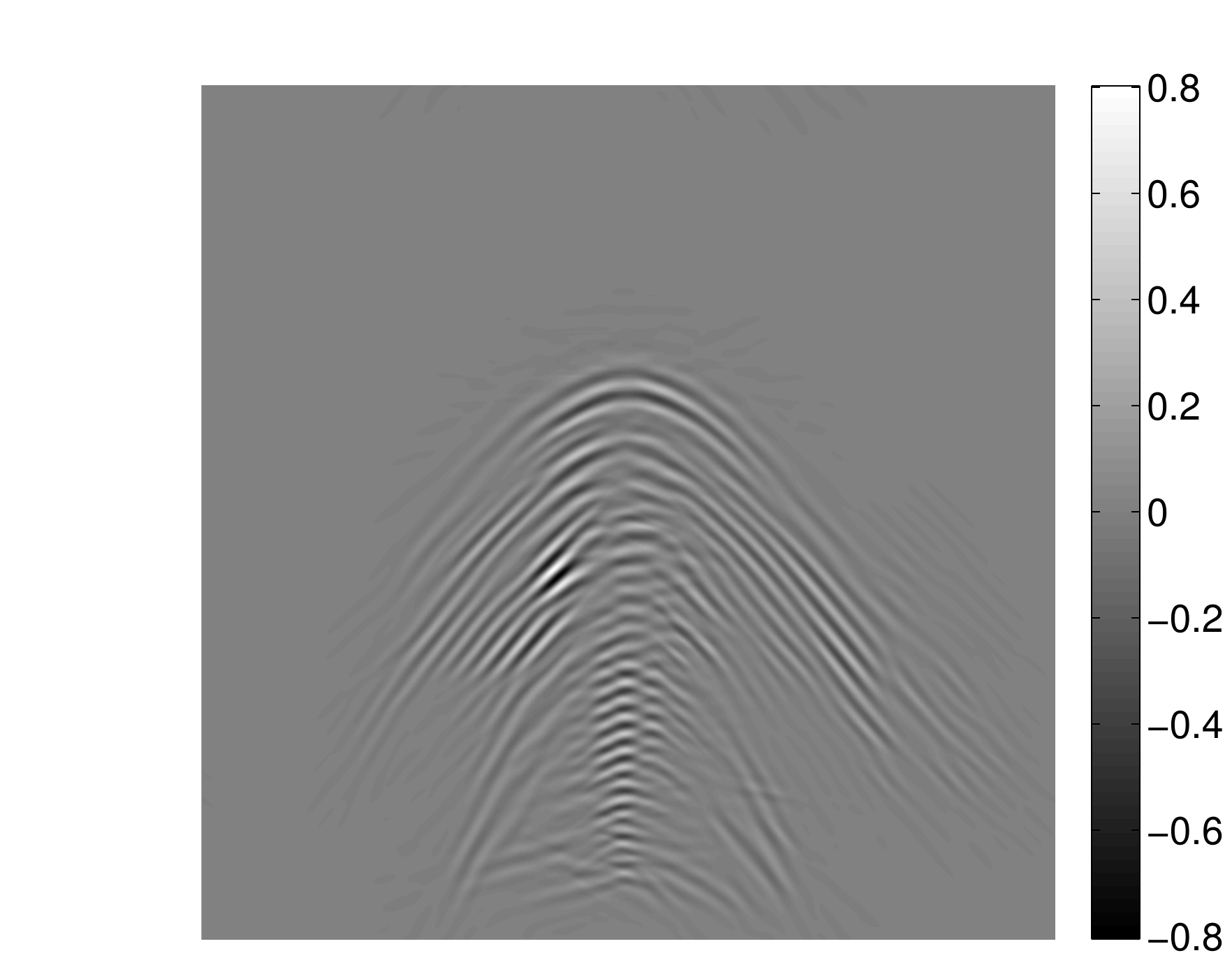}& \includegraphics[height=2.4in]{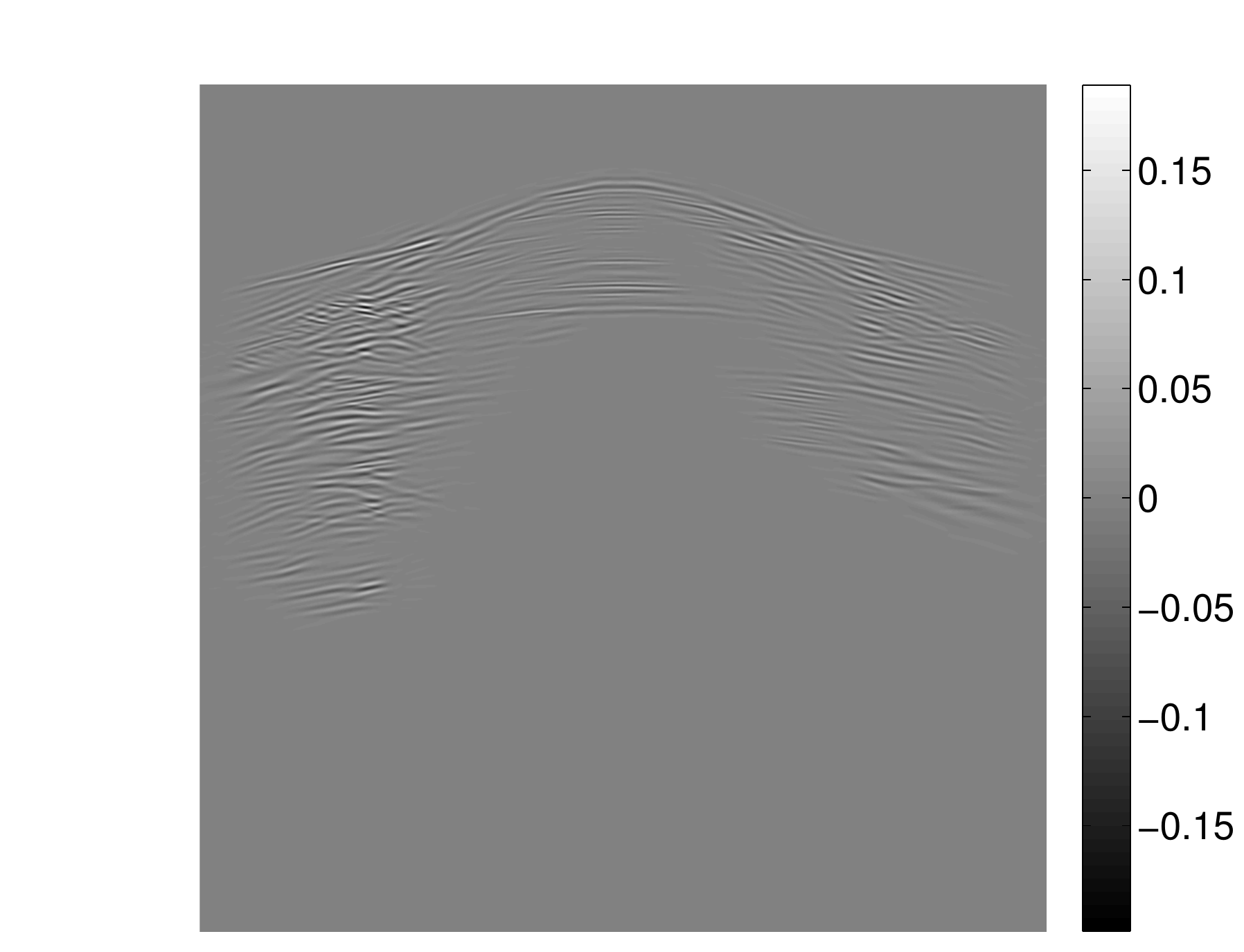}\\
      \includegraphics[height=2.4in]{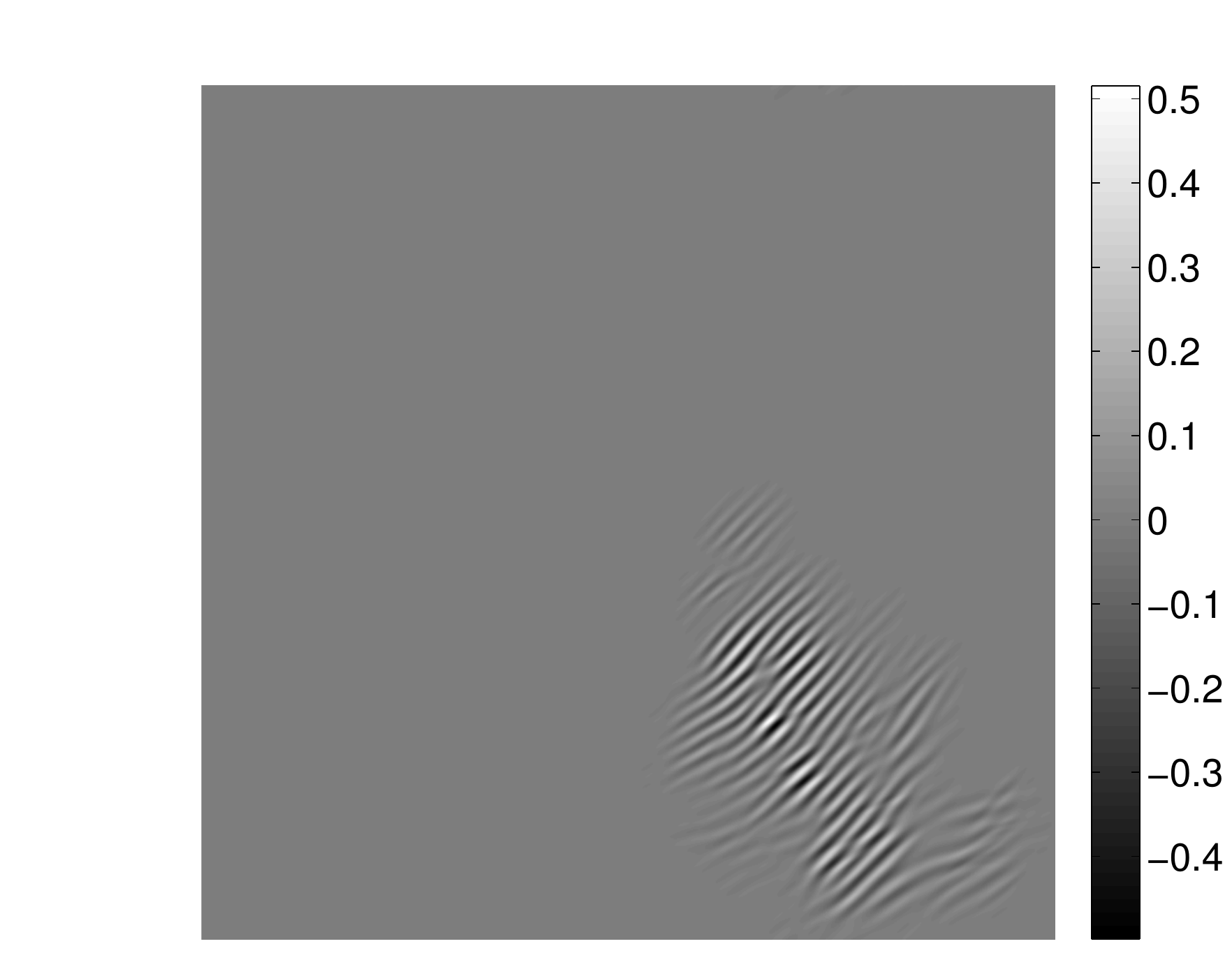}& \includegraphics[height=2.4in]{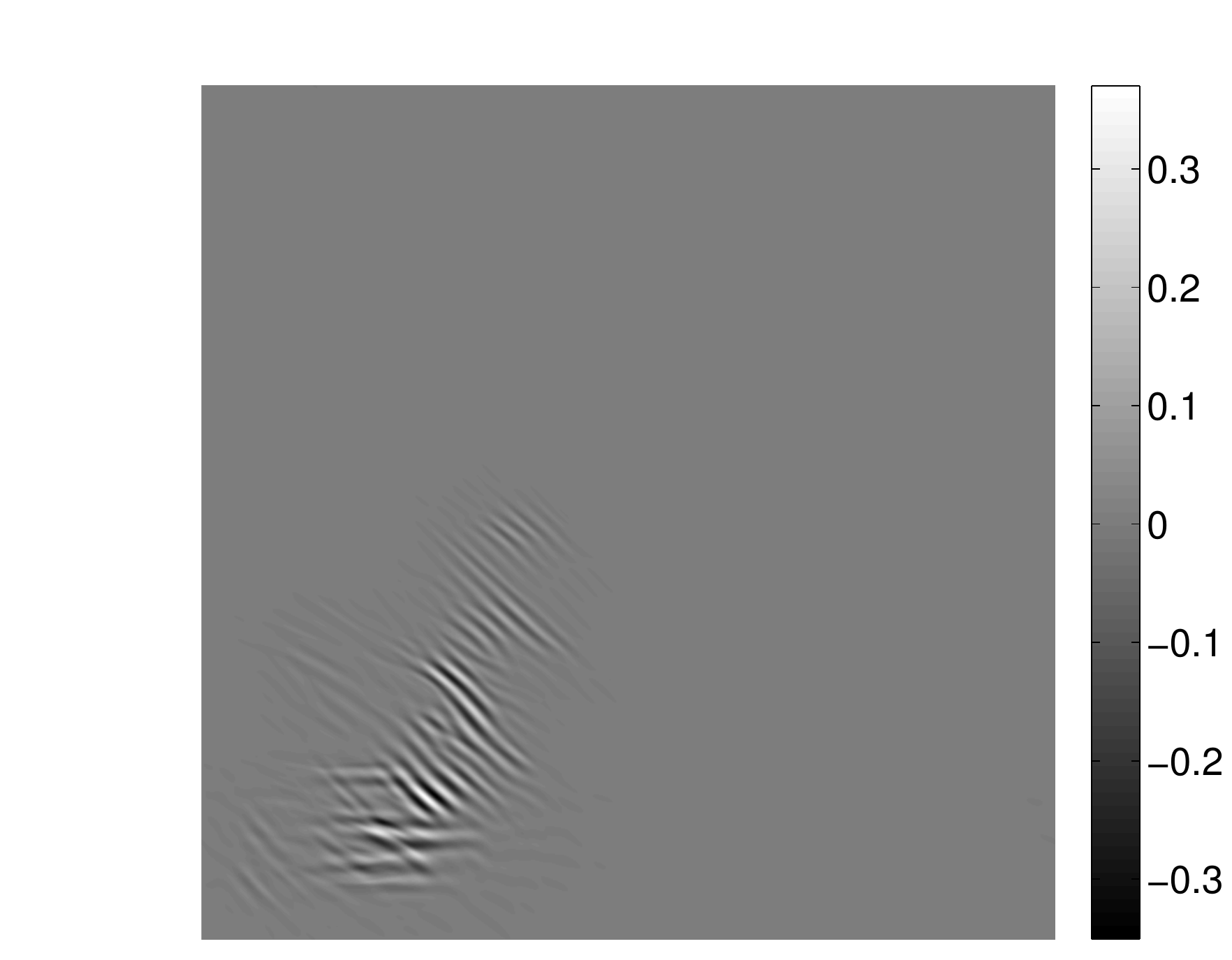}    
    \end{tabular}
  \end{center}
  \caption{Example 4. Top: Real seismic data. Middle and bottom:  Relevant recovered modes .}
  \label{fig:EX4}
\end{figure}

%----------------------------------------------------------
\section{Discussion}
\label{sec:diss}
This paper has proposed the synchrosqueezed curvelet transform as an optimal tool to analyze a superposition of high dimensional banded wave-like components. It serves as an example of applying a properly designed synchrosqueezing method to a superposition of components with specific structures for mode decompositions.

An appealing research direction is to study other type of data structures and other type of superpositions. In \cite{SSWPT} and this article, the data is assumed to be a superposition of wave-like components. In more general circumstances, the oscillatory pattern should not be restricted to wave functions.

Another promising direction would be the optimization scheme for 2D mode decomposition. Hard thresholding can cause some energy loss while reducing the noise. In other cases, some part of the data is missing or has extremely weak energy. One would desire a fast optimization scheme to estimate a clear structure of each component, even if there is missing data or severe noise.

Like the synchrosqueezed wave packet transform, the current approach can be easily extended to 3D or higher dimensions. This direction should be relevant for applications.

%----------------------------------------------------------
{\bf Acknowledgments.} H.Y. was partially supported by NSF grant
CDI-1027952. L.Y. was partially supported by NSF grants CAREER
DMS-0846501, DMS-1027952, and CDI-1027952.  H.Y.and L.Y. thank Jianfeng
Lu for discussion, Sergey Fomel and Jingwei Hu for providing seismic application.

\bibliographystyle{abbrv}
\bibliography{ref}

\end{document}